\documentclass[11pt]{amsart}
\usepackage{amsmath,amsfonts,amsthm,bbm, amssymb}
\usepackage[cmtip, all]{xy}
\usepackage[left=3cm, right=3cm, top = 2.5cm, bottom = 2.5cm ]{geometry}
\usepackage{comment}
\usepackage[color, notref, notcite, final]{showkeys}     
\definecolor{refkey}{gray}{.5}   % graylevel for refs
\definecolor{labelkey}{gray}{.5} % graylevel for labels
\definecolor{Red}{rgb}{1,0,0}

\usepackage{soul}

\newtheorem{thm}{Theorem}[section]
\newtheorem{prop}[thm]{Proposition}
\newtheorem{lem}[thm]{Lemma} 
\newtheorem{cor}[thm]{Corollary}
\newtheorem{conj}[thm]{Conjecture}

\newtheorem{ques}[thm]{Question}

\theoremstyle{definition}

\newtheorem{defn}[thm]{Definition}
\theoremstyle{remark}
\newtheorem{remk}[thm]{Remark}
\newtheorem{remks}[thm]{Remarks}

\newtheorem{exm}[thm]{Example}
\newtheorem{exms}[thm]{Examples}
\newtheorem{notat}[thm]{Notation}
\numberwithin{equation}{section}

{\hfill$\square$\end{defn}}
{\hfill$\square$\end{remk}}
{\hfill$\square$\end{remks}}
{\hfill$\square$\end{exm}}
{\hfill$\square$\end{exms}}
{\hfill$\square$\end{notat}}

\newcommand{\thmref}{Theorem~\ref}
\newcommand{\propref}{Proposition~\ref}
\newcommand{\corref}{Corollary~\ref}

\newcommand{\lemref}{Lemma~\ref}

\newcommand{\sD}{{\mathcal D}}

\newcommand{\sF}{{\mathcal F}}
\newcommand{\sG}{{\mathcal G}}

\newcommand{\sI}{{\mathcal I}}

\newcommand{\sK}{{\mathcal K}}
\newcommand{\sL}{{\mathcal L}}
\newcommand{\sM}{{\mathcal M}}

\newcommand{\sO}{{\mathcal O}}

\newcommand{\sR}{{\mathcal R}}

\newcommand{\sZ}{{\mathcal Z}}

\newcommand{\cD}{{\mathcal D}}

\newcommand{\cF}{{\mathcal F}}
\newcommand{\cG}{{\mathcal G}}

\newcommand{\cK}{{\mathcal K}}
\newcommand{\cL}{{\mathcal L}}
\newcommand{\cM}{{\mathcal M}}

\newcommand{\cO}{{\mathcal O}}

\newcommand{\cR}{{\mathcal R}}

\newcommand{\cZ}{{\mathcal Z}}

\newcommand{\C}{{\mathbb C}}

\newcommand{\F}{{\mathbb F}}
\newcommand{\G}{{\mathbb G}}
\renewcommand{\H}{{\mathbb H}}

\newcommand{\N}{{\mathbb N}}
\renewcommand{\P}{{\mathbb P}}
\newcommand{\Q}{{\mathbb Q}}
\newcommand{\R}{{\mathbb R}}

\newcommand{\Z}{{\mathbb Z}}

\newcommand{\fm}{{\mathfrak m}}
\newcommand{\fn}{{\mathfrak n}}

\newcommand{\CH}{{\rm CH}}

\newcommand{\surj}{\twoheadrightarrow}
\newcommand{\inj}{\hookrightarrow}

\newcommand{\codim}{{\rm codim}}

\newcommand{\Pic}{{\rm Pic}}

\newcommand{\divf}{{\rm div}}

\newcommand{\Ker}{{\rm Ker}}

\newcommand{\Spec}{{\rm Spec \,}}
\newcommand{\sing}{{\rm sing}}

\newcommand{\Sch}{{\operatorname{\mathbf{Sch}}}}

\newcommand{\Sm}{{\mathbf{Sm}}}

\newcommand{\cyc}{{\operatorname{\rm cyc}}}

\newcommand{\ds}{{/\kern-3pt/}}

\newcommand{\ov}{\overline}

\renewcommand{\dim}{\text{\rm dim}}

\newcommand{\tuborg}{\left\{\begin{array}{ll}}
\newcommand{\sluttuborg}{\end{array}\right.}

\newcommand{\wt}{\widetilde}

\newcommand{\tensor}{\otimes}

\def\cO{\mathcal{O}}
\def\cF{\mathcal{F}}
\def\ZDeSXr{\Z(r)^{\mathcal{D}^*}_{S_X}}
\def\ZDeDr{\Z(r)^{\mathcal{D}^*}_{D}}
\def\ZDeXr{\Z(r)^{\mathcal{D}^*}_{X}}
\def\ZDeXrun{\Z(r)^{\mathcal{D}}_{X}}

\def\ZDe#1#2{\Z(#2)^{\mathcal{D}^*}_{#1}}  %usage: $\ZDe{X}{2}$ produces weight 2 modified D-complex on X
\def\ZDeun#1#2{\Z(#2)^{\mathcal{D}}_{#1}}
\def\ol#1{\overline{#1}}
\def\dx{{\rm d}x}
\definecolor{winered}{rgb}{0.8,0,0}
\usepackage[pdfauthor={Federico Binda and Amalendu Krishna}, %
				pdftitle={Zero cycles with modulus and 
zero-cycles on singular varieties} ,colorlinks=true]{hyperref}
\hypersetup{
	bookmarksnumbered=true,
	linkcolor={winered},
	citecolor=blue,
	pagecolor=black, 
	urlcolor=black,	
}

\newcounter{elno}   
\newenvironment{romanlist}{
                         \begin{list}{\roman{elno})
                                     }{\usecounter{elno}}
                      }{
                         \end{list}}
                         
\newcounter{elno-abc}   
\newenvironment{listabc}{
                         \begin{list}{\alph{elno-abc})
                                     }{\usecounter{elno-abc}}
                      }{
                         \end{list}}
\newcounter{elno-abc-prime}   
\newenvironment{listabcprime}{
                         \begin{list}{\alph{elno-abc-prime}')
                                     }{\usecounter{elno-abc-prime}}
                      }{
                         \end{list}}
\usepackage{epigraph}

\begin{document}
    
\title{Zero cycles with modulus and zero cycles on singular varieties}
\author{Federico Binda and Amalendu Krishna}
\address{Fakult\"at f\"ur Mathematik, Universit\"at Regensburg, 
93040, Regensburg, Germany}
\email{federico.binda@mathematik.uni-regensburg.de}
\address{School of Mathematics, Tata Institute of Fundamental Research,  
1 Homi Bhabha Road, Colaba, Mumbai, India}
\email{amal@math.tifr.res.in}

\thanks{F.B.~is supported by the DFG SFB/CRC 1085 ``Higher Invariants'' and was partially supported by the DFG SPP 1786 ``Homotopy Theory and Algebraic Geometry'' during the preparation of this paper.}


%\maketitle

%\tableofcontents

\begin{abstract}
Given a smooth variety $X$ and an effective Cartier divisor $D \subset X$,
we show that the cohomological Chow group of 0-cycles on the double of $X$ 
along $D$ 
has a canonical decomposition in terms of the Chow group of 0-cycles 
$\CH_0(X)$ and the Chow group of 0-cycles with modulus $\CH_0(X|D)$ on $X$.
When $X$ is projective, we construct an Albanese variety with modulus and 
show that this is the universal regular quotient of $\CH_0(X|D)$.

As a consequence of the above decomposition, we prove the Roitman torsion
theorem for the 0-cycles with modulus. We show that $\CH_0(X|D)$ is
torsion-free and there is an injective cycle class map
$\CH_0(X|D) \inj K_0(X,D)$ if $X$ is affine. 
For a smooth affine surface $X$, this is strengthened to show that 
$K_0(X,D)$ is an extension of $\CH_1(X|D)$ by $\CH_0(X|D)$. 
%We derive from this Bloch's formula for 0-cycles with modulus on surfaces, improving a result of Kerz-Saito.
\end{abstract} 
\maketitle
%\tableofcontents  

{%\hypersetup{linkcolor=black} 
\tableofcontents}

\section{Introduction}\label{sec:Intro}

 When $X$ is a smooth quasi-projective scheme over a base field $k$, 
the motivic cohomology groups of $X$
% as defined, for example, in \cite{Voev-PST} or in \cite{Levine-Mixed}, 
admit an explicit description in terms of groups of 
algebraic cycles, called higher Chow groups, first defined by Bloch
\cite{Bloch-M}. These groups have all the properties that one expects, including 
Chern classes and a Chern character isomorphism from higher $K$-groups, as 
established in \cite{Levine-HigherChow} and \cite{Friedlander-Suslin}, 
generalizing the well-known relationship between the Chow ring of cycles modulo 
rational equivalence and the Grothendieck group of vector bundles. 

Leaving the safe harbor of smooth varieties leads to a different world, where the 
picture is substantially  less clear. One of the simplest examples of singular 
varieties is the nilpotent thickening $X_m = X\times_k k[t]/(t^m) $ of a smooth 
scheme $X$. For such a scheme, the beautiful correspondence between motivic 
cohomology, algebraic cycles and $K$-groups is destroyed, since one has
\[H_\mathcal{M}^*(X, \mathbb{Q}(*)) =H_\mathcal{M}^*(X_m, \mathbb{Q}(*)) \]  
according to the currently available definitions, preventing the existence of a 
Grothendieck-Riemann-Roch--type formula relating the motivic cohomology groups 
of $X_m$ with its higher $K$-groups

With the aim of understanding the algebraic $K$-theory of the ring $k[t]/(t^2)$ 
in terms of algebraic cycles, Bloch and Esnault first conceived the idea of 
algebraic cycles ``with modulus'' -- called additive Chow groups at the 
time -- defined by imposing suitable congruence condition at infinity on 
admissible cycles. This idea subsequently became the starting point of the 
discovery
of the theory of additive cycle complexes and additive higher Chow 
groups of schemes in the works of  R\"ulling \cite{Ruelling}, Park \cite{P1} 
and Krishna-Levine \cite{KLevine}.  \bigskip

The additive higher Chow groups are conjectured to give a 
cycle-theoretic
interpretation of the relative $K$-groups $K_*(X\times_k\mathbb{A}^1, X_m)$ for a 
smooth
scheme $X$. In recent works of Binda-Saito  \cite{BS} and Kerz-Saito \cite{KS},
the construction of the additive higher Chow groups was generalized to
 develop a theory of higher Chow groups with modulus.
These groups, denoted $\CH^*(X|D, *)$,  are designed to 
study the arithmetic and geometric properties of a smooth variety $X$ with fixed 
conditions along an effective  (possibly non-reduced) Cartier divisor $D$ on it, 
and are supposed to give a cycle-theoretic description of the mysterious relative 
$K$-groups $K_*(X,D)$, defined as the homotopy groups of the homotopy fibre of 
the restriction map $K(X)\to K(D)$. 
On the arithmetic side, when $X$ is a smooth variety  over  a finite field, Kerz 
and Saito studied the group $\CH_0(\overline{X}|D)$ for $\overline{X}$ an 
integral compactification of $X$ and $D$ a  non-reduced closed subscheme 
supported on $\overline{X} \setminus D$ (see \cite{KS} and 
\ref{def:DefChowMod-Definition} for the definition), and this has proven to be a 
fundamental ingredient in the study of wildly ramified class field theory.

 %In particular, they are supposed to give a cycle-theoretic
%description of the otherwise intractable relative $K$-groups $K_*(X, D)$. 
Although recently established results by various authors (see \cite{Kai}, 
\cite{RS}) have indicated that the Chow groups with modulus 
(and, more generally, the relative motivic cohomology groups of \cite{BS})  have 
some of the above expected properties, many questions remain widely open.

In order to  provide new evidence that the Chow groups with modulus are the 
right motivic
cohomology groups to compute the relative $K$-theory of a smooth scheme
with respect to an effective divisor, one would like to know if these 
groups share enough of the known structural properties of the Chow groups without
modulus, and to relate them to some geometric or cohomological invariants of the 
pair $(X,D)$. This is the subject of this paper.
Our interest is to establish these properties and present (an almost complete) 
picture
for the Chow groups of 0-cycles with modulus.

We now state our main results. The precise statement
of each of these results and the underlying hypothesis and notations will be
explained at appropriate places in this text. 

\subsection{Albanese variety and Roitman torsion theorem 
with modulus}\label{sec:MR}
One of the most important things known about the ordinary Chow group of 
0-cycles of a smooth projective variety is that it admits a universal
abelian variety quotient (the Albanese variety) which is useful for studying the 
question of the
representability of the Chow group. The celebrated theorem of Roitman
\cite{Roitman} (see also \cite{Milne} for the case of positive characteristic) 
says that this quotient map
is isomorphism on torsion. 
This theorem has had profound consequences in the study of 
the Chow group of 0-cycles.  One of the main goals of this paper is to establish 
these results (under some restrictions in positive characteristic)
for the Chow group of 0-cycles with modulus.

\begin{thm}[see Theorems \ref{thm:univ-Alb-overC} and 
\ref{thm:univ-Alb-overk}]\label{thm:Intro-1}
Let $X$ be a smooth projective scheme over an algebraically closed field $k$
and let $D \subset X$ be an effective Cartier divisor.  %If $k=\mathbb{C}$ or if 
%condition $(\bigstar)$ of \ref{ssec:Construction-Albanese-anyk}  is satisfied, 
Then there is a
smooth connected algebraic group ${\rm Alb}(X|D)$ and a 
group homomorphism $\rho_{X|D}\colon \CH_0(X|D)_{{\rm deg} \ 0} \to {\rm Alb}(X|D)$
which is a universal regular quotient of $\CH_0(X|D)_{{\rm deg} \ 0}$.
\end{thm}

\begin{thm}[see Theorem \ref{thm:Roitman-char-0} and 
\ref{thm:Roitman-Main+ve}]\label{thm:Intro-2}
Let $X$ be a smooth projective scheme over an algebraically closed field $k$
and let $D \subset X$ be an effective Cartier divisor.
% Assume again  that $k=\mathbb{C}$ or that condition $(\bigstar)$ of 
%\ref{ssec:Construction-Albanese-anyk}  is satisfied. 
Let $n \in \N$ be an integer prime to the characteristic of $k$. Then
$\rho_{X|D}$ induces an isomorphism 
$\rho_{X|D}\colon {_{n}\CH_0(X|D)_{{\rm deg} \ 0}} \xrightarrow{\simeq} 
{_{n}{\rm Alb}(X|D)}$ on the $n$-torsion subgroups.
\end{thm}

% The condition $(\bigstar)$ refers to a comparison between two competitive definitions of rational equivalence among zero cycles on a singular variety, and we will go back to this point in \S~3. We remark here that for projective varieties it is in particular satisfied whenever $k$ is algebraically closed and has characteristic $0$,  or $X$ has dimension at most $2$ without further assumptions on $k$.
Note that when $k=\mathbb{C}$ and $D_{\rm red} \subset X$ is a normal crossing
divisor, Theorem \ref{thm:Intro-1} has been proven, using a completely different 
approach, in \cite{BS}.

\subsection{Bloch's conjecture for 0-cycles with modulus}Let $X$ be a smooth projective surface over $\C$. Recall that the 
well-known Bloch conjecture predicts that the Abel-Jacobi map 
$\rho_X\colon \CH_0(X)_{{\rm deg} \ 0} \to {\rm Alb}(X)$ is an isomorphism
if $H^2(X, \sO_X) = 0$. Assuming this, we can show that the analogous 
statement for the Chow group with modulus also holds. In particular, the Bloch 
conjecture for 0-cycles with modulus is true
if $X$ is not of general type. Remarkably, instead of the  vanishing of the 
second cohomology group of the structure sheaf $\sO_X$, we have to assume the 
vanishing of the second cohomology group of the ideal sheaf $\sI_D$ of $D$.

\begin{thm}[see Theorem \ref{thm:Bloch-mod-thm}]\label{thm:Intro-2**}
Let $X$ be a smooth projective surface over $\C$ 
and let $D \subset X$ be an effective Cartier divisor.
Let $\sI_D$ denote the sheaf of ideals defining $D$.
Assume that the Bloch conjecture is true for $X$.
Then the map $\rho_{X|D}\colon \CH_0(X|D)_{{\rm deg} \ 0} \to {\rm Alb}(X|D)$ 
is an isomorphism if $H^2(X, \sI_D) = 0$.
\end{thm}

\subsection{Torsion theorem for 0-cycles with modulus on affine
schemes} 
Assume now that $X$ is a smooth affine variety over an 
algebraically closed field $k$. One of the consequences of Roitman's theorem is 
that the Chow group of 0-cycles 
on $X$ has no torsion, and this itself has had
many applications to projective modules on smooth affine
varieties.  Here comes the extension of this statement to the 0-cycles with 
modulus.

\begin{thm}[see Theorem \ref{thm:Intro-3-Pf}]\label{thm:Intro-3}
Let $X$ be a smooth affine scheme of dimension $d \ge 2$
over an algebraically closed field $k$
and let $D \subset X$ be an effective Cartier divisor.
Then $\CH_0(X|D)$ is torsion-free.
\end{thm}
 
%and an induction on dimension using Bertini to reduce to surface case
In the presence of a modulus, however, the  classical argument to deduce 
\thmref{thm:Intro-3} from Roitman's Theorem does not go through. For example, 
the localization sequence for the
ordinary Chow groups, which is one of the steps of the proof of the classical
case,
fails in the modulus setting, as explained in \cite[Theorem~1.5]{Krishna-3}.
Our approach is  to deduce  \thmref{thm:Intro-3} directly from  
\thmref{thm:Intro-6} below.

\subsection{Cycle class map to relative $K$-theory}
In the direction of understanding the relation between 0-cycles with modulus and
relative $K$-theory, we have the following results.
\begin{thm}[see Theorem \ref{thm:Intro-4-A}]\label{thm:Intro-4}
Let $X$ be a smooth quasi-projective scheme of dimension $d \ge 1$ over 
a perfect field $k$ and let $D \subset X$ be an effective Cartier 
divisor.
Then, there is a cycle class map 
\[
cyc_{X|D}\colon\CH_0(X|D) \to K_0(X,D).
\]

This map is injective if $k$ is algebraically closed and $X$ is affine.
\end{thm}

When $X$ has dimension 2, we can prove the following stronger
statement which completely describes $K_0(X,D)$ in terms of the 
Chow groups with modulus.

\begin{thm}[see Theorem \ref{thm:Intro-4-B-aff} and 
\ref{thm:divisor-support}]\label{thm:Intro-5}
Let $X$ be a smooth affine surface over an  algebraically closed field $k$ and let 
$D \subset X$ be an effective Cartier divisor.
Then, the canonical map $\CH_0(X|D) \to \CH_0(X|D_{\rm red})$ is an isomorphism
and there is an exact sequence
\[
0 \to \CH_0(X|D) \to K_0(X,D) \to \CH_1(X|D) \to 0.
\]
\end{thm}

Finally, for arbitrary quasi-projective surfaces, we prove the following 
structural result,  that we may see as an integral version of 
a Riemann-Roch--type formula for the relative $K_0$-group of the pair $(X,D)$.

\begin{thm}[see Theorem \ref{thm:Intro-4-B}]\label{thm:Intro-4**}
Let $X$ be a smooth quasi-projective surface over an
algebraically closed field $k$ and let $D \subset X$ be an effective Cartier 
divisor.
Then, there is a cycle class map $cyc_{X|D}\colon \CH_0(X|D) \to K_0(X,D)$ 
and a short exact sequence
\[
0 \to \CH_0(X|D) \to K_0(X,D) \to \Pic(X,D) \to 0.
\]
\end{thm}

\subsection{Bloch's formula}
As an application of \thmref{thm:Intro-4**}, we get the following
Bloch's formula for cycles with modulus on surfaces.

\begin{thm}$($Bloch's formula$)$\label{thm:Intro-4**-BF}  
Let $X$ be a smooth quasi-projective surface over an algebraically closed field
$k$. Let 
$D \subset X$ be an effective Cartier divisor. Then, there are isomorphisms
\[
\CH_0(X|D) \xrightarrow{\simeq} H^2_{\rm zar}(X, \sK^M_{2, (X,D)})
\xrightarrow{\simeq} H^2_{\rm nis}(X, \sK^M_{2, (X,D)}).
\]
\end{thm}

\subsection{The decomposition theorem}Essentially no case of the above results was previously known, 
and
in order to prove them, we develop a new approach to
study the Chow groups with modulus by drawing inspiration from the world of 
cycles on singular varieties. Given a smooth scheme $X$ with an
effective Cartier divisor $D$, we consider the notion of `doubling' $X$ along
$D$. This idea has previously been  used by Milnor \cite{Milnor} 
to study the patching of projective modules over commutative rings
(see \cite[Chap.~2]{Milnor}), and also by Levine \cite{Levine-HigherChow}
to study algebraic cycles in a different context.
Doubling $X$ along $D$ gives rise to a new scheme, that we denote by $S(X,D)$, 
which is, in general, highly singular. 

The novelty of our approach is the observation
that the Chow group of 0-cycles with modulus $\CH_0(X|D)$ can  
(under some conditions) be
suitably realized as a direct summand of the cohomological Chow group of 0-cycles
on $S(X,D)$ in the sense of Levine-Weibel \cite{LW}.  
This allows us to transport many of the known statements about the
Chow groups of 0-cycles on (possibly singular) schemes to
0-cycles with modulus. The following decomposition theorem can therefore
be called the central result of this paper 
(see \thmref{thm:Main-PB-PF-gen} for a precise statement).

\begin{thm}\label{thm:Intro-6} 
Let $X$ be a smooth quasi-projective scheme over a perfect 
field $k$. % satisfying one of the conditions (1)---(3) of Theorem \ref{thm:0-cycle-affine-proj-char0}.
Let $D \subset X$ be an effective Cartier divisor. Then, there
is a split short exact sequence
\[
0 \to \CH_0(X|D) \to \CH_0(S(X,D)) \to \CH_0(X) \to 0.
\]
% If the above Chow groups are all tensored with $\mathbb{Q}$, then the sequence is split exact without any further assumption on $X$ if $k$ is algebraically closed.
\end{thm}
In fact, it turns out that this approach can be taken forward to
study the Chow groups with modulus $\CH_*(X|D)$
in any dimension using the theory of Chow groups of singular schemes
developed by Levine \cite{Levine-5}. This generalization will be
studied in a different project. In this paper, we shall show how this
approach works for the relative Picard groups, apart from the above result
for 0-cycles.\medskip 

We conclude this Introduction by remarking that an Albanese variety with modulus 
has been previously constructed by Kato and Russell in \cite{KatoRussell} and 
\cite{Russell}. Their construction uses different techniques and starts from a 
definition of the Chow group of $0$-cycles with modulus that does not agree with 
the one proposed by Kerz and Saito: as a consequence of this discrepancy, our 
construction and the Kato-Russell construction are not directly related.

\subsection{Outline of the proofs}\label{sec:Outline}
This paper is organized as follows. Our principal task is to prove
the decomposition \thmref{thm:Intro-6}  for the Chow group of 0-cycles.
The proof of this takes up the next five sections of this paper.
We describe the double construction in \S\ref{sec:Double} and we prove several 
properties of it that are
used throughout the paper. 

The proof of \thmref{thm:Intro-6}
requires a non-trivial Bertini-type argument which allows us
to give a new description of the Cartier curves in the definition of the
Chow group of 0-cycles on the double. We do this first for surfaces
and we then explain how to reduce the general case to this one. This is done
in \S\ref{section:SSM}.
To relate the 0-cycles with modulus with the group of 0-cycles on the 
double, we define a variant of the Levine-Weibel Chow group of
0-cycles on the double and then show that the two definitions agree in 
as many cases as possible (see Theorem \ref{thm:0-cycle-affine-proj-char0}). This 
is done in \S\ref{section:0-C-S}.

We construct the Albanese variety with modulus attached to the pair 
$(X,D)$ in \S \ref{sec:Alb-C} which turns out to be a commutative algebraic group 
of general type. In characteristic zero, we give an explicit construction
of the Albanese variety with modulus 
using a relative version of Levine's modified Deligne-Beilinson cohomology.
We then use \thmref{thm:Intro-6} to prove the universality of this
Albanese and also the Roitman torsion theorem. We use
\thmref{thm:Intro-6} and the main results of \cite{Krishna-1} to 
deduce the Bloch conjecture for 0-cycles with modulus in \S\ref{sec:Alb-C}.
Other applications to affine schemes are obtained in \S\ref{sec:Rel-K},
\S~\ref{sec:BF-surfaces} and \S~\ref{sec:Affine**}.

\section*{Notations}\label{sec:Notations}
Let $k$ be a field. Since our arguments are geometric in nature, all schemes in 
this text are assumed to be quasi-projective over $k$ and we shall let $\Sch_k$ 
denote this category. Let $\Sm_k$ denote the full subcategory of 
$\Sch_k$ consisting of smooth schemes over $k$. We shall let 
${\Sch_k^{\rm ess}}$ denote the category of schemes which are essentially
of finite type over $k$. For a closed subscheme $Z \subset X$, we shall denote
the support of $Z$ by $|Z|$.
For a scheme $X$, the notation $X_{\rm sing}$ will mean the singular
locus of the associated reduced scheme $X_{\rm red}$.
The nature of the field $k$ will be specified in
each section of this paper.

\section{The double construction}\label{sec:Double}
The doubling of a scheme along a closed subscheme is the building block
of the proofs of our main results of this paper.
In this section, we define this double construction and study its many 
properties. These properties play crucial roles in the later parts
of this paper.

\subsection{The definition of the double} Recall that given surjective ring homomorphisms $f_i \colon A_i \to A$ for $i = 1,2$,
the subring $R = \{(a_1, a_2) \in A_1 \times A_2| f_1(a_1) = f_2(a_2)\}$
of $A_1 \times A_2$ has the property that the diagram  
\begin{equation}\label{eqn:double-set-up}
\xymatrix@C1pc{
R \ar[r]^>>>{p_1} \ar[d]_{p_2} & A_1 \ar[d]^{f_1} \\
A_2 \ar[r]_{f_2} & A}
\end{equation}
is a Cartesian square in the category of commutative unital rings,
where $p_i\colon R \to A_i$ is the composite map $R \inj A_1 \times A_2
\to A_i$ for $i = 1,2$.
Using the fact that every morphism $X \to Y$ in $\Sch_k$, with $Y$ affine,
factors through $X \to \Spec(\sO(X)) \to Y$, one can easily check that
the diagram 
\begin{equation}\label{eqn:double-set-up-1}
\xymatrix@C1pc{
\Spec(A) \ar[r]^{f_1} \ar[d]_{f_2} & \Spec(A_1) \ar[d]^{p_1} \\
\Spec(A_2) \ar[r]_>>>{p_2} & \Spec(R)}
\end{equation}
is a Cartesian and co-Cartesian square in $\Sch_k$.

Let us now assume that $X \in \Sch_k$ and let $\iota\colon D \inj X$ be a closed 
subscheme. If $f_1 = f_2 = \iota$, we see that the
construction of ~\eqref{eqn:double-set-up} is canonical and so it glues 
(see \cite[Ex.~II.2.12]{Hart}) to
give us the push-out scheme $S(X,D)$ and a commutative diagram
\begin{equation}\label{eqn:double-0}
\xymatrix@C2pc{
D \ar[r]^{\iota} \ar[d]_{\iota} & X \ar[d]^{\iota_{+}} \ar@/^.3cm/[ddr]^{id} & \\
X \ar[r]_<<<<>>{\iota_{-}} \ar@/_.3cm/[drr]_{id} & S(X,D) \ar[dr]^{\Delta} & \\
& & X.}
\end{equation}
One can in fact check, by restricting to affine parts of $X$ and then by using the gluing construction, that the top square in ~\eqref{eqn:double-0} is
co-Cartesian in $\Sch_k$. It is also a Cartesian square.
%One way to see this is the following.
%Let $X \amalg_D X$ denote the push-out of the diagram $X \leftarrow D 
%\rightarrow X$. If we restrict the resulting push-out diagram to each
%affine open subset $U \inj X$, then it follows from ~\eqref{eqn:double-set-up-1}
%that there is a unique map $S_{U|D} \to X \amalg_D X$. By varying the affine
%opens, we see that these maps glue to yield a map $S(X,D) \to X \amalg_D X$
%such that the diagram 
%\[
%\xymatrix@C1pc{
%D \ar[r]^{\iota} \ar[d]_{\iota} & X \ar[d]^{\iota_{+}} 
%\ar@/^.3cm/[ddr] & \\
%X \ar[r]_<<<<{\iota_{-}} \ar@/_.4cm/[drr] & S(X,D) \ar[dr] & \\
%& & X \amalg_D X.} 
%\]
%But this can happen if and only if the map $S(X,D) \to X \amalg_D X$ is an
%isomorphism. 
The scheme $S(X,D)$ constructed above will called the {\sl double} of $X$ 
along $D$. We shall mostly write $S(X,D)$ in short as $S_X$ if the 
closed subscheme $D \subset X$ is fixed and remains unchanged in a given
context. 

Notice that there is a canonical map $\pi\colon X \amalg X 
\xrightarrow{(\iota_{+}, \iota_{-})} S(X,D)$
which is an isomorphism over $S(X,D) \setminus D$.
Given a map $\nu\colon C\to S(X,D)$, we let $C_{+} = C \times_{S(X,D)} X_{+}$,
$C_{-} = C \times_{S(X,D)} X_{-}$ and $E = C \times_{S(X,D)} D$.
Here, $X_{\pm}$ is the component of $X \amalg X$ where $\pi$ restricts
to $\iota_{\pm}$.
We then have
\begin{equation}\label{eqn:double-0-*}
E = C \times_{S(X,D)} D = C_+ \times_{X} D = C_- \times_X D.
\end{equation}

More generally, we may often consider the following variant of the double construction.
\begin{defn}\label{defn:almost-int}
Let $j\colon D \inj X$ be a closed immersion of quasi-projective schemes over $k$
and let $f\colon T \to X$ be a morphism of quasi-projective schemes.
We shall say that $T$ is a join of $T_+$ and $T_-$ along $D$,
if there is a push-out diagram
\begin{equation}\label{eqn:almost-int-0}
\xymatrix@C1pc{
f^*(D) \ar[r]^{j_+} \ar[d]_{j_-} & T_+ \ar[d]^{\iota_+} \\
T_- \ar[r]_{\iota_-} & T}
\end{equation}
such that $T_{\pm}$ are quasi-projective schemes and $j_{\pm}$ are
closed immersions. 
\end{defn}

%Note that in the above case, it follows at once that
%$T$ is reduced with two components $T_{\pm}$ if $T_{\pm}$ are integral.

The following lemma related to the double construction will be
often used in this text.

\begin{lem}\label{lem:Cycle-mod-1}
Let $\nu\colon C \to S(X,D)$ be an affine morphism. Then
the push-out $C_{+} \amalg_{E} C_{-}$ is a closed subscheme of $C$.
This closed immersion is an isomorphism if $C$ is reduced.
\end{lem}
\begin{proof}
There is clearly a morphism $C_{+} \amalg_{E} C_{-} \to C$. Showing that this
map has the desired properties is a local question on $X$. So it suffices to
verify these properties at the level of rings.

If we set $X = \Spec(A)$, $S(X,D) = \Spec(R)$ and let $I$ be the defining 
ideal for $D$, then we have an exact sequence of $R$-modules 
\begin{equation}\label{eqn:double-2}
0 \to R \xrightarrow{\phi} A \times A \to A/I \to 0.
\end{equation} 

Since $\nu$ is affine, we can write $C = \Spec(B)$. Let
$J \subset B$ be the ideal defining the closed subscheme $E$.
Tensoring ~\eqref{eqn:double-2} with $B$, we get an exact sequence
\begin{equation}\label{eqn:double-2-0}
B \xrightarrow{\phi_B} B_{+} \times B_{-} \to B/J \to 0 
\end{equation}
and this shows that $C_{+} \amalg_{E} C_{-} \inj C$ is a closed immersion
of schemes. It is also clear that this inclusion is an isomorphism
in the complement of $E$. Furthermore, the surjectivity of the map
$C_{+} \amalg C_{-} \to C$ shows that the above inclusion is also surjective
on points. We conclude from this that the closed immersion 
$C_{+} \amalg_E C_{-} \inj C$ induces identity on the underlying reduced
schemes. In particular, it is an isomorphism if $C$ is reduced.
Equivalently, $\phi_B$ is injective.
%Note that the push-out $C_{+} \amalg_E C_{-}$ always exists as a sheaf on $\Sch_k$. The question is about its representability.
%Note that $C_{\pm}$ may not be reduced even if $C$ is.
\end{proof}

\subsection{More properties of the double}\label{sec:More-on-double}
We now prove some general properties of the double construction that will be
used repeatedly in this text. We shall also show that the double shares many
of the nice properties of the given scheme if the underlying closed subscheme
is an effective Cartier divisor. This will be our case of interest
in the sequel.

For $X \in \Sch_k$, let $k(X)$ denote the sheaf of rings of total
quotients of $X$. For a reduced scheme $X$, let $k_{\min}(X)$ denote the product
of the fields of fractions of the irreducible components of $X$. Note that
there are maps of sheaves of rings $\sO_X \inj k(X) \to k_{\min}(X)$ and
the latter map is an isomorphism if $X$ is reduced and has no embedded primes. 

\begin{prop}\label{prop:double-prp}
Let $X$ be a scheme in $\Sch_k^{\rm ess}$ and 
let $\iota\colon D \inj X$ be a closed subscheme not containing any irreducible component of $X$. 
Then the following hold.
\begin{enumerate}
\item
There are finite maps 
\[
X \amalg X \xrightarrow{\pi} S(X,D) \xrightarrow{\Delta} X
\]
such that $(X \setminus D) \amalg (X \setminus D) \xrightarrow{\pi} 
S(X,D) \setminus D = \Delta^{-1}(X \setminus D)$ is an isomorphism.
In particular, $S$ is affine (projective) if and only if $X$ is so.
\item
$S(X,D)$ is reduced if $X$ is so. In this case, one has 
\[
k_{\rm min}(S(X,D) \setminus D) \simeq k_{\rm min}(X\setminus D) \times k_{\rm min}
(X \setminus D).
\]
If $D$ contains no component of $X$, then $k_{\rm min}(S(X,D) \setminus D)
= k_{\rm min}(S(X,D))$.
\item
The composite map
$D \stackrel{\iota_{\pm} \circ \iota}{\rightarrow} \Delta^*(D) 
\xrightarrow{\Delta} D$ is identity and $|\Delta^*(D)| = |D|$.
\item
If $Y \subseteq X$ is a closed (resp. open) subscheme of $X$ and
$Y \cap D$ is the scheme-theoretic intersection, then $S(Y, Y \cap D)$
is a closed (resp. open) subscheme of $S(X,D)$.
There is an inclusion of subschemes $S(Y, Y \cap D) \inj \Delta^*(Y)$
which is an isomorphism if $Y$ is open.
\item
Let $Y$ be a subscheme of $X$. Then $|\Delta^*(Y) \cap D| = |Y \cap D| = |S(Y, Y \cap D) \cap D|$.
\item
$S(X,D)_{\rm sing} = D \cup \Delta^{-1}(X_{\rm sing})$.
In particular, $S(X,D)_{\rm sing} = D$ if $X$ is non-singular.
\item  
If $f\colon Y \to X$ is a flat morphism, then $S(Y, f^*(D)) 
\simeq S(X,D) {\underset{X}\times} Y$. In particular, the map
$S(f)\colon S(Y, f^*(D)) \to S(X,D)$ is flat (resp. smooth) if $f$ is so.
\item
$\pi$ is the normalization map and $D$ is a conducting subscheme, if
$X$ is normal.
\end{enumerate}
\end{prop}
\begin{proof}
To prove the proposition, we can assume that $X = \Spec(A)$ is affine. 
Let $p_1, p_2 \colon A \times A \to A$ denote the projections and let 
$q \colon A \to A/I$ be the
quotient map. Set $q_i = q \circ p_i$. Set $\psi_i = p_i \circ \phi$
for $i =1,2$.  Let $\delta\colon A \inj R \inj  A \times A$ denote the diagonal map.
We then have $\psi_i \circ \delta = id_A$ for $i =1,2$ and this yields 
\[
A \times A = \phi \circ \delta(A) \oplus {\rm Ker}(p_2);
\]
\begin{equation}\label{eqn:double-4}
R = \delta(A) \oplus {\rm Ker}(\psi_2) = \delta(A) \oplus I \times \{0\}
\simeq A \oplus I. 
\end{equation}

Since $A \times A$ is a finite free $A$-module and $R$ is an $A$-submodule,
it follows that $R$ is a finite $A$-module. This proves (1).
The item (2) follows immediately from ~\eqref{eqn:double-2}. 

The ideal of $D$ inside $S(X,D)$ is ${\rm Ker}(R \xrightarrow{q \circ \psi_i}
A/I)$, which is  $I \times I$.
Since $\delta^*(I) \subseteq I \times I$, we see that $D \subseteq \Delta^*(D)$
and the composite $D \inj \Delta^*(D) \xrightarrow{\Delta} D$ is clearly
identity. Furthermore, it is clear that $R[(a, b)^{-1}] = A[a^{-1}] \times 
A[b^{-1}]$ and $\delta^*(I)[(a, b)^{-1}] = R[(a, b)^{-1}]$, whenever 
$a, b \in I \setminus \{0\}$.
Hence, we have $|\Delta^*(D)| = |D|$. This proves (3).

To prove (4), we only need to consider the closed part.
Let $A' = A/J$, where $J$ is the ideal defining $Y$ and let
$R' = \{(a', b') \in A' \times A'| a' - b' \in (I+J)/J\}$. Let $\ov{a}$
denote the residue class of $a \in A$ modulo $J$.
Suppose there exist $a, b \in A$ such that $\ov{a} - \ov{b} \in (I+J)/J$.
This means $a-b \in I + J$ and so we can write $a-b = \alpha + \beta$,
where $\alpha \in I$ and $\beta \in J$.
We set $a' = a - \beta$ and $b' = b$. This yields $a' - b' = 
a-b - \beta = \alpha \in I$ and $a'- a = \beta \in J, \ b'-b = 0 \in J$.
We conclude that $(a', b') \in R$ and it maps to $(\ov{a}, \ov{b}) \in R'$.
Hence $R \surj R'$.

An element of $\delta^*(J)$ is of the form $(a \alpha, b \alpha)$, where
$a, b, \in A, \alpha \in J$ and $a-b \in I$. This element clearly dies
in $R'$. Hence $S(Y, Y \cap D) \subseteq \Delta^*(Y)$. 

To prove (5), let $S_Y = S(Y, Y \cap D)$. Then 
\[
|\Delta^*(Y) \cap D| = |\Delta^*(Y) \cap \Delta^*(D)| = |\Delta^*(Y \cap D)|,
\]
where the first equality follows from (3).  
On the other hand, we have
\[
|S_Y \cap D| = |\iota^Y_1 \circ \iota^Y(Y \cap D)| = |\Delta^*_Y(Y \cap D)|
= |\Delta^*(Y \cap D)|,
\]
where the second equality follows from (3) with $X$ replaced by $Y$.
The item (5) now follows. The item (6) follows from (1) and the fact
that more than one components of $S(X,D)$ meet along $D$.

To prove (7), let $Y = \Spec(B)$ and tensor ~\eqref{eqn:double-2} with
$B$. The flatness of $B$ over $A$ yields the short exact sequence
\[
0 \to R {\underset{A}\otimes}B \xrightarrow{\phi} B \times B \to B/{IB} \to 0
\]
and this proves the first part of (7). The second part follows because
a base change of a flat (resp. smooth) map is flat (resp. smooth).

The item (8) follows because $\pi$ is finite and birational and
the ideal of $\pi^*(D)$ in $X \amalg X$ is $I_D \times I_D$ which is actually
contained in $\sO_{S(X,D)}$. So $D$ is a conducting subscheme. 
\end{proof}

\subsection{Double along a Cartier divisor}\label{sec:DCar}
Recall that a morphism $f\colon X \to S$ of schemes is called a {\sl local 
complete intersection} (l.c.i.) at a point $x \in X$ if it is of finite type and
if there is an open 
neighborhood $U$ of $x$ and a factorization 
\[
\xymatrix@C1pc{
& Z \ar[d]^g \\
U \ar[ur]^i \ar[r]_f & S,}
\]
where $i$ is a regular closed immersion and $g$ is a smooth morphism.
We say that $f$ is a local complete intersection morphism if it is so at
every point of $X$. We say that $f$ is l.c.i. along a closed subscheme
$S' \inj S$ if it is l.c.i. at every point in $f^{-1}(S')$.

\begin{prop}\label{prop:double-prp-fine}
Continuing with the notations of \propref{prop:double-prp}, assume further
that $D$ is an effective Cartier divisor on $X$. Then the following
hold.
\begin{enumerate}
\item
$\Delta$ is finite, flat and $\sO_{S(X,D)}$ is a locally free 
$\sO_X$-module of rank two via $\Delta$.
\item
$S(X,D)$ is Cohen-Macaulay if $X$ is so.
\item
If $f\colon Y \to X$ is any morphism, then there is a closed immersion
of schemes $S(Y, f^*(D)) \inj S(X,D) {\underset{X}\times} Y$.
This embedding is an isomorphism if $f$ is transverse to $D \inj X$.
\item
If $f\colon Y \to X$ is any morphism such that $Y$ is Cohen-Macaulay
and $f^*(D)$ does not contain any irreducible component of $Y$, then
the embedding $S(Y, f^*(D)) \inj S(X,D) {\underset{X}\times} Y$ is an
isomorphism. In this case, $f^*(D)$ is an effective Cartier divisor on $Y$.
\item
If $f\colon Y \to X$ is l.c.i. along $D$, then $Y {\underset{X}\times} S(X,D) \to 
S(X,D)$ is l.c.i. along $D$. 
\end{enumerate}
\end{prop}
\begin{proof}
We can again assume that $X = \Spec(A)$ is affine such that
$I = (a)$ is a principal ideal such that $a \in A$ is not a zero-divisor.
It follows then that $I$ is a free $A$-module of rank one.
We can now apply ~\eqref{eqn:double-4} to conclude (1) as 
the finiteness of $R$ over $A$ is already shown in \propref{prop:double-prp}.

To prove (2), let $\fm \subsetneq R$ be a maximal ideal and let $\fn =
\delta^{-1}(\fm)$. Then $A_{\fn} \to R_{\fn}$ is a finite and flat map
and hence $A_{\fn} \to R_{\fm}$ is a faithfully flat local homomorphism of 
noetherian local rings of same dimension. Since $A_{\fn}$ is
Cohen-Macaulay and since this local homomorphism preserves regular sequences,
it follows that ${\rm depth}(R_{\fm}) \ge \dim(A_{\fn}) = \dim(R_{\fm})$.
Hence, $R_{\fm}$ is Cohen-Macaulay.

To prove (3), we let $Y = \Spec(B)$ and tensor ~\eqref{eqn:double-2} (over $A$)
with $B$ to get an exact sequence
\[
0 \to {\rm Tor}^1_A(A/I, B) \to
R {\underset{A}\otimes}B \xrightarrow{\phi} B \times B \to B/{IB} \to 0
\]
and $S(B, IB)$ is (by definition) the kernel of the map $B \times B \to B/{IB}$.
In particular, we get a surjective map of rings $R {\underset{A}\otimes}B \surj
S(B, IB)$. This proves the first part of (3).
The transversality of $B$ with $A/I$ means precisely that
${\rm Tor}^1_A(A/I, B) =0$ and we get that $R {\underset{A}\otimes}B
\xrightarrow{\simeq} S(B, IB)$. This proves (3).

Suppose next that $Y$ is Cohen-Macaulay and no irreducible component of
$Y$ is contained in $f^*(D)$. It suffices to show in this case that
${\rm Tor}^1_A(A/I, B) =0$. Let $f\colon A \to B$ be the map on the
coordinate rings. That no component of $Y$ is contained in $f^*(D)$ 
means that $f(a)$ does not belong to any minimal prime of $B$.
The Cohen-Macaulay property of $B$ implies that it has no embedded 
associated prime. In particular, $f(a)$ does not belong to any 
associated prime and hence is not a zero-divisor in $B$.

We have a short exact sequence
\[
0 \to A \xrightarrow{a} A \to A/I \to 0
\]
which says that ${\rm Tor}^1_A(A/I, B) = {\rm Ker}(B \xrightarrow{f(a)} B)$
and we have just shown that the latter group is zero.
We have also shown above that $f(a)$ is not a zero-divisor on $B$
and this implies that $f^*(D)$ is an effective Cartier divisor on $Y$.
This proves (4).
The item (5) follows from (1) and an elementary fact that l.c.i. morphisms
are preserved under a flat base change.
\end{proof}

\section{Chow group of $0$-cycles on singular schemes}
\label{section:0-C-S}
In this section, we give a definition of the Chow group of $0$-cycles on 
singular schemes that modifies slightly the one given in \cite{LW}. While 
using the same set of generators, we change the geometric condition imposed on 
the curves giving the rational equivalence. In many cases, we are able to show 
that this new definition coincides with the classical one.
It turns out that the modified Chow group of 0-cycles has better 
functorial properties and is more suitable 
for proving \thmref{thm:Intro-6}. 

\subsection{Some properties of l.c.i and perfect morphisms}
\label{sec:Perf}
Recall that a finite type morphism $f\colon X \to S$ of noetherian
schemes is called 
\emph{perfect} if the local ring $\sO_{X,x}$ has finite Tor-dimension
as a module over the local ring $\sO_{S, f(x)}$ for every point $x \in X$.
Equivalently, given any point $x \in X$, there are affine neighborhoods $U$ 
of $x$ and $V$ of $f(x)$ such that $\sO(U)$ is an $\sO(V)$-module of finite
Tor-dimension. Recall also  the following
%Recall from \cite[Proposition~5.12]{Srinivas-1} that if $f$ is 
%a proper and perfect morphism, there is a well defined push-forward map
%$K_0(X)\xrightarrow{f_*} K_0(S)$
%between the Grothendieck groups of vector bundles. 
\begin{prop}[Proposition~5.12, \cite{Srinivas-1}] \label{prop:pushforward-perfect-maps}Let $f\colon X \to S$ be a proper and perfect morphism of noetherian schemes. Then there is a well defined push-forward map
$K_0(X)\xrightarrow{f_*} K_0(S)$
between the Grothendieck groups of vector bundles. 
    \end{prop}
Some known elementary properties of l.c.i and perfect morphisms are recalled 
in the following lemmas.

\begin{lem}\label{lem:lci-5}
\begin{enumerate}
\item
The l.c.i. and perfect morphisms are preserved under flat base change.
\item
A flat morphism of finite type is perfect.
\item
An l.c.i. morphism is perfect.
\item
l.c.i. and perfect morphisms are closed under composition.
\item
l.c.i. and perfect morphisms satisfy faithfully flat (fpqc) descent.
\end{enumerate}
\end{lem}

\begin{lem}\label{lem:lci-6}
Let $f\colon X \to S$ be a finite type morphism of noetherian schemes 
such that for every $x \in
f^{-1}(S_{\rm sing})$, the map $f$ is l.c.i. at $x$. Then $f$ is perfect.
\end{lem}
\begin{proof}
It follows from the definition of a perfect morphism because if $x \in X$
is such that $s = f(x)$ is a regular point of $S$, then $\sO_{X,x}$ has
finite Tor-dimension over $\sO_{S,s}$. This property for the points
over the singular locus of $S$ follows from the hypothesis of the lemma.
\end{proof}

\subsection{Divisor classes for singular curves}\label{ssec:DivClassesCurves}
We fix a field $k$. For $X$ an equidimensional quasi-projective $k$-scheme and $Y\subsetneq X$ a closed subscheme of $X$ not containing any component of $X$, write $\sZ_0(X,Y)$ for the free abelian group on the closed points of $X$  not in $Y$.

A \textit{curve} $C$ will be in what follows a quasi-projective $k$-scheme of pure dimension $1$. We let $k(C)$ denote the ring of total quotients of $C$.
Let $\{\eta_1,\ldots, \eta_r\}$ denote the set of generic points of $C$
with closures $\{C_1, \cdots, C_r\}$. Let $T$ be a set of closed points of $C$ containing $C_{\rm sing}$ and $Z = T\cup \{\eta_1,\ldots, \eta_r\}$. Write $\sO_{C,Z}$ for the semi-local ring on the points of $T$.
This yields a sequence of maps
\begin{equation}\label{eqn:Sing-Rat-eq}
\sO_{C,Z}^{\times} \inj k(C)^{\times} \to  \prod_{i=1}^r k(C_i)^{\times}.
\end{equation}

We let $\theta_{(C,Z)}$ denote the composite map. 
Letting $k(C,Z)^{\times} = \sO^{\times}_{C,Z}$, the localization sequence in 
$K$-theory yields a natural map
\begin{equation}\label{eqn:Sing-Rat-eq-0}
\partial_{C,Z}: k(C,Z)^{\times} \to \amalg_{p \in C \setminus Z} \ G_0(p) 
= \sZ_0(C,Z).  
\end{equation}

If $C$ is a \textit{reduced} curve, it is a Cohen-Macaulay scheme and hence 
the second map in ~\eqref{eqn:Sing-Rat-eq} is an isomorphism.
Thus the group $\cO_{C,Z}^\times$ is the subgroup of $k(C)^\times$ consisting of those $f$ which are regular and invertible in the local rings $\cO_{C,x}$ for 
every $x\in Z$. In this case, the boundary $\partial_{C,Z}(f)$ has a familiar
expression: if we let $\theta_{(C,Z)}(f) = \{f_i\}$, then
$\divf(f) = \sum_i \divf(f_i)$, where $\divf(f_i)$ is the divisor of the rational
function $f_i$ on the integral curve $C_i$. 
If $C$ is not reduced, $\partial_{C,Z}$ has a more complicated expression
which we do not use in this text.

\subsection{A Chow group of 0-cycles on singular schemes}
\label{ssec:Rat-eq-sing-var}
Let $X$ be an 
equidimensional reduced quasi-projective scheme over $k$ of dimension 
$d \ge 1$. 
Let $X_{\rm sing}$ and $X_{\rm reg}$ denote the singular and regular loci of $X$,
respectively.
Let $Y \subsetneq X$ be a closed subset containing $X_{\rm sing}$, but
not containing any component of $X$.
Write again $\sZ_0(X,Y)$ for the free abelian group on closed points of $X \setminus Y$.
We shall often write $\sZ_0(X,X_{\rm sing})$ as $\sZ_0(X)$.

Let $f\colon X' \to X$ be a proper morphism from another reduced 
equidimensional scheme over $k$. Let $Y' \subsetneq X'$ be a closed subset 
not containing any component of $X'$ such that $f^{-1}(Y) \cup 
X'_{\rm sing}\subseteq Y'$.
Then there is a push-forward map 
\begin{equation}\label{eqn:cycle-PF}
f_*\colon  \sZ_0(X', Y') \to \sZ_0(X,Y).
\end{equation}
This is defined on a closed point $x' \in X' \setminus Y'$ with $f(x') = x$ by
$f_*([x']) = [k(x'): k(x)] \cdot [x]$ .

\begin{defn}\label{defn:0-cycle-S-1}
Let $C$ be a  reduced curve in $\Sch_k$ and let $\nu\colon C \to X$ be a finite morphism. 
We shall say that $\nu\colon (C, Z)\to (X,Y)$ is \emph{a good curve
relative to $(X,Y)$} if there exists  a closed proper subscheme $Z \subsetneq C$ such that the following hold.
\begin{enumerate}
\item
No component of $C$ is contained in $Z$.
\item
$\nu^{-1}(Y) \cup C_{\rm sing}\subseteq Z$.
\item
$\nu$ is locally complete intersection morphism  at every 
point $x \in C$ such that $\nu(x) \in Y$.
\end{enumerate}
\end{defn}

Given any good curve $(C,Z)$ relative to $(X,Y)$, we have a 
pushforward map as in 
\nolinebreak
\eqref{eqn:cycle-PF} 
\[
\sZ_0(C,Z)\xrightarrow{\nu_{*}} \sZ_0(X,Y).
\]
We shall write $\sR_0(C, Z, X)$ for the subgroup
of $\sZ_0(X,Y)$ generated by the set 
$\{\nu_*({\rm div}(f))| f \in \sO^{\times}_{C, Z}\}$, where ${\rm div}(f)$ for a rational function $f\in \cO_{C,Z}^\times$ is defined as in 
~\eqref{eqn:Sing-Rat-eq-0} for reduced curves.
Let $\sR_0(X,Y)$ denote the subgroup of $\sZ_0(X,Y)$ which is the image of the
map
\begin{equation}\label{eqn:0-cycle-S-2}
{\underset{\nu\colon (C,Z)\to (X,Y) \ {\rm good}}\bigoplus} \sR_0(C, Z, X) \to 
\sZ_0(X, Y).
\end{equation}
We define the \emph{Chow group of 0-cycles on $X$ relative to $Y$} to be  the 
quotient
\begin{equation}\label{eqn:0-cycle-S-3}
\CH_0(X,Y) = \frac{\sZ_0(X, Y)}{\sR_0(X,Y)}.
\end{equation}
We write $\CH_0(X, X_{\rm sing})$  as $\CH_0(X)$ for short and call it 
the \emph{Chow group of $0$-cycles on $X$}.

The following result shows that we can always  assume that the morphisms 
$\nu\colon C\to X$ are l.c.i
in the definition of our rational equivalence.

\begin{lem}\label{lem:lci-curves}
Let $(X,Y)$ be as above.
Given any good curve $\nu\colon (C,Z) \to (X,Y)$ relative to $(X,Y)$
and any $f \in \sO^{\times}_{C,Z}$, there exists a good curve
$\nu'\colon (C', Z') \to (X,Y)$ relative to $(X,Y)$ and 
$f' \in \sO^{\times}_{C',Z'}$
such that the following hold.
\begin{enumerate}
\item
$\nu_*({\rm div}(f)) = \nu'_*({\rm div}(f'))$.
\item
$\nu'\colon C' \to X$ is an l.c.i. morphism.
\end{enumerate}
\end{lem}
\begin{proof}
Let $U_1 \subseteq C$ be an open subset of $C$ containing 
$S_1 = \nu^{-1}(X_{\rm sing})$ such that
$(C_{\rm sing} \setminus S_1) \cap U_1= \emptyset$.
This is possible because $S_1$ is a finite set.
Let $\pi\colon (C \setminus S_1)^N \to C \setminus S_1$ 
denote the normalization map.
It follows that $\pi\colon  \pi^{-1}(U_1 \setminus S_1) \to U_1 \setminus S_1$ is an 
isomorphism. Setting $U_2 = (C \setminus S_1)^N$, we see that that
$U_1$ and $U_2$ glue along $\pi^{-1}(U_1 \setminus S_1)$ to give a unique
scheme $C'$ and a unique map $p\colon  C' \to C$. %  (see \cite[Ex.~II.2.12]{Hart}).
This scheme has the property that $p$ is finite, $p^{-1}(U_1) \to U_1$ is an 
isomorphism and $p^{-1}(C \setminus S_1) = (C \setminus S_1)^N$.
%(this is just the construction of the $S_1$-normalization of $C$, in the notations of Definition \ref{def:ZNormalization}).  
 
Setting $Z' = p^{-1}(Z)$ and $f' = p^*(f) \in k(C')^{\times}$, we see that
$f' \in \sO^{\times}_{C',Z'}$ and ${\rm div}(f) = p_*({\rm div}(f'))$.
If we let $\nu' = \nu \circ p$, we get $\nu_*({\rm div}(f)) = 
\nu'_*({\rm div}(f'))$.
Furthermore, $\nu'^{-1}(X_{\rm reg}) \to X_{\rm reg}$ is a finite type morphism
of regular schemes and hence is an l.c.i. morphism. 
Since $\nu$ is l.c.i. over $X_{\rm sing}$ and $p$ is an isomorphism in a 
neighborhood of $\nu^{-1}(X_{\rm sing})$, we conclude that $\nu'$ is an l.c.i.
morphism.
\end{proof}

\subsection{The Levine-Weibel Chow group}\label{sec:LWC}
We now recall the Levine-Weibel (cohomological) Chow group of 0-cycles 
for singular schemes as defined in \cite[Definition 1.2]{LW}. 
Let $X$ be an equidimensional 
quasi-projective scheme of dimension $d \ge 1$ over $k$, 
$X\supsetneq Y\supseteq X_{\rm sing}$ a closed subscheme not containing any 
component of $X$.

\begin{defn} A \emph{ Cartier curve on $X$ relative to $Y$} is a 
purely $1$-dimensional closed subscheme $C\hookrightarrow X$ that 
has no component contained in $Y$ and is  defined by a regular 
sequence in $X$ at each point of $C\cap Y$.
\end{defn}
One example of Cartier curves we shall encounter in this text is given by the 
following.
\begin{lem}\label{lem:PB-Cartier}
    Let $X$ be a connected smooth quasi-projective scheme over $k$ and let 
$D\subset X$ be an effective Cartier divisor. 
Let $\nu\colon C\hookrightarrow X$ be an integral curve which is not 
contained in $D$. Assume that $C$ is l.c.i along $D$. Let 
$\Delta\colon S(X,D) \to X$ denote the double construction. Then 
$S(C, \nu^*D)$ is a Cartier curve on $S(X,D)$ relative to $D$.
\begin{proof}

        We write $S(X,D)$ and $S(C, \nu^*(D))$ as $S_X$ and $S_C$, respectively,
        in this proof.
        Since the inclusion $\nu\colon  C \inj X$ is l.c.i. along $D$, 
        it follows \propref{prop:double-prp-fine} that the square

        \begin{equation}\label{eqn:PB-0-cycle-1}
        \xymatrix@C2pc{
        S_C  \ar[d]_{\Delta_C} \ar[r]^{S_{\nu}} & S_X \ar[d]^{\Delta_X} \\
        C \ar[r]_{\nu} & X}
        \end{equation}
        is Cartesian. It also follows from \propref{prop:double-prp-fine}(5)
        that $S_{\nu}\colon  S_C \inj S_X$ is l.c.i. along $D$.
        Moreover, a combination of  \propref{prop:double-prp-fine}(4)
        and \propref{prop:double-prp}(2) tells us that $S_C$ is reduced
        %we can throw away the 0-dimensional components of $S_C$ because it is a
        %dijoint union of its 1-dim. and 0-dim components
        with two components, both isomorphic to $C$.
        We conclude that $S_C \inj S_X$ is a (reduced) Cartier curve relative to $D$.
                \end{proof}
    \end{lem}

Given a Cartier curve $\iota: C \inj X$ relative to $Y$, we let
$\sR^{LW}_0(C,Y, X)$ denote the image of the composite map
$k(C, C\cap Y)^{\times} \xrightarrow{\partial_{C, C \cap Y}} \sZ_0(C, C \cap Y) 
\xrightarrow{\iota_*} \sZ_0(X,Y)$.
We let $\sR^{LW}_0(X,Y)$ denote the subgroup of $\sZ_0(X,Y)$ generated by 
$\sR^{LW}_0(C, Z, X)$, where $C \subset X$ runs through all Cartier curves relative to $Y$.
\begin{defn}\label{defn:LW-Chow-grp}
The \emph{Levine-Weibel Chow group of 0-cycles} of $X$ relative to 
$Y$  is defined as the quotient
\[
\CH^{LW}_0(X, Y) = {\sZ_0(X, Y)}/{\sR^{LW}_0(X,Y)}.
\]
The group $\CH^{LW}_0(X, X_{\rm sing})$ is often denoted by $\CH^{LW}_0(X)$. 
\end{defn}

We recall here the following important moving Lemma, due to Levine, that simplifies the set of relations in case $X$ satisfies additional assumptions.

\begin{prop}[See Lemma~1.4 \cite{Levine-2} and  Lemma~2.1 \cite{BS-1}]\label{prop:LevineMoving} Let $X$ be an equidimensional quasi-projective $k$-scheme and let $X_{\rm sing} \subset Y \subsetneq X$ be a closed subset of $X$ as above. Assume that $X$ is reduced. Then the subgroup $\sR^{LW}_0(X, Y)$ of $\sZ_0(X, Y)$ agrees with the subgroup $\sR^{LW}_0(X, Y)_{\rm red}$, generated by divisors of rational functions on reduced Cartier curves on $X$ relative to $Y$. If $X$ is moreover irreducible, then the Cartier curves generating the rational equivalence can be chosen to be irreducible as well.
    \end{prop}

\begin{lem}\label{lem:lci-LW}
Let $X$ be a reduced quasi-projective $k$-scheme. Then there is a canonical surjection
\begin{equation}\label{eqn:0-cycle-S-3-b}
\CH^{LW}_0(X,Y) \surj \CH_0(X,Y).
\end{equation}
\end{lem}
\begin{proof} The map \eqref{eqn:0-cycle-S-3-b} is induced by the identity on the set of generators, so we just have to show that it is well defined. Since $X$ is reduced, by Proposition \ref{prop:LevineMoving}, we can assume that the Cartier curves defining the rational equivalence on the Levine-Weibel Chow group are reduced.  Now, we just note that a reduced Cartier curve is a
good curve relative to $(X,Y)$.
\end{proof}

\begin{lem}\label{lem:0-cycle-com-0}
Let $X$ be a reduced quasi-projective scheme over $k$ and let $Y\subsetneq X$ 
be a closed subset containing $X_{\rm sing}$ and containing no components of 
$X$. Let $(C,Z)$ be a good curve relative to $(X,Y)$.
Then there are cycle class maps $cyc_C\colon \sZ_0(C,Z) \to K_0(C)$ and
$cyc_X\colon \sZ_0(X,Y) \to K_0(X)$ making the diagram 
\begin{equation}\label{eqn:0-cycle-com-0-1}
\xymatrix@C2pc{
\sZ_0(C,Z) \ar[r]^{cyc_C} \ar[d]_{\nu_*} & K_0(C) \ar[d]^{\nu_*} \\
\sZ_0(X,Y) \ar[r]_{cyc_X} & K_0(X)}
\end{equation}
commutative.
\end{lem}
\begin{proof}
Since $\nu^{-1}(Y) \cup X'_{\rm sing} \subseteq Y'$, we have a push-forward
map $\nu_*\colon\sZ_0(C, Z) \to \sZ_0(X,Y)$, given by $\nu_*([x]) =
[k(x): k(\nu(x))] \cdot [\nu(x)]$.
Since $\nu$ is l.c.i. along $X_{\rm sing}$, it follows from
\lemref{lem:lci-6} that the map $\nu: C \to X$ is perfect.
Hence, there is a push-forward map on $K_0$-groups $\nu_*\colon K_0(C)\to K_0(X)$
by Proposition \ref{prop:pushforward-perfect-maps}. 

To construct the cycle class maps and show that the square commutes, let
$x \in C \setminus Z$ be a closed point and set $y = \nu(x)$. 
Let $\iota_x\colon \Spec(k(x)) \to C$ and $\iota_y\colon \Spec(k(y)) \to X$ be the
closed immersions. Since these maps as well as $\nu$ are perfect
(see \lemref{lem:lci-6}), we have the induced push-forward maps on Grothendieck 
groups of vector bundles and a commutative diagram by Proposition \ref{prop:pushforward-perfect-maps}: 
\begin{equation}\label{eqn:0-cycle-com-1-0}
\xymatrix@C2pc{
\mathbb{Z} = K_0(k(x)) \ar[r]^-{{\iota_x}_*} \ar[d]_{\nu_*} &  K_0(C) \ar[d]^{\nu_*} \\
\mathbb{Z} = K_0(k(y)) \ar[r]_-{{\iota_y}_*} & K_0(X).}
\end{equation}

Setting $cyc_C([x])$ to be ${\iota_x}_*(1)$, we get the
cycle class maps $cyc_C\colon \sZ_0(C) \to K_0(C)$ and $cyc_X\colon \sZ_0(X) \to K_0(X)$
such that ~\eqref{eqn:0-cycle-com-0-1} commutes.
\end{proof}

\begin{lem}\label{lem:0-cycle-com-1}
Suppose that $X$ is reduced and purely $1$-dimensional. Then there is a 
canonical isomorphism $\CH_0(X, Y) \simeq
\CH^{LW}_0(X,Y) \simeq \Pic(X)$.
\end{lem}
\begin{proof}
Let $\nu\colon C \to X$ be a finite map from a reduced curve and let 
$Z \subsetneq C$
be a closed subset such that $(C,Z)$ is good relative to $(X,Y)$.
By \lemref{lem:0-cycle-com-0}, there is a commutative diagram:
\begin{equation}\label{eqn:0-cycle-com-1-1}
\xymatrix@C2pc{
\sZ_0(C, Z) \ar[r]^{cyc_C} \ar[d]_{\nu_*} & K_0(C) \ar[d]^{\nu_*} \\
\sZ_0(X, Y) \ar[r]_{cyc_X} & K_0(X).}
\end{equation}

Let $f \in \sO^{\times}_{C, Z}$. It follows from \cite[Proposition~2.1]{LW}
that $cyc_C({\rm div}(f)) = 0$. In particular, we get 
$cyc_X \circ \nu_*({\rm div}(f)) = \nu_* \circ cyc_C({\rm div}(f)) = 0$.
It follows again from \cite[Proposition~1.4]{LW} that $\nu_*({\rm div}(f)) =0$
in $\CH^{LW}_0(X,Y) \simeq \Pic(X) \inj K_0(X)$. 
We have thus shown that the surjective map
$\CH^{LW}_0(X) \surj \CH_0(X)$ is also injective, hence an isomorphism.
\end{proof}

%\enlargethispage{25pt}

\begin{lem}\label{lem:cycle-class-0-cycles}
Let $X$ be a reduced quasi-projective scheme of dimension $d \ge 1$ 
over $k$ and let $Y\subsetneq X$ 
be a closed subset containing $X_{\rm sing}$ and containing no components of 
$X$. Then the cycle class map $cyc_X\colon \sZ_0(X,Y) \to K_0(X)$ given by 
Lemma \ref{lem:0-cycle-com-0} descends to  group homomorphisms
\[cyc_X\colon \CH_0(X,Y) \to K_0(X); \ \ 
cyc_X^{LW}\colon \CH_0^{LW}(X,Y) \to K_0(X)\]
    making the diagram
     \begin{equation}\label{eq:cyc-class-LW-lci}
     \xymatrix@C1pc{
     \CH_0^{LW}(X,Y) \ar[rd]_{cyc_X^{LW}} \ar@{->>}[rr]^{can}  & & 
\CH_0(X,Y)\ar[ld]^ {cyc_X} \\
      & K_0(X) &
     }
     \end{equation}
    commutative.
    \end{lem}
\begin{proof} The fact that $cyc_X$ yields a cycle class map 
$cyc_X^{LW}\colon \CH_0^{LW}(X,Y) \to K_0(X)$ is proved in 
\cite[Proposition~2.1]{LW}. 
        To show that $cyc_X$ descends to a map on our modified version of the Chow group, let 
        $\nu\colon (C, Z) \to (X,Y)$ be a good curve relative to $(X,Y)$ and 
let 
        $f \in \sO^{\times}_{C,Z}$.
        We then have $cyc_X \circ \nu_*({\rm div}(f)) = 
\nu_* \circ cyc_C({\rm div}(f))$
        by \lemref{lem:0-cycle-com-0}.
        On the other hand, it follows from \lemref{lem:0-cycle-com-1} that
        $cyc_C({\rm div}(f)) =0$. This shows that $cyc_X$ is defined on the 
Chow 
        groups. The commutativity of \eqref{eq:cyc-class-LW-lci} is clear from 
the definitions.
        \end{proof}

\subsection{Comparison of two Chow groups in higher dimension}
\label{sec:Comp>1}
In this section, we prove a comparison theorem for the two Chow groups
in higher dimension. More comparison results in positive characteristic
will be given in Theorems~\ref{thm:Main-Comparison-Chow} and
~\ref{thm:LW-lci-iso-pprimary-torsion}.

Suppose that the field $k$ is algebraically closed and let $d=\dim(X)$. Write $F^d K_0(X)$ for the subgroup of $K_0(X)$ generated by the cycle classes of smooth, closed points in $X$. In \cite[Corollary 5.4]{Levine-5} 
(see also \cite[Corollary~2.7]{Levine-2}), Levine showed the existence of a top  Chern class 
$c_d\colon F^d K_0(X) \to \CH^{LW}_0(X)$ such that $c_d \circ cyc_X^{LW}$ is multiplication
by $(d-1)!$. In particular, the kernel of ${cyc_X^{LW}}$ is torsion. An immediate consequence of Lemma \ref{lem:cycle-class-0-cycles} is then the following.

\begin{cor}\label{cor:0-cycle-Rat}
Let $X$ be a reduced quasi-projective scheme over an  algebraically closed field $k$.
Then the canonical map $\CH^{LW}_0(X)_{\Q} \to \CH_0(X)_{\Q}$ is an isomorphism.
\end{cor}

In order to integrally 
compare the two Chow groups in dimension $\ge 2$, we use the
following.

\begin{prop}\label{prop:cyc-class}
Let $k$ be an algebraically closed field of characteristic zero
and let $X$ be a reduced projective scheme of dimension $d \ge 1$ over $k$. 
Then $cyc_X^{LW}$ is injective.
\end{prop}
\begin{proof}
Let $\alpha \in \CH_0^{LW}(X)$ be such that $cyc_X^{LW}(\alpha) = 0$. By Levine's theorem recalled above, we know that $\alpha$ is a torsion class in $\CH_0^{LW}(X)$.
%It was shown in \cite[Corollary~5.4]{Levine-3} that there is a
%Chern class map $c_{d,X}\colon K_0(X) \to \CH_0^{LW}(X)$ such that
%$c_{d,X} \circ cyc_{X}^{LW} = (d-1)!$.
%It follows that $\alpha$ is a torsion class in $\CH_0^{LW}(X)$.

To show that $\alpha = 0$, we can use the Lefschetz principle argument and 
rigidity of the Chow group of zero-cycles over algebraically closed fields
and assume that $k = \C$. Let $\H^{2d}_{\sD^*}(X, \Z(d))$ denote the 
modified Deligne-Beilinson cohomology of $X$ defined in \cite[\S~2]{Levine-4} (see also Section \ref{sec:Alb-C} below). 
There is then a short exact sequence
\[
0 \to A^d(X) \to \H^{2d}_{\sD^*}(X, \Z(d)) \to H^{2d}(X_{\rm an}, \Z(d)) \to 0
\]
and it was shown in \cite[\S~2]{Levine-4} that there is a Chern class map
$c^d_{\sD^*, X}\colon K_0(X) \to \H^{2d}_{\sD^*}(X, \Z(d))$ which induces an 
Abel-Jacobi map ${\rm AK}^d_X\colon \CH_0^{LW}(X)_{\rm deg \ 0} \to A^d(X)$ given by
${\rm AK}^d_X = c^d_{\sD^*, X} \circ cyc_X^{LW}$, where
$\CH_0^{LW}(X)_{\rm deg \ 0}: = {\rm Ker}(\CH_0^{LW}(X) \to 
H^{2d}(X_{\rm an}, \Z(d)))$.

Since $H^{2d}(X_{\rm an}, \Z(d)))$ is torsion-free and $\alpha$ is torsion, 
it follows that $\alpha \in 
\CH_0^{LW}(X)_{\rm deg \ 0}$. In particular, it is a torsion class in
$\CH_0^{LW}(X)_{\rm deg \ 0}$.
A cycle class map 
\[\wt{{\rm AK}}^d_X\colon  \CH_0^{LW}(X)_{\rm deg \ 0} \to A^d(X)\]
is also constructed in \cite[\S~2]{ESV} and it is shown in 
\cite[Lemma~2.2]{Krishna-1} that $\wt{{\rm AK}}^d_X = {\rm AK}^d_X$ up to
a sign. Now, $cyc_X^{LW}(\alpha) = 0$ implies that 
$c^d_{\sD^*, X} \circ cyc_X^{LW}(\alpha) = 0$
in $A^d(X)$. We conclude that $\alpha$ is a torsion class in 
$\CH_0^{LW}(X)_{\rm deg \ 0}$ such that $\wt{{\rm AK}}^d_X(\alpha) = 0$.
We now apply \cite[Theorem~1.1]{BS-1} to conclude that $\alpha = 0$.
This finishes the proof.
\end{proof}

\begin{remk} Let $X$ be a projective variety over an algebraically closed field $k$ of exponential characteristic $p\geq 1$ and let $Y\subsetneq X$ be a closed subset of $X$ containing $X_{\rm sing}$ and containing no components of $X$. When $\codim_{X}(Y) \geq 2$, the map $cyc_X^{LW}\colon \CH_0^{LW}(X)\to F^dK_0(X)$ is an isomorphism modulo $p$-torsion by \cite[Theorem 3.2]{Levine-2}. In particular, this shows directly that for such $(X,Y)$ the canonical map $\CH_0^{LW}(X,Y)\to \CH_0(X,Y)$ is an isomorphism up to $p$-torsion. Since in this text we are interested in studying cycles on a double $(S({X,D}),D)$, that is not regular in codimension $1$, we can't invoke directly Levine's result even in the projective case over a field of characteristic $0$, and we need the detour of Proposition \ref{prop:cyc-class}. 
\end{remk}

We can now deduce 
our final result comparing the two Chow groups as follows.

\begin{thm}\label{thm:0-cycle-affine-proj-char0}
Let $X$ be a reduced quasi-projective scheme of dimension $d\ge 1$ over 
an algebraically closed field $k$. Then the canonical map
$\CH^{LW}_0(X)  \to \CH_0(X)$ is an isomorphism in the following cases.
\begin{enumerate}
\item
$d \le 2$.
\item
$X$ is affine.
\item
${\rm char}(k) = 0$ and $X$ is projective.
\end{enumerate}
\end{thm}
\begin{proof}
In each case, it suffices to show using \lemref{lem:cycle-class-0-cycles}
that $cyc_X^{LW}$ is injective.
The case $d \le 1$ follows from \lemref{lem:0-cycle-com-1}. 
The $d = 2$ case follows from \cite[Theorem~7]{Levine-1},
where it is shown that the map $\CH_0^{LW}(X) \to F^2K_0(X)$ is an isomorphism. 
If $X$ is affine,
this follows from \cite[Corollary~7.3]{Krishna-2} and
\cite[Corollary~2.7]{Levine-2}.
The last case follows from \propref{prop:cyc-class}.
\end{proof}

\subsection{Some functorial properties of  the Chow group of 0-cycles}
\label{sec:Functorial-prop}
Recall that any proper map $\phi\colon  X' \to X$ admits a push-forward map on 
the Chow groups of 0-cycles when $X$ is smooth. This can not be true  
if $X$ is singular. But we expect such a push-forward to exist in the 
singular case provided $f$ is an l.c.i. morphism. Our next goal is to prove
this in special cases. We shall use this result later in this text.
\begin{prop}\label{prop:0-cycle-PF}
Let $X,Y$ be again as in \lemref{lem:cycle-class-0-cycles}. 
Let $p\colon X' \to X$ be a proper morphism which is l.c.i. over $X_{\rm sing}$
such that $X'$ is reduced.
%and $p^{-1}(X_{\rm sing}) \subseteq X'_{\rm sing}$.
Let $Y' \subsetneq X'$ be a closed subset containing 
$p^{-1}(Y) \cup X'_{\rm sing}$ and not 
containing any component of $X'$. Then there are
push-forward maps $p_*\colon \CH_0(X', Y') \to \CH_0(X,Y)$ and $p_*\colon
K_0(X') \to K_0(X)$ and a commutative diagram
\begin{equation}\label{eqn:0-cycle-PF-1}
\xymatrix@C2.3pc{
\CH_0(X',Y') \ar[r]^>>>>>{cyc_{X'}} \ar[d]_{p_*} & K_0(X') \ar[d]^{p_*} \\
\CH_0(X,Y) \ar[r]_>>>>>{cyc_X} & K_0(X).}
\end{equation}
\end{prop}
\begin{proof}
It follows from our assumption and \lemref{lem:lci-6} that $p$ is perfect and
hence there is a push-forward map $p_*\colon K_0(X') \to K_0(X)$.
We have seen before that there is also a push-forward map $p_*\colon\sZ_0(X', Y')
\to \sZ_0(X,Y)$. 

Let us now consider a good curve $\nu'\colon (C, Z) \to (X',Y')$ relative to 
$(X',Y')$. It follows from our assumption that
$\nu = p \circ \nu' \colon (C, Z) \to (X,Y)$ is a good curve relative to $(X,Y)$.
We have the push-forward maps $\sZ_0(C, Z) \xrightarrow{\nu'_*} \sZ_0(X', Y')
\xrightarrow{p_*} \sZ_0(X,Y)$ such that $\nu_* = p_* \circ \nu'_*$.
This shows that $p_*(\nu'_*({\rm div}(f))) = \nu_*({\rm div}(f))$ for any
$f \in \sO^{\times}_{C,Z}$. This implies that $p_*$ descends to a push-forward
map on the Chow groups.
The commutativity of ~\eqref{eqn:0-cycle-PF-1} is shown exactly as in
the proof of \lemref{lem:0-cycle-com-0}.  
\end{proof}

Combining this with \thmref{thm:0-cycle-affine-proj-char0}, we have a 
similar result for the Levine-Weibel Chow group of $0$-cycles. Note that this type of functoriality was not previously known.
\begin{prop}\label{prop:Sing-PF}
Let $X$ be as in \thmref{thm:0-cycle-affine-proj-char0}.
Let $p \colon  X' \to X$ be a proper morphism between reduced 
quasi-projective schemes over $k$. Let $Y$ denote the singular locus of $X$
and let $Y' \subset X'$ be a closed subscheme containing $p^{-1}(Y) \cup
X'_{\rm sing}$ and not containing any component of $X'$. 
Assume that $p$ is l.c.i. along $Y$.
Then there is a push-forward map 
$p_*\colon  \CH_0^{LW}(X', Y') \to \CH_0^{LW}(X,Y)$.
\end{prop}
\begin{proof}
\propref{prop:0-cycle-PF} says that there are maps
\[
\CH_0^{LW}(X', Y') \to \CH_0(X', Y') \xrightarrow{p_*} \CH_0(X, Y)
\leftarrow \CH^{LW}_0(X, Y)
\]
and \thmref{thm:0-cycle-affine-proj-char0} says that the last map is an 
isomorphism. The result follows. 
\end{proof}

\subsection{Cycles in good position}\label{sec:GP} 
Let $X$ be a smooth quasi-projective 
scheme of pure dimension $d$ over a field $k$. 
For any closed subset $W \subsetneq X$, 
let $\cZ_0(X,W)$ denote the free abelian group on the set of closed points in 
$X\setminus W$. Let $\cR_0(X)_W$ denote the subgroup of $\cZ_0(X,W)$ generated 
by cycles of the form $\nu_*{\rm div}(f)$, where $f$ is a rational function on 
an integral curve $\nu\colon C\hookrightarrow X$ such that 
$C \not\subset W$ and 
$f\in \cO_{C,C\cap W}^\times$. We denote by $\CH_0(X)_W$ the quotient 
$\cZ_0(X,W)/\cR_0(X)_W$. It is the group of $0$-cycles that are 
\emph{in good position} with respect to $W$ (i.e., $0$-cycles missing $W$). 
We have a canonical map $\CH_0(X)_W\to \CH_0(X)$.
The following result is a consequence of Bloch's moving lemma.

\begin{lem}\label{lem:Sing-ML}
Let $X$ be a smooth quasi-projective scheme over $k$. Let
$W \subset X$ be a proper closed subscheme of $X$. 
Then the map $\CH_0(X)_W \to \CH_0(X)$ is an isomorphism.
\end{lem}
\begin{proof}
We can assume that $X$ is connected.
Let $\sZ_1(X)_W$ denote the free abelian group on integral curves in
$X \times \P^1_k$ which have the following properties.
\begin{enumerate}
\item
$C \cap (X \times \{0, \infty\})$ is finite.
\item
 $C \cap (W \times \P^1)$ is finite.
\item
$C \cap (W \times \{0, \infty\}) = \emptyset$.
\item
$C \neq \{x\} \times \P^1$ for any $x \in X$. 
\end{enumerate}
Then, the moving lemma of Bloch \cite{Bloch-M} says that the inclusion of
chain complexes
\[
(\sZ_1(X)_W \xrightarrow{\partial_{\infty} - \partial_0}
\sZ_0(X,W)) \to (\sZ_1(X) \xrightarrow{\partial_{\infty} - \partial_0}
\sZ_0(X))
\]
induces isomorphism on $H^0$. In particular, we get exact sequence
\[
\sZ_1(X)_W \xrightarrow{\partial_{\infty} - \partial_0}
\sZ_0(X,W) \to \CH_0(X) \to 0.
\]

On the other hand, there is an isomorphism 
$(\partial_{\infty} - \partial_0)(\sZ_1(X)_W) \to \sR_0(X)_W$
given by $\partial([C]) \mapsto {\rm div}(N(f))$, where $f$ is the
projection map $C \to \P^1_k$, $C'$ is its image in $X$ and $N(f)$ is the
norm of $f$ under the finite map $k(C') \inj k(C)$.  
One checks easily that $N(f) \in \sO^{\times}_{C', C' \cap Y}$.
\end{proof}

\section{The pull-back maps $\Delta^*$ and $\iota^*_{\pm}$}
\label{sec:PB}
Let $k$ be a field.
Let $X$ be a smooth and connected quasi-projective scheme of dimension 
$d \ge 1$ over $k$ and let $D \subset X$ be an effective Cartier divisor.
Our goal in this section is to define pull-back maps 
$\Delta^*\colon \CH_0(X) \to
\CH_0(S(X,D))$ and $\iota^*_{\pm}\colon \CH_0(S(X,D)) \to \CH_0(X)$.
As before, we shall write $S(X,D)$ in short as $S_X$ as long as the 
divisor $D$ is understood. We shall denote the closed subschemes
$\iota_{\pm}(X)$ of $S_X$ by $X_{\pm}$, each of which is a copy of $X$.

\subsection{The map $\Delta^*$}\label{sec:Sing-PB}
We define the map $\Delta^*\colon \sZ_0(X,D) \to \sZ_0(S_X,D)$ by letting
$\Delta^*([x])$ be the 0-cycle on $S_X$ associated to the closed 
subscheme $\{x\} \times_X S_X$. It follows from \propref{prop:double-prp}(1) 
that $\Delta^*([x]) =
[x_{+}]  + [x_{-}]$, where $x_{\pm}$ is the point $x$ in $X_{\pm} \setminus D$.
Note also that $D = (S_X)_{\rm sing}$ by \propref{prop:double-prp}(6) since
$X$ is non-singular.
We show that $\Delta^*$ preserves rational equivalences.

\begin{thm}\label{thm:PB-main}
Let $X$ be a smooth quasi-projective scheme of dimension $d \ge 1$ over $k$ and
let $D \subsetneq X$ be an effective Cartier divisor. Then 
$\Delta^*\colon  \sZ_0(X,D) \to \sZ_0(S_X, D)$ induces a map
\[
\Delta^*\colon  \CH_0(X) \to \CH_0(S_X).
\]
\end{thm}
\begin{proof}
In view of Lemma \ref{lem:Sing-ML}, we need to show that 
$\Delta^*((f)_C) \in \sR_0(S_X,D)$ for $C \inj X$ an integral curve 
not contained in $D$ and $f \in \sO^{\times}_{C, C \cap D}$.
Let $\nu: C^N \to X$ denote the induced map from the normalization of $C$
and let $E = \nu^*(D)$. We then have $f \in \sO^{\times}_{C^N, E}$
and $(f)_C = \nu_*(\divf(f))$.
We can thus assume that $C$ is normal and allow the possibility that
$\nu$ need not be a closed immersion.
With this reduction, we now
need to show that $\Delta^* \circ \nu_* (\divf(f)) \in \sR_0(S_X, D)$.

Since $\Delta$ is flat, it follows from
~\eqref{eqn:PB-0-cycle-1} that there is a commutative square 
(see \cite[Proposition~1.7]{Fulton})
\begin{equation}\label{eqn:PB-0-cycle-2}
\xymatrix@C2pc{
\sZ_0(C, E) \ar[r]^{\nu_*} \ar[d]_{\Delta^*_C} & \sZ_0(X,D) \ar[d]^{\Delta^*_X} \\
\sZ_0(S_C, E) \ar[r]_{{S_{\nu}}_*} & \sZ_0(S_X,D).}
\end{equation}
%{\color{blue} 
where the bottom horizontal arrow is well defined since 
$S_{\nu}^{-1}(D) \subseteq E$. Note  that by 
Proposition \ref{prop:double-prp-fine}, the map 
$S_\nu\colon S_C\to S_X$ is l.c.i along $D$. 
We also have a commutative square of monomorphisms
\begin{equation}\label{eqn:PB-0-cycle-3}
\xymatrix@C2pc{
\sO^{\times}_{C,E} \ar[d]_{\Delta^*_C} \ar[r]^{\iota_C} & k(C)^{\times} 
\ar[d]^{\Delta^*_C} \\
\sO^{\times}_{S_C, E} \ar[r]_<<<<<{\iota_{S_C}} 
& k(C)^{\times} \times k(C)^{\times}.} 
\end{equation}
Setting $g = \Delta^*_C(f) \in \sO^{\times}_{S_C, E}$,
it is then clear that $\iota_{S_C}(g) = (f, f) \in k(C)^{\times} \times k(C)^{\times} = k(S_C)^{\times}.$  % \Delta^*_C \circ \iota_C(f) = 
This yields 
\[
{\rm div}(\Delta^*_C(f))) = {\rm div}(f,f) = 
{\rm div}(f) + 
{\rm div}(f) = \Delta^*_C({\rm div}(f)).
\]
Combining this with ~\eqref{eqn:PB-0-cycle-2}, we get
\[\Delta^*_X \circ \nu_*({\rm div}(f)) = {S_{\nu}}_* \circ \Delta^*_{C}({\rm div}(f)) = {S_{\nu}}_* ({\rm div}(\Delta^*_C(f))) ={S_{\nu}}_* ({\rm div}(g)).
\]
Since $(S_C, E)$ is clearly a good curve relative to $(S_X, D)$,
the last term lies in $\sR_0(S_X, D)$. 
This finishes the proof. 
\end{proof}

\subsection{The maps $\iota^*_{\pm}$}\label{sec:Sing-Res}
Recall that $\iota_{\pm}\colon X \inj S_X$ denote the two inclusions of $X$ 
in $S_X$ via the map $\pi\colon X \amalg X \to S_X$.
We define two pull-back maps $\iota^*_{\pm}\colon 
\CH_0(S_X) \to \CH_0(X)$. 
We do this for $\iota_{+}$ as the other case is identical.
By Proposition \ref{prop:double-prp}, we have that $(S_X)_{\rm sing} = D$ and $(S_X)_{\rm reg} = (X \setminus D)
\amalg (X \setminus D)$, so that the natural map
\begin{equation}\label{eqn:PB-0-cycle-6}
\sZ_0(S_X,D)\to \sZ_0(X,D) \oplus \sZ_0(X,D) % \xrightarrow{(\iota^*_{+}, \iota^*_{-})} 
\end{equation}
is an isomorphism. We define then $\iota^*_{+}\colon \sZ_0(S_X,D) \to \sZ_0(X,D)$ to be the
first projection of the direct sum in ~\eqref{eqn:PB-0-cycle-6}.
Notice that there are push-forward inclusion maps
${\iota_{\pm}}_* \colon \sZ_0(X,D) \to \sZ_0(S_X,D)$ such that 
$\iota^*_{+} \circ {\iota_{+}}_* = {\rm Id}$
and $\iota^*_{+} \circ {\iota_{-}}_* = 0$.
\begin{prop}\label{prop:PB-pm}
The map 
$\sZ_0(S_X,D) \xrightarrow{(\iota^*_{+}, \iota^*_{-})} \sZ_0(X,D) \oplus 
\sZ_0(X,D)$
descends to the pull-back maps
\begin{equation}\label{eqn:PB-0-cycle-9} 
\iota^*_{\pm}\colon \CH_0(S_X) \to \CH_0(X)
\end{equation}
such that $\iota^*_{\pm} \circ \Delta^* = {\rm Id}$.
\begin{proof}
%{\color{blue} 
For the rational equivalence, we argue as follows: suppose at first that $\nu\colon C \inj S_X$ is an l.c.i curve relative to $D$ and contained in $S_X$ (so, it is in particular a Cartier curve on the double). 
Set $E = \nu^{-1}(D)\cup C_{\rm sing}$.
% As $C$ is Cohen-Macaulay (being a reduced curve)
%and Cartier relative to $D$, it follows that 
%$E$ is an effective Cartier divisor on $C$.}
% by \propref{prop:double-prp-fine}(4).
Let $C'_{\pm}$ denote the unique reduced closed subscheme of $X_{\pm}$
such that $C'_{\pm} \setminus D = (C \setminus D) \cap X_{\pm}$.
%Let $C_{+} = C \times_{S_X} X_+ \inj C$. 
Then $C'_{+}$ is a closed subscheme of $C$ with $\dim(C'_{+}) \leq 1$. 
%{\color{blue} 
It's easy to check that  $\dim(C'_{+}) =0$ would violate the condition of $C$ being locally defined by a regular sequence at every point of intersection $C\cap D \subset S_X$, so that we can assume that $C'_+$ (and similarly $C'_-$) is a union of some irreducible components of $C$.

Let $f \in \sO^{\times}_{C,E}$.
%If $\dim(C'_{+}) = 0$, then it is clear that $\iota^*_+({\rm div}(f)) = 0$.
%This case can occur
Let $C = C'_+ \cup C'_{-} = (\stackrel{r_{+}}{\underset{i=1}\cup} C^i_{+}) \cup
(\stackrel{r_{-}}{\underset{j=1}\cup} C^j_{-})$.
%and $C_+ = (\stackrel{r_{+}}{\underset{i=1}\cup} C^i_+)$. 
We can clearly assume that we have $E_{\rm red} = C'_+ \cap C'_{-}$ and a commutative square

\begin{equation}%\label{eqn:PB-0-cycle-6}
\xymatrix@C2pc{
\sO^{\times}_{C,E} \ar[d]_{\iota^*_+} \ar[r]^{\theta_C} & k(C)^{\times} 
\ar[d]^{\iota^*_+} \ar[r]^<<<<{\simeq} &  k(C'_+)^{\times} \times  
k(C'_{-})^{\times} \ar[dl]^{p_+} \\
\sO^{\times}_{C'_+, E} \ar[r]_<<<<<{\theta_{C'_+}} 
& k(C'_+)^{\times}. & } 
\end{equation}

We can write $\theta_C(f) = (f_+, f_{-})$ with $f_{\pm} \in k(C'_{\pm})^{\times}$
and ${\rm div}_C(f) \ (= (f)_C ) = (f_+)_{C_{+}} + (f_{-})_{C_{-}}$, by definition.
We now consider the diagram

\begin{equation}\label{eqn:PB-0-cycle-7}
\xymatrix@C1pc{
\sZ_0(C'_+, E) \oplus \sZ_0(C'_{-}, E) \ar[dr]_{p_+} &
\sZ_0(C, E) \ar[r]^{\nu_*} \ar[d]_{\iota^*_+} \ar[l]_>>>>{\simeq}& 
\sZ_0(S_X,D) \ar[d]^{\iota^*_+} \ar[r]^>>>{\simeq} &
\sZ_0(X, D) \oplus \sZ_0(X, D) \ar[dl]^{p_+} \\
& \sZ_0(C'_+, E) \ar[r]_{{\nu_+}_*} & \sZ_0(X,D). & }
\end{equation}

This diagram is clearly commutative and yields
\begin{equation}\label{eqn:PB-0-cycle-8}
\begin{array}{lll}
\iota^*_+ \circ \nu_*((f)_C) & = & 
{\nu_+}_* \circ \iota^*_+((f)_C) \\
& = & {\nu_+}_* \circ \iota^*_+\left[{\iota_+}_*((f_+)_{C'_{+}}) +
{\iota_{-}}_*((f_{-})_{C'_{-}}))\right] \\
& = & {\nu_+}_*((f_+)_{C'_+}) + 0 \\
& = & {\nu_+}_*((f_+)_{C'_+}). 
\end{array}
\end{equation}

On the other hand, ${\nu_+}_*((f_+)_{C'_+}) = 
\stackrel{r_+}{\underset{i = 1}\sum}  \divf(f^i_+)$ by definition,
where \[f_+ = (f^1_+, \cdots , f^{r_+}_+) \in k(C'_+)^{\times} = \stackrel{r_+}{\underset{i = 1}\prod}  k(C^i_+)^{\times}.\]
Since each $\divf(f^i_+) \in \sR_0(X)$, we conclude that
$\iota^*_+ \circ \nu_*((f)_C) \in \sR_0(X)$. 
%{\color{blue}
In particular, we obtain that the maps $\iota^*_{\pm}$ descend to group homomorphisms 
$\iota^*_{\pm}\colon \CH_0^{LW}(S_X) \to \CH_0(X)$ from the Levine-Weibel Chow group. 

For the general case of a good curve $\nu\colon C\to S_X$ which is not necessarily an embedding, we can assume
by Lemma \ref{lem:lci-curves} that $C \to S_X$ is a finite l.c.i. morphism.
We now factor $\nu$ as $C\hookrightarrow \mathbb{P}^N_{S_X}\xrightarrow{\pi} S_X$, where $\pi$ is the projection and $\mu\colon C\hookrightarrow \P^N_{S_X}$ is a regular embedding. By Proposition \ref{prop:double-prp}, we can identify $\P^N_{S_X} = \P^N_{S(X,D)}$ with $S(\P^N_X, \P^N_D)$. We have then a commutative diagram
\begin{equation}\label{eqn:iota-pi}
\xymatrix{
\sZ(\P^N_{S_X}, \P^N_D)=\sZ(S(\P^N_X, \P^N_D), \P^N_D) \ar[r]^-{\iota_{\pm}^*} \ar[d]_{\pi_*} & \sZ(\P^N_X) \ar[d]^{\pi_*}\\
 \sZ(S_X, D) \ar[r]_{\iota^*_\pm} & \sZ(X)
}
\end{equation}
by the definition of $\iota_{\pm}^*$.
For $f\in \cO^\times_{C,\nu^*D}$, we have $\iota_\pm^* (\mu_*((f)_C)) \in \sR_0(\P^N_X)$ by the embedded case. In particular, we get
$\iota_\pm^* ( \nu_*((f)_C)) = \iota_\pm^*\pi_* \mu_*((f)_C) =  \pi_* \iota_\pm^* (\mu_*((f)_C))  \in \sR_0(X)$, 
completing the proof.
%}
\end{proof}
\end{prop}

\section{$0$-Cycles on $S_X$ and $0$-cycles with 
modulus on $X$}\label{section:SSM}

\subsection{Overture} 
The goal of this section is to define the {\sl difference map} $\tau^*_X$
from the Chow 
group of 0-cycles on $S_X$ to the Chow group of 0-cycles with modulus $D$ on 
$X$. 
The heart of this section is the proof that this
difference map kills the group of rational equivalences in
the Chow group with modulus $\CH_0(X|D)$. We start by recalling the definition of the Chow group with modulus, following Kerz and Saito.
%{\color{red} We assume\marginpar{\bf Check the finite field case.} in this section that the ground field $k$ is infinite and perfect.}

\subsection{0-cycles with modulus}\label{def:DefChowMod1}
Let $k$ be any field.
Given an integral normal curve ${C}$ over $k$ and an effective divisor
$E \subset {C}$, we say that a rational function $f$ on ${C}$ has modulus
$E$ if $f \in {\rm Ker}(\sO^{\times}_{C, E} \to \sO^{\times}_E)$.
Here, $\sO_{C, E}$ is the semilocal ring of $C$ at the union of
$E$ and the generic point of $C$.
In particular, ${\rm Ker}(\sO^{\times}_{C, E} \to \sO^{\times}_E)$ is just
$k(C)^{\times}$ if $|E| = \emptyset$.
Let $G({C}, E)$ denote the group of such rational functions.

Let ${X}$ be an integral scheme of finite type over $k$ and let $D$ be an 
effective Cartier divisor on $X$. Let $\cZ_0(X\setminus D)$ be the free abelian group on 
the set of closed points of $X\setminus D$. Let ${C}$ be an integral normal 
curve over $k$ and
let $\varphi_{{C}}\colon{C}\to {X}$ be a finite morphism such that 
$\varphi_{{C}}({C})\not \subset D$.  
The push forward of cycles along $\varphi_{{C}}$  
gives a well defined group homomorphism
\[
\tau_{{C}}\colon G({C},\varphi_{{C}}^*(D)) \to \cZ_0(X\setminus D).
\]
Recall now the following Definition
\begin{defn}[Kerz-Saito]\label{def:DefChowMod-Definition}
We define the Chow group $\CH_0({X}|D)$ of 0-cycles of ${X}$ with 
modulus $D$ as the cokernel of the homomorphism 
\begin{equation}\label{eqn:DefChowMod-0}
\divf \colon\bigoplus_{\varphi_{{C}}\colon {C}\to \ov{X}}G({C},
\varphi_{{C}}^*(D)) \to Z_0(X\setminus D),
\end{equation}
where the sum is taken over the set of finite morphisms 
$\varphi_{{C}}\colon {C} \to {X}$ from an integral normal curve such that
$\varphi_{{C}}({C}) \not\subset D$.
\end{defn}

%We remark that the above definition of $\CH_0({X}|D)$ does not require any condition on $k$.
It is known that the Chow group of 0-cycles with modulus is covariantly
functorial for  proper maps: if $f\colon {X'} \to {X}$ is  proper,
$D$ and $D'$ are effective Cartier divisors on ${X}$ and ${X'}$ respectively
such that $f^*(D) \subset D'$, then there is a push-forward
map $f_*\colon \CH_0({X'}|D') \to \CH_0({X}|D)$
(see \cite[Lemma~2.7]{BS} or \cite[Proposition~2.10]{KPv}).

\subsection{Setting and goals}\label{ss:setting-goal-tau}From now on, we fix a smooth   connected quasi-projective scheme ${X}$   
over $k$ and an effective Cartier divisor $D \subset {X}$ on it. 
Let  $S_X$ be the double of ${X}$ along $D$ as defined in 
\S~\ref{sec:Double}. Let $\iota_{\pm}\colon {X} \to S_X$ denote the closed 
embeddings of the two components $X_+$ and $X_-$ of the double. 
We want to construct maps 
\[
\tau_X^*\colon  \CH_0(S_X) \to \CH_0(X|D) \quad 
\text{and }\quad p_{\pm, *}\colon \CH_0(X|D) \to \CH_0(S_X) 
\]
and prove that $ \tau_X^* \circ p_{\pm, *} = {\rm Id}$. At the level of the free abelian group $\sZ_0(S_X,D)$, the map $\tau_X^*$ is simply $\iota^*_{+} -
\iota^*_{-}$:
\[ \cZ_0(S_X, D) = \cZ_0(X\setminus D) \oplus \cZ_0(X\setminus D) \to \cZ_0(X\setminus D), \]
\[ (\alpha_+, \alpha_-) \mapsto \alpha_+ - \alpha_- =\iota^*_{+} ((\alpha_+, \alpha_-) ) -  \iota^*_{-} ((\alpha_+, \alpha_-) ) \]
%{\color{blue} 

We construct the first map in several steps, starting by considering only embedded l.c.i. curves in the definition of the rational equivalence on the double $S_X$ (or, in other words, by proving the existence of the map $\tau_X^*$ for the Levine-Weibel Chow group of zero cycles).  The general case is then  treated using the same trick as in Proposition \ref{prop:PB-pm}, thanks to the following Lemma.

\begin{lem}\label{lem:factor-PN-projection} 
Assume that for every smooth  connected quasi-projective scheme $Y$ over $k$ and effective Cartier divisor $E$ on $Y$, the map $\tau_Y^*$ given above descends to a well defined homomorphism 
    \[\tau_Y^*\colon \CH_0^{LW}(S(Y,E))\to \CH_0(Y|E).\]
Then the map $\CH_0^{LW}(S_X) \to \CH_0(X|D)$ factors through $\CH_0(S_X)$, giving a well defined homomorphism
\[\tau_X^*\colon  \CH_0(S_X) \to \CH_0(X|D).\]
\begin{proof} Let $\delta\colon \sZ_0(S_X, D)\to \CH_0(X|D)$ be the composition \[\sZ_0(S_X, D)\to \CH_0^{LW}(S_X)\to \CH_0(X|D).\] We have to show that $\delta$ factors through $\CH_0(S_X)$, defined using good l.c.i. curves as in Definition \ref{defn:0-cycle-S-1}. Using again Lemma \ref{lem:lci-curves}, we have to show more precisely that $\delta( \nu_*( {\rm div}_C(f))=0$ for every $\nu\colon C\to S_X$ finite l.c.i. morphism from a reduced curve $C$ that is good relative to $(S_X,D)$ and for every rational function $f$ on $C$ that is regular and invertible along $E= (\nu^{-1}(D)\cup C_{\rm sing})$. We factor $\nu$ as composition $\nu = \pi\circ \mu$, where $\mu\colon C\hookrightarrow \P^N_{S_X} = S(\P^N_X, P^N_D)$ (using Proposition \ref{prop:double-prp}) is a regular embedding and $\pi\colon \P^N_{S_X}\to S_X$ is the projection. In particular, $\mu(C) =C$ is a Cartier curve on the double $S(\P^N_X, P^N_D)$ relative to $\P^N_D$. 

It follows from ~\eqref{eqn:iota-pi} and the formula
$\delta = \iota^*_+ - \iota^*_-$ that the square
\[
\xymatrix@C1pc{
\sZ_0(\P^N_{S_X}, \P^N_D) \ar[r]^{\delta_{\P^N_{S_X}}} \ar[d]_{\pi_*} & 
\CH_0(\P^N_X|\P^N_D) \ar[d]^{\pi_*} \\
\sZ_0(S_X, D) \ar[r]_{\delta} & \CH_0(X|D)}
\]
commutes, where $\delta_{\P^N_{S_X}}$ is the composition
$\sZ_0(\P^N_{S_X}, \P^N_D) \to \CH_0^{LW}(S(\P^N_X, \P^N_D)) \to 
\CH_0(\P^N_X| \P^N_D)$. That is,
$\delta(\nu_*((f)_C) )=\delta (\pi_* (\mu_* (f)_C)) = 
\pi_*( \delta_{\P^N_{S_X}}(\mu_*( (f)_C)))$.

By assumption, we have 
$\delta_{\P^N_{S_X}}(\mu_*( (f)_C)) =0 \in \CH_0(\P^N_X| \P^N_D)$.
Since the push-forward map $\pi_*\colon \CH_0(\P^N_X| \P^N_D)\to \CH_0(X|D)$ is 
well defined, we can conclude.
    \end{proof}   
\end{lem} 

We have therefore reduced the problem to showing that the map $\tau^*_X$ is 
well defined from the Levine-Weibel Chow group of zero cycles.
%}

To begin, we need to keep some control over
the Cartier curves on the double $S_X$ which generate
$\sR_0(S_X,D)^{LW}$. We do this in the next few lemmas.

\subsection{The Cartier curves on the double $S_X$}
\label{sec:Car-db}
Our next goal is to use the specific structure of $S_X$ as a 
join of two smooth schemes to refine the set of Cartier
curves used to define $\CH_0^{LW}(S_X,D)= \CH_0^{LW}(S_X)$.
In this section, by Cartier curve on $S_X$ (resp. $X$) we will always mean 
curve on $S_X$ (resp. $X$) which is Cartier with respect to $D$.
For the rest of \S~\ref{section:SSM}, we shall assume that the ground field
$k$ is infinite and perfect.

\begin{notat} Let $S$ be a quasi-projective $k$-scheme and let $\cL$ be a line bundle on $S$. For a global section $t\in H^0(S, \cL)$, we write $(t)$ for the divisor  of zeros of $s$, that we consider as a closed subscheme of $S$.
    \end{notat}

\begin{lem}\label{lem:good-Cartier-curves-dim>2}
Let $X$ be a   connected smooth quasi-projective scheme of dimension
$d \ge 3$ over $k$.
The group of rational equivalences $\sR^{LW}_0(S_X,D)$ 
is generated by the 
divisors of functions on (possibly non-reduced) 
Cartier curves $C\inj S_X$ relative to $D$,
where $C$ satisfies the following.
\begin{enumerate}
\item 
There is a locally closed embedding
$S_X \inj \P^N_k$ and distinct hypersurfaces \[H_1, \cdots , H_{d-2} \inj \P^N_k\] 
such that 
$Y = S_X \cap H_1 \cap \cdots \cap H_{d-2}$ is a complete
intersection which is geometrically reduced. % and Cohen-Macaulay.
\item
$X_{\pm} \cap Y = X_{\pm} \cap H_1 \cap \cdots \cap H_{d-2}:= Y_{\pm}$ are
  integral.
\item
No component of $Y$ is contained in $D$.
\item
$C \subset Y$.
\item
$C$ is a Cartier divisor on $Y$.
\item
$Y_{\pm}$ are smooth away from $C$.
\item 
$C$ is Cohen-Macaulay.
\end{enumerate}
\end{lem}    
\begin{proof}
Let $C \inj S_X$ be a reduced Cartier curve relative to $D$ and let 
$f \in \sO^{\times}_{C, C \cap D}$.
Since $C$ is Cartier along $D = (S_X)_{\rm sing}$, it follows that
it is Cartier in $S_X$ along each of its generic points. 

We will replace $C$ by the curves of desired type using
a combination of \cite[Lemma~1.3]{Levine-2} and 
\cite[Lemma~1.4]{Levine-2} as follows. We also refer to 
\cite[Chap.~I, \S~6]{Jou} for the Bertini theorem for 
geometrically reduced schemes and \cite[Theorem~12]{Seidenberg} irreducible
schemes.

Since $S_X$ is reduced quasi-projective, $X$ is  integral and 
$C \inj S_X$ is Cartier along $D$,
we can apply \cite[Lemma~1.3]{Levine-2} to find a locally closed embedding
$S_X \inj \P^N_k$ and distinct hypersurfaces $H_1, \cdots , H_{d-2} \inj \P^N_k$ 
such that:
\begin{enumerate}
\item
$Y = S_X \cap H_1 \cap \cdots \cap H_{d-2}$ is a complete
intersection which is (geometrically) reduced (note that since $k$ is perfect, $Y$ is reduced if and only if is geometrically reduced). 
\item
$X_{\pm} \cap Y = X_{\pm} \cap H_1 \cap \cdots \cap H_{d-2}:= Y_{\pm}$ are  
integral.
\item
No component of $Y$ is contained in $D$.
\item
$C \subset Y$.
\item
$C$ is locally principal in $Y$ in a neighborhood of the finite set
$C \cap D$ and at each generic point of $C$.
\end{enumerate}
    
Since $X_{\pm} = X$ are smooth, we can apply \cite[Theorem~1]{KL}
to further assume that $Y_{\pm}$ are smooth away from $C_{\pm}$. Here $C_\pm$ denotes $C\cap X_\pm = C\times_{S_X} X_{\pm}$.

Since $Y \inj S_X$ is constructed as a complete intersection of hypersurfaces
of arbitrarily large degrees (see \cite[Lemma~1.4]{Levine-2}),
we can furthermore find a locally principal closed subscheme
$C_1$ of $Y$ such that
\begin{listabc}
\item
$C \subset C_1$.
\item
$C_1$ equals $C$ at every generic point of $C$.
\item
$\ov{(C_1 \setminus C)} \cap C \cap D = \emptyset$.
\end{listabc}

Since $Y$ is geometrically reduced, it follows from \lemref{lem:Cycle-mod-1} that $Y$ is 
the join of $Y_{\pm}$ along $D$ and there is a short exact sequence
\begin{equation}\label{eqn:good-Cartier-curves-dim>2-0}
0 \to \sO_Y \to \sO_{Y_+} \times \sO_{Y_-} \to \sO_{Y \cap D} \to 0.
\end{equation}

Since $C_1$ is locally principal in $Y$, it follows that 
$C_1 \cap X_{\pm} = C_1 \cap Y_{\pm}$ are locally principal in $Y_{\pm}$. 
Since $Y_{\pm}$ are   integral surfaces, it follows that $C_1 \cap X_{\pm}$ are 
Cartier divisors in $Y_{\pm}$.
It follows from ~\eqref{eqn:good-Cartier-curves-dim>2-0} that 
$C_1$ is a Cartier divisor in $Y$. Since $S_X$ is Cohen-Macaulay and
$Y \subset S_X$ is a complete-intersection, it follows that $Y$ is also
Cohen-Macaulay. Since $C_1 \subset Y$ is a Cartier divisor, we conclude that
$C_1$ is Cohen-Macaulay.

It follows from  (c) that $f$ extends to a function 
$g \in \sO^{\times}_{C_1, C_1 \cap D}$ by setting
$g = f$ on $C$ and $g = 1$ on $C_1 \setminus C$. In particular, we have
$\divf(f) = \divf(g)$.
This finishes the proof.
\end{proof}

We now further refine the rational equivalence by specifying the shape of the Cartier curves that generate the group of relations.

\begin{lem}\label{lem:reduction-basic}
Let $X$ be a   connected smooth quasi-projective scheme of dimension
$d \ge 2$ over $k$. Let $\nu\colon C\to S_X$ be a (possibly non-reduced)
Cartier curve relative to $D\subset S_X$. 
Assume that either $d = 2$ or there are inclusions 
$C \subset Y \subset S_X$, where
$Y$ is a geometrically reduced complete intersection surface and $C$ is a Cartier divisor on 
$Y$, as in \lemref{lem:good-Cartier-curves-dim>2}.
Let $f \in \cO_{C, \nu^*D}^{\times} \subset k(C)^\times$, where
$k(C)$ is the total quotient ring of $\sO_{C, \nu^*D}$. 

We can then find two Cartier curves $\nu'\colon C' \inj S_X$ and 
$\nu''\colon C'' \inj S_X$ relative to $D$ satisfying the following.
\begin{enumerate}
\item
There are very ample line bundles $\sL', \sL''$ on $S_X$ and sections
$t' \in H^0(S_X, \sL'), \ t'' \in H^0(S_X, \sL'')$ such that
$C' = Y \cap (t')$ and $C'' = Y \cap (t'')$ (with the convention $Y=S_X$ if $d=2$).
\item
$C'$ and $C''$ are geometrically  reduced.
\item
The restrictions of both $C'$ and $C''$ to $X$ via the two closed 
immersions $\iota_{\pm}$ are integral curves in $X$,
which are Cartier and smooth along $D$.
\item
There are functions $f' \in \sO_{C', (\nu') ^*D}^{\times}$ and
$f'' \in \sO_{C'', (\nu'') ^*D}^{\times}$ such that
$\nu'_*({\divf}(f'))  + \nu''_*({\divf}(f''))  = 
\nu_*({\divf}(f))$ in $\sZ_0(S_X, D)$.  
\end{enumerate}
\end{lem}
\begin{proof}
In this proof, we shall assume that $Y = S_X$ if $d =2$.
In this case, a Cartier curve on $S_X$ along $D = (S_X)_{\sing}$ must be  
an effective Cartier divisor on $S_X$. Hence we can assume that 
$C$ is an effective Cartier divisor on $Y$ for any $d \ge 2$. 
Notice that $Y = Y_+ \amalg_D Y_-$, again by  \lemref{lem:Cycle-mod-1}.

Arguing as in \lemref{lem:good-Cartier-curves-dim>2}, since $S_X$ (and hence $Y$) is quasi-projective over $k$, we can find
an effective Cartier divisor $\tilde{C}$ on $Y$ such that
\begin{enumerate}
\item
$C \subset \tilde{C}$.
%\item
%$C'$ equals $C$ at each generic point of $C$.
\item
$\ov{\tilde{C}\setminus C} \cap (C \cap D) = \emptyset$.
\item
$\sO_{Y}(\tilde{C})$ is a very ample line bundle on $Y$.
\end{enumerate}

We can extend the function $f$ on $C$ to a function $\tilde{f}$ on $\tilde{C}$ by setting $\tilde{f} = f$ on $C$ and $\tilde{f} =1$ on $\tilde{C} \setminus C$. Condition (2) guarantees that $\tilde{f}$ is regular and invertible at each point of $\tilde{C}\cap D$, and it is clear by construction that
$\divf(\tilde{f}) = \divf(f)$. Replacing $C$ with $\tilde{C}$ (and changing the notation for simplicity), we can thus assume that $C$ is a Cartier divisor on $Y$ such that the associated line bundle $\sO_{Y}(C)$ is very ample. 
%Note that $C'$ is not claimed to be reduced. 
Choose $t_0 \in H^0(Y, \sO_{Y}(C))$ such that $C = (t_0)$.
Since $Y$ is geometrically reduced, standard Bertini (see the proof of \cite[Lemma~1.3]{Levine-2})  allows us to choose another divisor $C_{\infty}$ in the linear system $H^0(Y, \sO_{Y}(C))$ such that:
\begin{enumerate}
\item
$C_{\infty}$ is reduced.
\item
$C_{\infty} \cap C \cap D = \emptyset$.
\item
$C_{\infty}$ contains no component of $C$.
\item
$D$ contains no component of $C_{\infty}$.
\end{enumerate}

Denote by $t_{\infty}$ the section of $\cO_Y(C)$ with  $C_\infty = (t_{\infty})$. As before, we extend the function $f$ on $C$ to a function $h$ on $C_{\infty} \cup C$ by setting $h=f$ on $C$ and $h=1$ on $C_{\infty}$. Notice that $h$ is meromorphic on $C_{\infty} \cup C$ and regular
invertible in a neighborhood of 
$(C_{\infty} \cup C) \cap D$ by (2). Let $S$ denote the finite
set of points 
$S = \{\text{poles of $f$}\} \cup (C\cap C^{\infty})$. Note that $S\cap D = \emptyset$.  

Choose now a very ample line bundle $\sL = j^*(\sO_{\P^r}(1))$ on $S_X$, corresponding to an embedding $j\colon S_X\to \P^r_k$ for $r \gg 0$ of $S_X$ as locally closed subscheme. We can assume that $r$ is sufficiently large so that $\sL_{|Y} \tensor_{\sO_Y} \sO_Y(C)$ is also very ample on $Y$. Since $k$ is infinite, there is a dense open subset $V_X$ of the dual projective space $(\P_k^r)^\vee$ such that for $L\in V_X$, the scheme theoretic intersections $L\cdot S_X = L\times_{\P^r_k}S_X$ and $L\cdot Y = L\times_{\P^r_k} Y$ satisfy the following list of properties:
\begin{listabc}
	\item $L\cdot S_X$ and $L\cdot Y$ are geometrically reduced (since both $S_X$ and $Y$ are),
	\item $Y\not\subset L\cdot S_X$,
	\item $L\cdot X_{\pm} = L\times_{\P^r_k} X_\pm$ and $L\cdot Y_{\pm} = L\times_{\P^r_k} Y_\pm$ are   integral (since $X_\pm$ and $Y_\pm$ are),
	\item $L\cdot Y \cap (C\cup C_\infty) \cap D = \emptyset$,
	\item $L\cdot Y\supset S$,
    \item  $L\cdot Y$ intersects $C_{\infty} \cup C$ in a finite set of points.
\end{listabc}
The hyperplane $L$ corresponds to a section $l$ of the linear system $H^0(\P^r_k, \sO_{\P^r}(1))$. Write $s_\infty$ for the global section of $\cL$ that is the restriction of $l$ to $S_X$. Then $L\cdot S_X = (s_\infty)$. Note that we can assume that $s_\infty$ does not have poles on $(C\cup C_\infty) \cap D$. 
Let $\ov{S_X}$ be the closure of $S_X$ in $\P^r_k$ and let $\sI$ be the ideal sheaf of $\ov{S_X} \setminus S_X$ in $\P^r_k$.  We can find a section $s'_{\infty}$ of the sheaf $\sI \otimes \sO_{\P^N_k}(m)$ for some
$m \gg 0$, which restricts to a section $s_{\infty}$ on $S_X$ satisfying
the properties (a) - (f) on $S_X = \ov{S_X} \setminus V(\sI)$.
This implies in particular that 
$S_X \setminus (s_{\infty}) = \ov{S_X} \setminus (s'_{\infty})$ 
is affine. Thus, up to taking a further Veronese embedding of $\P^r_k$ (replacing $s_\infty$ with $s_\infty'$) we can assume that (a) - (f) as above as well as the following hold
\\
\hspace*{1cm} g) \ $S_X \setminus (s_\infty) = S_X \setminus L\cdot S_X$ is affine.

We now know that $Y_\pm$ is smooth away from $C$ and d) tells us that 
$(s_\infty)$ intersects $Y_\pm$ along $D$ at only those points which are away 
from $C$. It follows that $(s_\infty) \cap Y_\pm \cap D \subset
(s_\infty) \cap (Y_+)_{\rm reg}$. We can then use the Bertini theorem of Altman and Kleiman  
\cite[Theorem~1]{KL} to ensure that $(s_\infty) \cap (Y_+)_{\rm reg}$ is
smooth away from the subscheme $(C\cup C_\infty )\cap X_+$ where we ask the containment condition. In particular, we can assume that $(s_\infty)\cap Y_+$ is smooth along $D$.  The same holds for $Y_-$ as well. We conclude that we can moreover achieve the following property
\\
\hspace*{1cm} h) \ $(s_{\infty}) \cap Y_{\pm}$ are smooth along $D$.

Consider again the function $h$ on $C \cup C_\infty$. By our choice of $S$, $h$ is regular on $(C\cup C_\infty) \setminus S$. By g) above, $h$ extends to a regular function $H$ on the affine open $U = S_X \setminus (s_{\infty})$. Since $H$ is a meromorphic function on $S_X$ which has poles only along
$(s_{\infty})$, it follows that for $N\gg 0$ the section $H s_\infty^N$ is an element of $H^0(U, \cL^{\tensor N})$ which extends to a section $s_0$ of $\cL^{\tensor N}$ on all $S_X$. Since $h$ is regular and invertible at each point of $C\cup C_\infty \cap D$ and since $s_\infty$ does not have zeros or poles on $C\cup C_\infty \cap D$, it follows that $(s_0)\cap (C\cup C_\infty)\cap D = \emptyset$ and (using f) above) that $(s_0)$ does not contain any component of $C\cup C_\infty$. Note that up to replacing $s_0$ by $s_0s_\infty^i$, we are free to choose $N$ as large as needed. 

Write $\sI_{C \cup C_\infty}$ for the ideal sheaf of $C\cup C_\infty$ in $S_X$. We can then find sections $s_1,\ldots, s_m$ of $H^0(S_X, \cL^{\tensor N} \tensor\sI_{C \cup C_\infty})$ such that the rational map $\phi\colon S_X\dashrightarrow \P^{m-1}_k$ that they define is a locally closed immersion on $S_X\setminus (C\cup C_\infty)$. In particular, there exists an affine open neighbourhood $U_x$ of every 
$x\in S_X\setminus (C\cup C_\infty)$ where at least one of the $s_i$ is not identically zero and where the $k$-algebra $k[s_1/s_i, \ldots, s_m/s_i]$ generated by $s_1/s_i, \ldots, s_m/s_i$ coincides with the coordinate ring of $U_x$. But then the same must be true for the algebra  $k[s_0/s_1, s_1/s_i, \ldots, s_m/s_i]$ obtained by adding the element $s_0/s_i$. Hence, the rational map $\psi\colon S_X \dashrightarrow \P^m_k$ given by the sections $(s_0, s_1, \ldots, s_m)$ of $H^0(S_X, \cL^\tensor N)$ is also a locally closed immersion on $S_X\setminus (C\cup C_\infty)$, and since the base locus of the linear system associated to $(s_0, s_1, \ldots, s_m)$  is $(s_0)\cap (C\cup C_\infty)$, it is in fact a morphism away from $(s_0) \cap (C\cup C_\infty)$. 

In particular, $\psi$ is birational (hence separable) and has image of dimension at least two, so that the linear system $V = (s_0, s_1, \ldots, s_m)$ is not composite with a pencil. By the classical Theorem of Bertini (see, for example, \cite[Theorem I.6.3]{Zar58}), a general divisor $E$ in $V$ is generically geometrically reduced. Moreover, $E$ is itself a Cohen-Macaulay scheme (since $S_X$ is Cohen-Macaulay). But a locally Noetherian Cohen-Macaulay scheme that is generically reduced is in fact reduced by \cite[Prop. 14.124]{GoW}. Hence the general divisor $E$ in $V$ is indeed geometrically reduced (hence reduced). We can apply the same argument to $E\cdot Y$, noting that being a complete intersection in $S_X$, the surface
$Y$ is Cohen-Macaulay.
%the base locus of the linear system $V$ restricted to $Y$ lives in the Cohen-Macaulay locus of $Y$.

Since $Y_\pm$ and $X_\pm$ are all   integral, a general divisor $E$ of $V$ will be moreover   irreducible when restricted to $X_\pm$ and $Y_\pm$. We can therefore assume that there is a global section $s_0'$ of $H^0(S_X, \cL^{\tensor N})$ of the form $s_0' = s_0 +\alpha$ with $\alpha \in (s_1,\ldots, s_m)$, that satisfies the following properties:
\begin{listabcprime}
\item $(s'_0)$ and $(s'_0) \cap Y$ are (geometrically) reduced.
\item
$Y \not\subset (s'_0)$.
\item
$(s'_0) \cap X_{\pm}$ and $(s'_0) \cap Y_{\pm}$ are   integral.
\item
$(s'_0) \cap (C \cup C_{\infty}) \cap D = \emptyset$.
%\item
%$(s'_0) \supset S$.
\item
$C_{\infty} \cup C$ contains no component of $(s'_0) \cap Y$.
\item
$(s'_0) \cap Y_{\pm}$ are smooth along $D$.
\end{listabcprime}

The items a') and c') follow from the previous discussion. Properties b'), d') and e') are clearly open conditions on the space of sections $V$, and are therefore satisfied by the general divisor. As for the last item f'), it follows from the classical Theorem of Bertini on smoothness applied to $Y_\pm$, the assumption on $Y_\pm$ and the item d'). 

We then have 
\[
\frac{s'_0}{s^N_{\infty}} = \frac{Hs^N_{\infty} + (\alpha s^{-N}_{\infty})
s^N_{\infty}}{s^N_{\infty}} = H + \alpha s^{-N}_{\infty} = H', \ (\mbox{say}).
\]

Since $\alpha$ vanishes along $C \cup C_{\infty}$ and $s_{\infty}$ does not vanish identically on $U\cap (C\cup C_\infty)$ by f), it follows that $H'_{|{(C \cup C_{\infty}) \cap U}} =
H_{|{(C \cup C_{\infty}) \cap U}} = h_{|U}$. In other words, we have
${s'_0}/{s^N_{\infty}} = h$ as rational functions on $C \cup C_{\infty}$.
We can now compute:

\[
\nu_*({\rm div}(f)) = (s'_0) \cdot C - N (s_\infty) \cdot C
\]
\[ 
0 = {\rm div}(1) = (s'_0) \cdot C_{\infty} - N (s_\infty)\cdot C_{\infty}.
\]

Setting $(s^Y_{\infty}) = (s_\infty) \cap Y$ and $(s'^Y_0) = (s'_0) \cap Y$,
we get

\[
\begin{array}{lll}
\nu_*(\divf(f)) & = & (s'_0) \cdot (C - C_{\infty}) - 
N(s_{\infty})(C - C_{\infty}) \\
& = & (s'^Y_0) \cdot (\divf({t_0}/{t_{\infty}})) - N(s^Y_{\infty}) \cdot
(\divf({t_0}/{t_{\infty}})) \\
& = & \iota_{{s'^Y_0}, *}(\divf(f')) - N \iota_{{s^Y_{\infty}}, *}(\divf(f'')),
\end{array}
\]
where $f' = ({t_0}/{t_{\infty}})|_{(s'^Y_0)} \in 
\sO^{\times}_{(s'^Y_0), D \cap (s'^Y_0)}$ (by (d')) and 
$f'' = (t_0/t_{\infty})|_{(s^Y_\infty)} \in 
\sO_{(s^Y_\infty), D\cap (s^Y_\infty) }^{\times}$ (by (d)). 

It follows from h) and f') that 
$(s'^Y_{0})_{| X_+}, \ (s'^Y_{0})_{| X_-} , \ (s^Y_{\infty})_{| X_+}$ and 
$(s^Y_{\infty})_{| X_-}$ are all smooth along $D$.
Setting $\sL'' = (\sL')^N, \ t'' = (s'_0)$ and $t' = (s_\infty)$,
the curves $C' = (t') \cap Y$ and $C'' = (t'') \cap Y$ together with the 
functions $f'$ and $f''$ satisfy the conditions of the Lemma.
\end{proof}

\begin{remk}\label{remk:Surface-Bertini}
It follows from \lemref{lem:Cycle-mod-1} that each of the curves 
$C'$ and $C''$ constructed in the proof of the lemma is of the form 
$C_+ \amalg_E C_-$ for $E = \nu_+^*(D) = \nu_-^*(D)$ and 
$\nu_{\pm} \colon C_{\pm} \inj X$ are   integral Cartier curves in $X$, which are
smooth along $D$. 
Moreover, by construction, $C_{\pm}$ are the zero loci of the
restrictions $t_{\pm}$ to $Y_{\pm}$ 
of a global section of a very ample line bundle 
$\cM$ on $S_X$. 
It follows that $(t_+)_{|D}$ and $(t_-)_{|D}$ agree.
\end{remk}

\subsection{The map $\tau_X^*$: the case of curves and surfaces}
\label{sec:dim-2}
For any smooth quasi-projective scheme $X$ over $k$ with effective 
Cartier divisor $D\subset X$, we have defined in \ref{ss:setting-goal-tau} the map
\begin{equation}\label{eqn:tau^*-0}
\tau^*_X\colon  \sZ_0(S_X,D) \to \sZ_0(X,D) \quad \text{by} \quad
\tau^*_X([x]) = \iota^*_+([x]) - \iota^*_-([x]),
\end{equation}
where $x \in S_X \setminus D$ is a closed point.

Suppose first that $X$ is a smooth curve.
We can clearly assume that $X$ is connected.
If $f \in \sO^{\times}_{S_X,D}$ and $\theta_X\colon  \sO^{\times}_{S_X,D} \to 
k(S_X)^{\times}
= k(X_{+})^{\times} \times k(X_-)^{\times}$ is the natural map, then
$\theta_X(f) = (f_+, f_-)$ with $f_{\pm} \in \sO^{\times}_{X_{\pm}, D} = 
\sO^{\times}_{X,D}$.
It follows from ~\eqref{eqn:double-2} that $g := f_+ f^{-1}_- \in 
\Ker(\sO^{\times}_{X,D}  \to \sO^{\times}_D)$. 
Moreover, $\tau^*_X(\divf(f)) = \iota^*_+(\divf(f)) -
\iota^*_-(\divf(f)) = \divf(f_+) - \divf(f_-) = \divf(g)$. We conclude that
$\tau^*_X$ descends to a map
\begin{equation}\label{eqn:tau^*-1}
\tau^*_X\colon  \CH_0^{LW}(S_X) = \CH_0(S_X) \to \CH_0(X|D).
\end{equation}

\vskip .4cm

\begin{prop}\label{prop:map-diff-surfaces} 
Let $X$ be a smooth   connected quasi-projective surface over $k$ with an 
effective Cartier divisor $D$. Then the map 
$\tau_X^*\colon \sZ_0(S_X, D) \to \sZ_0(X|D)$ of ~\eqref{eqn:tau^*-0} 
descends to a group homomorphism $\CH^{LW}_0(S_X)\to \CH_0(X|D)$.
\end{prop}
\begin{proof}
We shall continue to use the notations and various constructions in the
proof of \lemref{lem:reduction-basic}.
We have shown in \lemref{lem:reduction-basic} that in order to
prove that $\tau^*_X$ preserves the subgroups of rational equivalences,
it suffices to show that $\tau^*_X(\divf(f')) \in \sR_0(X|D)$, where $f'$ is a 
rational function on a Cartier curve $\nu\colon  C' \inj S_X$ that we can choose in the following way.
\begin{enumerate}
\item
There is a very ample line bundle $\sL$ on $S_X$ and sections 
$t \in H^0(S_X, \sL), t_{\pm} = \iota^*_{\pm}(t) \in H^0(X, \iota^*_{\pm}(\sL))$ 
such that
$C' = (t)$ and $C'_{\pm} = (t_{\pm})$.
\item
$C'$ is a (geometrically) reduced Cartier curve of the form $C' = C'_+ \amalg_E C'_-$,
where $E = \nu^*(D)$ such that $C'_{\pm}$ are   integral curves on 
$X$, none of which is contained in $D$ and each of which is smooth along $D$
(see Remark~\ref{remk:Surface-Bertini}). 
\end{enumerate}

Let $\iota_D = \iota_+ \circ \iota = \iota_- \circ \iota \colon  D \inj S_X$ denote 
the inclusion map. Then, it follows from (1) that 
\begin{equation}\label{eqn:map-diff-surfaces-0*}
(t_+)_{|D} = \iota^*_D(t) = (t_-)_{|D}.
\end{equation}

Let $(f'_+, f'_-)$ be the image of $f'$ in $\sO^{\times}_{C'_+, E} \times 
\sO^{\times}_{C'_-,E} \inj k(C'_+) \times k(C'_-)$.
It follows from \lemref{lem:Cycle-mod-1} that there is an exact sequence
\[
0 \to \sO_{C',E} \to \sO_{C'_+,E} \times \sO_{C'_-, E} \to \sO_E \to 0.
\]
In particular, we have 
\begin{equation}\label{eqn:map-diff-surfaces-0} 
(f'_+)_{|E} = (f'_-)_{|E} \in \sO^{\times}_E .
\end{equation}

Let us first assume that $C'_+ = C'_-$ as curves on $X$.
Let $C$ denote this curve and let $C^N$ denote its normalization. 
Let $\pi\colon C^N \to C \inj X$ denote the composite map.
Since $C$ is regular along $E$ by (2), we get 
$f'_+, f'_- \in \sO^{\times}_{C^N,E}$.
Setting $g := f'_+ f'^{-1}_- \in \sO^{\times}_{C^N,E}$, it
follows from ~\eqref{eqn:map-diff-surfaces-0} that
$g \in \Ker(\sO^{\times}_{C^N,E}  \to \sO^{\times}_E)$. 
Moreover, $\tau^*_X(\divf(f')) = \iota^*_+(\divf(f')) -
\iota^*_-(\divf(f')) = \divf(f'_+) - \divf(f_-) = \pi_*(\divf(g))$. 
We conclude from ~\eqref{eqn:DefChowMod-0} that
$\tau^*_X$ descends to a map $\tau^*_X\colon  \CH^{LW}_0(S_X) \to \CH_0(X|D)$.

We now assume that $C'_+ \neq C'_-$. Let $S$ denote the set of closed points
on $C'_+ \cup C'_-$, where $f'_+$ or $f'_-$ has a pole.
It is clear that $S \cap D = \emptyset$.
We now repeat the constructions in the proof of \lemref{lem:reduction-basic}
to find a very ample line bundle $\cL$ on $X$ and a section 
$s_{\infty} \in H^0(X, \cL)$ such that 
\begin{listabc}
\item $(s_{\infty})$ is  integral (because $X$ is smooth and  connected).
\item $(s_\infty) \cap (C'_+ \cup C'_-) \cap D = \emptyset$.
\item $(s_\infty) \supset S$.
\item $(s_{\infty}) \not\subset C'_+ \cup C'_-$.
\item $(s_{\infty})$ is smooth away from $S$. 
\item
$X \setminus (s_\infty)$ is affine.
\end{listabc}

It follows that $f'_{\pm}$ extend to regular functions $F'_{\pm}$ on
$U = X \setminus (s_\infty)$. Since $F'_{\pm}$ are meromorphic functions
on $X$ which have poles only along $(s_\infty)$, it follows that 
$F'_{\pm}s^N_{\infty}$ are elements of $H^0(U, \sL^N)$ which extend to
sections ${s_{0,{\pm}}}$ of $\sL^N$ on all of $X$, if we choose $N \gg 0$.

Since $s_\infty, \ F'_+$ and $F'_-$ are all invertible on $C'_+ \cup C'_-$ along 
$D$, we see that $(s_0)_{\pm}$ are invertible on $C'_+ \cup C'_-$ along $D$.
In particular, ${s_{0,{\pm}}} \notin H^0(X, \cL^{\tensor N} \tensor\sI_{C_+ \cup C_-})$. 
As before,
we can moreover find $\alpha_{\pm} \in H^0(X, \cL^{\tensor N} \tensor\sI_{C_+ \cup C_-}) \subset H^0(X, \cL^N)$
such that ${s'_{0,{\pm}}}:= {s_{0,{\pm}}} + \alpha_\pm$ have the following 
properties.
\begin{listabcprime}
\item $({s'_{0,{\pm}}})$ are integral.
\item
$({s'_{0,{\pm}}}) \not\subset C'_+ \cup C'_-$.  
\item $({s'_{0,{\pm}}}) \cap (C'_+ \cup C'_-) \cap D  = \emptyset$.
\item
$({s'_{0,{\pm}}})$ are smooth away from $C'_+ \cup C'_-$.
\end{listabcprime}

We then have 
\[
\frac{s'_{0,{\pm}}}{s^N_{\infty}} = 
\frac{F'_{\pm}s^N_{\infty} + (\alpha_{\pm} s^{-N}_{\infty})
s^N_{\infty}}{s^N_{\infty}} = F'_{\pm} + \alpha_{\pm} s^{-N}_{\infty} = H'_{\pm}, \ 
(\mbox{say}).
\]

Since $\alpha_{\pm}$ vanish along $C'_+ \cup C'_-$ and $s_{\infty}$ does not vanish identically along $U\cap (C'_+ \cup C'_-)$, it follows that $(H'_{\pm})_{|{(C'_+ \cup C'_-) \cap U}} =
(F'_{\pm})_{|{(C'_+ \cup C'_-)\cap U}}$. 
Since $(F'_{\pm})_{| C'_+ \cup C'_-}$ are invertible regular functions on 
$C'_+ \cup C'_-$ along $D$, it follows that $(H'_+)_{|{C'_+ \cup C'_-}}$ and 
$(H'_-)_{|{C'_+ \cup C'_-}}$ are both rational functions on $C'_+ \cup C'_-$
which are regular and invertible along $D$. 
In particular, we have
\begin{equation}\label{eqn:map-diff-surfaces-1}
(H'_+)_{|E} = (F'_+)_{|E} = (f'_+)_{|E} \ {=}^{\dagger} \ (f'_-)_{|E} 
= (F'_-)_{|E} = (H'_-)_{|E},
\end{equation}
where ${=}^{\dagger}$ follows from ~\eqref{eqn:map-diff-surfaces-0}.

Since $s'_{0,+}$ and $s'_{0,-}$ are both invertible functions on $C'_-$ in a 
neighborhood of $C'_- \cap D$ by c'), 
it follows that the restriction of $s'_{0,-}/s'_{0,+}$ on $C'_-$ is a 
rational function on $C_-$, which is regular and invertible in a neighborhood 
of $C'_- \cap D$.
On the other hand, we have 
\begin{equation}\label{eqn:map-diff-surfaces-2}
\frac{s'_{0,-}}{s'_{0,+}} = \frac{s'_{0,-}/{s^N_\infty}}
{s'_{0,+}/{s^N_\infty}} = \frac{H'_-}{H'_+}, 
\end{equation}
as rational functions on $X$.
Since $H'_+$ and $H'_-$ are also well-defined non-zero rational functions on
$C'_-$, we conclude that $s'_{0,-}/s'_{0,+}$ and ${H'_-}/{H'_+}$ 
both restrict to an identical and well-defined rational function on $C'_-$,
which is invertible along $D$.

We now compute
\[
\begin{array}{lll}
\tau^*_X(\divf(f')) & = & \iota^*_+(\divf(f')) -
\iota^*_-(\divf(f')) \\
& = & 
\divf(f'_+) - \divf(f'_-) \\
& = & 
\left[(s'_{0,+}) \cdot C'_+ - (s^N_{\infty}) \cdot C'_+\right]
- \left[(s'_{0,-}) \cdot C'_- - (s^N_{\infty}) \cdot C'_-\right] \\
& = & \left[(s'_{0,+}) \cdot C'_+ - (s'_{0,+}) \cdot C'_-\right]
-\left[(s'_{0,-}) \cdot C'_- - (s'_{0,+}) \cdot C'_-\right] \\
& & - \left[(s^N_{\infty}) \cdot C'_+ - (s^N_{\infty}) \cdot C'_-\right] \\
& = & \left[(s'_{0,+}) \cdot (C'_+ - C'_-)\right] -
\left[C'_- \cdot (s'_{0,-}) - (s'_{0,+}))\right] -
\left[(s^N_{\infty}) \cdot (C'_+ - C'_-)\right] \\
& = & (s'_{0,+}) \cdot (\divf({t_+}/{t_-})) - C'_- \cdot 
(\divf({s'_{0,-}}/{s'_{0,+}})) - N(s_{\infty}) \cdot (\divf({t_+}/{t_-})) \\
& = &  (s'_{0,+}) \cdot (\divf({t_+}/{t_-})) - C'_- \cdot 
(\divf({H'_-}/{H'_+})) - N(s_{\infty}) \cdot (\divf({t_+}/{t_-})).
\end{array}
\]

It follows from b) and c') that $t_{\pm}$ restrict to regular invertible
functions on $(s'_{0,+})$ and $(s_{\infty})$ along $D$.
We set $h_1 = {(\frac{t_+}{t_-})}_{| (s'_{0,+})}, \ 
h_2 = {(\frac{H'_-}{H'_+})}_{| C'_-}$ and
$h_3 =  {(\frac{t_+}{t_-})}_{| s_{\infty}}$.
Let $(s'_{0,+})^N \to (s'_{0,+})$, $(C'_-)^N \to C'_-$ and
$(s_\infty)^N \to (s_\infty)$ denote the normalization maps.
Let $\nu_1\colon  (s'_{0,+})^N \to X$, $\nu_2\colon  (C'_-)^N  \to X$ and
$\nu_3\colon  (s_\infty)^N \to X$ denote the composite maps.
Since $(s'_{0,+}), \  C'_-$ and $(s_\infty)$ are all regular along
$D$ by (2), e) and d'), it follows from
~\eqref{eqn:map-diff-surfaces-0*} and ~\eqref{eqn:map-diff-surfaces-1}
that $h_1 \in G((s'_{0,+}), \nu^*_1(D))$,
$h_2 \in G((C'_-)^N, \nu^*_2(D))$ and $h_3 \in G((s_\infty)^N , \nu^*_3(D))$.
We conclude from ~\eqref{eqn:DefChowMod-0} that $\tau^*_X(\divf(f'))$
dies in $\CH_0(X|D)$. In particular, $\tau^*_X$ descends to a
map $\tau^*_X\colon \CH^{LW}_0(S_X) \to \CH_0(X|D)$. This finishes the proof.
\end{proof}

\subsection{The map $\tau^*_X$: the case of higher dimensions}
\label{sec:dim>2}
We are left to show that the map $\tau_X^*\colon \sZ_0(S_X,D) \to 
\sZ_0(X|D)$ of ~\eqref{eqn:tau^*-0} descends to a group homomorphism 
$\CH^{LW}_0(S_X) \to \CH_0(X|D)$ when $X$ has dimension at least 3 when $k$ is infinite and perfect. This is the content of the following.

\begin{prop}\label{prop:map-diff-higher} 
Let $X$ be a smooth   connected quasi-projective scheme over $k$ of dimension $d \ge 3$
with an effective Cartier divisor $D$. Then the map 
$\tau_X^*\colon \sZ_0(S_X, D) \to \sZ_0(X|D)$ of ~\eqref{eqn:tau^*-0} 
descends to a group homomorphism $\CH^{LW}_0(S_X)\to \CH_0(X|D)$.
\end{prop}
\begin{proof}
Let $\nu\colon  C \inj S_X$ be a reduced Cartier curve relative to $D$ and let 
$f \in \sO^{\times}_{C, E}$, where $E = \nu^*(D)$.
By \lemref{lem:good-Cartier-curves-dim>2}, we can assume that 
there are inclusions $C \inj Y \inj S_X$ satisfying the conditions
(1) - (6) of \lemref{lem:good-Cartier-curves-dim>2}.
The only price we pay by doing so is that $C$ may no longer be reduced
(but still Cohen-Macaulay).
But we solve this by replacing $C$ with a reduced Cartier curve
(which we also denote by $C$) that is of the form given in
\lemref{lem:reduction-basic}. We shall now continue with the notations of
proof of \lemref{lem:reduction-basic}.

We write $C = (t) \cap Y$, where $t \in H^0(S_X, \sL)$ such that
$\sL$ is a very ample line bundle on $S_X$.
Let $t_{\pm} = \iota^*_{\pm}(t) \in H^0(X, \iota^*_{\pm}(\sL))$ and let
$C_{\pm} = (t_{\pm}) \cap Y = (t_{\pm}) \cap Y_{\pm}$.
Let $\nu_{\pm}\colon  C_\pm \inj X$ denote the inclusions.
It follows from our choice of the section that $(t_{\pm})$ are  
integral.
If $C_+ = C_-$, exactly the same argument as in the
case of surfaces in \propref{prop:map-diff-surfaces} applies to show that 
$\tau^*_X(\divf(f)) \in \sR_0(X|D)$.
So we assume $C_+ \neq C_-$.

Let $\Delta(C) = C_+ \cup C_-$ denote the scheme-theoretic image
in $X$ under the finite map $\Delta$.
Since $X$ is smooth and   connected, the Bertini Theorem of Altman and Kleiman \cite[Theorem 1]{KL} allows us to find once again a complete intersection
  integral surface $T \subset X$ satisfying the following.
\begin{enumerate}
\item
$T\supset \Delta(C)$.
\item 
$T\cap (t_{\pm})$ are integral curves.
\item
$T$ is smooth away from $\Delta(C)$.
\end{enumerate}

Set $t^T_\pm = (t_\pm)_{|T}$.
Since $C_\pm$ are integral and contained in $T \cap (t_{\pm})$, it follows that
\begin{equation}\label{eqn:map-diff>2-0}
(t^T_{\pm}) = C_\pm.
\end{equation}

Let $S$ be the finite set of points of $\Delta(C)$, where 
$f_+=\nu^*_+(f)$ or $f_-=\nu^*_-(f)$ have poles. It is clear that
$S \cap D = \emptyset$.
We now choose another very ample line bundle $\sM$ on $X$ and 
$s_\infty \in H^0(X, \sM)$ (see the proof of \lemref{lem:reduction-basic})
such that
\begin{romanlist}
\item $(s_\infty)$ is   integral.
\item The intersections $(s_\infty)\cap T$ and $(s_\infty) \cap (t_\pm)$ are proper and   integral.
\item $(s_\infty) \supset S$.
\item $(s_\infty) \cap \Delta(C)  \cap D = \emptyset$.
\item $X \setminus (s_\infty)$ is affine.
\item $(s_\infty)$ is smooth away from $S$.
\item
$(s_\infty) \cap T$ is smooth away from $\Delta(C)$.
\item
$(s_\infty)\cap T \not\subset \Delta(C)$.
\end{romanlist}

As shown in the proof of \lemref{lem:reduction-basic}, it
follows from (3), iv) and vii) 
above that $(s^T_\infty) := (s_\infty) \cap T$ is smooth along $D$.
Using v), we can lift $f_\pm \in k(C_\pm)^\times$ to regular 
functions $F_\pm$ on $U = X\setminus (s_\infty)$.
Using an argument identical to that given in the proof of 
\propref{prop:map-diff-surfaces}, we can extend the sections $s_{0,\pm} = s_\infty^N F_\pm$ 
(for some $N \gg 0$) 
to global sections $s_{0,\pm}$ of $\mathcal{M}^{N}$ on $X$ so that their zero loci satisfy:
\begin{listabc}
\item
$(s_{0,\pm})$ and $(s_{0,\pm}) \cap T$ are   integral.
\item
$(s_{0,\pm}) \cap T \cap \Delta(C) \cap D = \emptyset$. 
\item
$(s_{0,\pm}) \cap T \not\subset \Delta(C)$.
\item
$(s_{0,\pm}) \cap T$ are smooth away from $\Delta(C)$.
\end{listabc}

As we argued in the proof of \lemref{lem:reduction-basic}, it follows from
iv), vii), c) and d) that $(s^T_\infty)$ and $(s_{0,\pm}^T) := 
(s_{0,\pm}) \cap T$ are smooth along $D$.

Setting $H_\pm = {s_{0,\pm}}/{s^N_\infty}$ and using the argument of the 
proof of \propref{prop:map-diff-surfaces}, we get 
$H_\pm \in k(X)^{\times}$ and they restrict to rational functions on
$C_+$ as well as $C_-$ which are regular and invertible along $D$.
Moreover, we have
\begin{equation}\label{eqn:map-diff>2-1}
(H_+)_{|E} = (F_+)_{|E} = (f_+)_{|E} \ {=}^{\dagger} \ (f_-)_{|E} 
= (F_-)_{|E} = (H_-)_{|E},
\end{equation}
where ${=}^{\dagger}$ follows from our assumption that
$f \in \sO^{\times}_{C, E}$.
It follows that the restriction of the rational function
${s_{0,-}}/{s_{0,+}} = {H_-}/{H_+}$ to $C_-$ is an 
element of $\sO^{\times}_{C_-, C_- \cap D}$ such that ${({H_-}/{H_+})}_{|E} = 1$.

We now compute
\[
\begin{array}{lll}
\tau^*_X(\divf(f)) & = & \iota^*_+(\divf(f)) -
\iota^*_-(\divf(f)) \\
& = & 
\divf(f_+) - \divf(f_-) \\
& = & 
\left[(s^T_{0,+}) \cdot C_+ - N (s^T_{\infty}) \cdot C_+\right]
- \left[(s^T_{0,-}) \cdot C_- - N (s^T_{\infty}) \cdot C_-\right] \\
& = & \left[(s^T_{0,+}) \cdot C_+ - (s^T_{0,+}) \cdot C_-\right]
-\left[(s^T_{0,-}) \cdot C_- - (s^T_{0,+}) \cdot C_-\right] \\
& & - N \left[(s^T_{\infty}) \cdot C_+ - (s^T_{\infty}) \cdot C_-\right] \\
& = & \left[((s^T_0)_+) \cdot (C_+ - C_-)\right] -
\left[C_- \cdot ((s^T_{0,-}) - (s^T_{0,+}))\right] -
N \left[(s^T_{\infty}) \cdot (C_+ - C_-)\right] \\
& {=}^{\dagger} & \left[(s^T_{0,+}) \cdot ((t^T_+) - (t^T_-))\right] -
\left[C_- \cdot ((s^T_{0,-}) - (s^T_{0,+}))\right] -
N \left[(s^T_{\infty}) \cdot ((t^T_+) - (t^T_-))\right] \\
& = & (s^T_{0,+}) \cdot (\divf({t^T_+}/{t^T_-})) - C_- \cdot 
(\divf({s_{0,-}}/{s_{0,+}})) - N (s^T_{\infty}) \cdot (\divf({t^T_+}/{t^T_-})) \\
& = &  (s^T_{0,+}) \cdot (\divf({t^T_+}/{t^T_-})) - C_- \cdot 
(\divf({H_-}/{H_+})) - N (s^T_{\infty}) \cdot (\divf({t^T_+}/{t^T_-})),
\end{array}
\]
where ${=}^{\dagger}$ follows from ~\eqref{eqn:map-diff>2-0}.

It follows from iv) and b) that ${t^T_+}/{t^T_-}$ restricts to  regular and
invertible functions on $(s^T_{0,+})$ and $(s^T_{\infty})$ along $D$.
Since $t \in H^0(S_X, \sL)$ and $t_{\pm} = \iota^*_{\pm}(t) \in 
H^0(X, \iota^*_{\pm}(\sL))$, it follows that  
$(t_+)_{|D} = \iota^*_D(t) = (t_-)_{|D}$, where
$\iota_D = \iota_+ \circ \iota = \iota_- \circ \iota \colon  D \inj S_X$ denotes
the inclusion map. In particular, ${({t^T_+}/{t^T_-})}_{|E} = 1$.
We have seen before that ${(\frac{H_-}{H_+})}_{| C_-}$ is a
regular and invertible function on $C_-$ along $D$ and
${(\frac{H_-}{H_+})}_{| E} = 1$.

We set $h_1 = {(\frac{t^T_+}{t^T_-})}_{| (s^T_{0,+})}, \ 
h_2 = {(\frac{H_-}{H_+})}_{| C_-}$ and
$h_3 =  {(\frac{t^T_+}{t^T_-})}_{| s^T_{\infty}}$. We conclude now using exactly the same argument as in the proof of Proposition \propref{prop:map-diff-surfaces}.
Let $(s^T_{0,+})^N \to (s^T_{0,+})$, $(C_-)^N \to C_-$ and
$(s^T_\infty)^N \to (s^T_\infty)$ denote the normalization maps.
Let $\nu_1\colon  (s^T_{0,+})^N \to X$, $\nu_2\colon  (C_-)^N  \to X$ and
$\nu_3\colon  (s^T_\infty)^N \to X$ denote the composite maps.
The curves $(s^T_{0,+})$ and $(s^T_\infty)$ are all smooth along
$D$, and $C_-$ is smooth along $D$ by \lemref{lem:reduction-basic}.
It follows that $h_1 \in G((s^T_{0,+})^N, \nu^*_1(D))$,
$h_2 \in G((C_-)^N, \nu^*_2(D))$ and $h_3 \in G((s^T_\infty)^N , \nu^*_3(D))$.
We conclude from ~\eqref{eqn:DefChowMod-0} that $\tau^*_X(\divf(f))$
dies in $\CH_0(X|D)$. In particular, $\tau^*_X$ descends to a
map $\tau^*_X\colon  \CH^{LW}_0(S_X) \to \CH_0(X|D)$. This finishes the proof.
\end{proof}

\subsection{The maps $p_{\pm,*}$}\label{sec:PF-mod}
Let $X$ and $D$ and $k$ be as above. Now that we have constructed the map $\tau_X^*$, we build two maps in the opposite direction.
\[
p_{\pm,*} \colon   \sZ_0(X|D) 
\rightrightarrows \sZ_0(S_X,D)
\] 
by $p_{+,*}([x]) = \iota_{+, *}([x])$ 
(resp.~ by $p_{-,*}([x]) = \iota_{-, *}([x])$) for a closed point 
$x \in X \setminus D$. 
Concretely, the two maps $p_{+,*}$ and $p_{-,*}$ copy a cycle $\alpha$ in one 
of the two components of the double $S_X$ (the $X_+$ or the $X_-$ copy). 
Since $\alpha$ is supported outside $D$ (by definition of $\sR_0(X|D)$), 
the cycles $p_{+,*}(\alpha)$ and  
$p_{-,*}(\alpha)$ give classes in $\CH_0(S_X)$.

\begin{prop}\label{prop:PF-mod-double}
The maps $p_{\pm,*}$ descend to group homomorphisms 
$p_{\pm,*}\colon \CH_0(X|D) \to \CH_0(S_X)$.
\end{prop}
\begin{proof}
Let $\nu\colon {C} \to X$ be a finite map from a normal integral curve 
such that $\nu({C}) \not\subset D$ and let $E=\nu^*(D)$. 
Since both $X$ and ${C}$ are smooth, the map $\nu$ is automatically a local 
complete intersection. Let $f\in k({C})^\times$  be a rational function on 
${C}$ such that $f\in G({C}, E)$ 
(in the notations of \ref{def:DefChowMod1}). 

Since $C$ is smooth, it follows from Proposition \ref{prop:double-prp}
that $S_C := S(C,E)$ is reduced and is smooth away from $E$.
If follows from Proposition \ref{prop:double-prp-fine} that the double map 
$S_C \to S_X$ is l.c.i. along $D$.
In particular, the pair $(S_C, E)$ is a good curve relative to 
$(S_X, D)$. We now consider the rational function $h = (h_+, h_-) := 
(f,1)$ on $S_C$. 
The modulus condition satisfied by $f$ on $C$ guarantees that $h$ is regular 
and invertible along  $E\subset S_C$. It is also easy to see that the divisor 
of $h$ trivializes $p_{+,*}(\nu_*{\rm div}(f))$. 
The argument for $p_{-,*}$ is symmetric.
\end{proof}

To summarize, we have shown the following.
\begin{thm}\label{thm:Main-PB-PF-mod}
Let $X$ be a smooth   connected quasi-projective scheme of dimension $d \ge 1$ over
an infinite perfect field $k$ and let $D\subset X$ be an effective Cartier 
divisor. Then there are maps
\begin{equation}\label{eqn:Main-PB-PF-mod-0}
\tau_X^*\colon \CH_0(S_X) \to \CH_0(X|D) \ \ \ \mbox{and} \ \ \
p_{\pm,*}\colon \CH_0(X|D) \to \CH_0(S_X).
\end{equation}
such that $\tau_X^*(\alpha) = \iota_+^*(\alpha) - \iota_-^*(\alpha)$
and $p_{\pm,*}(\beta) = \iota_{\pm, *}(\beta)$
on cycles.
\begin{proof}This is a combination of Lemma \ref{lem:factor-PN-projection} and of Propositions \ref{prop:map-diff-surfaces},  \ref{prop:map-diff-higher} and \ref{prop:PF-mod-double}.
    \end{proof}
\end{thm}

\section{Reduction to infinite base 
field}\label{sec:Fin-Infinite}
In the previous section, our results were based on the assumption that the
ground field $k$ is infinite. In order to prove our main theorem for
finite fields, we shall use the following descent tricks
for cycles on singular varieties and cycles with modulus.

\begin{prop}\label{prop:PF-fields} 
Let $k \inj k'$ be separable algebraic (possibly infinite)
extension of fields. Let
$X$ be a reduced quasi-projective scheme 
over $k$ and let $X' = X_{k'}:= X \otimes_k k'$.
Let ${\rm pr}_{{k'}/{k}}: X' \to X$ be the
projection map. Then the following hold.
\begin{enumerate}
\item
There exist pull-back maps
${\rm pr}^*_{{k'}/{k}}: \CH^{LW}_0(X) \to \CH^{LW}_0(X')$ and
${\rm pr}^*_{{k'}/{k}}: \CH_0(X) \to \CH_0(X')$ which commute with
the canonical map $\CH^{LW}_0(X) \to \CH_0(X)$.
\item
If there exists a sequence of separable  
field extensions $k = k_0 \subset k_1 \subset
\cdots \subset k'$ with $k' = \cup_i k_i$ such that $X_i := X_{k_i}$
for each $i \ge 1$, then
we have ${\underset{i}\varinjlim} \ \CH_0(X_i) \xrightarrow{\simeq} 
\CH_0(X')$. The same holds for $\CH^{LW}_0(-)$ as well.
\item
If $k \inj k'$ is finite, then there exists a push-forward
$({\rm pr}_{{k'}/{k}})_*: \CH_0(X') \to \CH_0(X)$ such that
$({\rm pr}_{{k'}/{k}})_* \circ {\rm pr}^*_{{k'}/{k}}$ is multiplication by
$[k': k]$.
\end{enumerate}
\end{prop}
\begin{proof}
The proofs of (1) and (2) for $\CH_0(-)$ and $\CH^{LW}_0(-)$
are identical. So we consider only $\CH_0(-)$ in the proof below.

We first note that as $k \inj k'$ or $k \inj k_i$ is a
separable algebraic extension, the scheme $X'$ is reduced and
$X'_{\rm sing} = X_{\rm sing} \times_k k'$. The same holds for each $X_i$ as
well.

Let $x \in X \setminus X_{\rm sing}$ be a closed point. Since
${\rm pr}_{{k'}/{k}}$ is flat, it follows from our hypothesis that 
${\rm pr}^*_{{k'}/{k}}([x])$ is a well defined 0-cycle in
$\sZ_0(X', X'_{\rm sing})$. We thus have a pull-back map
${\rm pr}^*_{{k'}/{k}}: \sZ_0(X, X_{\rm sing}) \to \sZ_0(X', X'_{\rm sing})$.

We show that this map preserves rational equivalences.
So let $\nu: (C, Z) \to (X, X_{\rm sing})$ be a good curve and let
$f: C \dashrightarrow \P^1_k$ be a rational function which is regular
in an open neighborhood $Z \subsetneq U \subseteq C$. 
The base change via
$k \inj k'$ gives a diagram of Cartesian squares 
\begin{equation}\label{eqn:PF-fields-0}
\xymatrix@C1pc{
Z_{k'} \ar[d] \ar@{^{(}->}[r] & U_{k'} \ar@{^{(}->}[r] \ar[d] &
C_{k'} \ar[r]^{\nu_{k'}} \ar[d] & X' 
\ar[d]^{{\rm pr}_{{k'}/{k}}} \\
Z \ar@{^{(}->}[r] & U \ar@{^{(}->}[r] & C \ar[r]_{\nu} & X.}
\end{equation} 

Since $k \inj k'$ is a separable field extension, we see that
$C_{k'}$ is reduced and $C_{k'} \setminus Z_{k'} = (C \setminus Z)_{k'}$
is regular. We also have
$\nu^{-1}_{k'}(X'_{\rm sing}) = \nu^{-1}_{k'}(X_{\rm sing} \times_k k')
\subseteq Z_{k'}$. 
Moreover, the flatness of ${\rm pr}_{{k'}/{k}}$ ensures that
the map $\nu_{k'}: C_{k'} \to X'$ is l.c.i. over $X'_{\rm sing}$.
It follows that $\nu_{k'}: (C_{k'}, Z_{k'}) \to (X', X'_{\rm sing})$ is
a good curve. 
Since $f: C \dashrightarrow \P^1_k$ is regular and invertible along $U$,
it follows that $f':= f_{k'}: C_{k'} \dashrightarrow \P^1_{k'}$ is a rational
function which is regular and invertible along $U_{k'}$.

Let $D \subset C \times_k \P^1_k$ denote the closure of the graph of $f$
with the reduced structure. Note that this graph is already reduced over
the regular locus of $f$. Let $p: C \times_k \P^1_k \to \P^1_k$ 
and $q:  C \times_k \P^1_k \to C$ denote
the projection maps so that $p|_U = f|_U$. We thus have a commutative diagram
\begin{equation}\label{eqn:PF-fields-1}
\xymatrix@C1pc{
\P^1_{k'} \ar[d] & D_{k'}  \ar[r]^{q_{k'}} \ar[d] \ar[l]_{p_{k'}} &
C_{k'} \ar[r]^{\nu_{k'}} \ar[d] & X' \ar[d]^{{\rm pr}_{{k'}/{k}}} \\
\P^1_k & D \ar[r]_{q} \ar[l]^{p} & C \ar[r]_{\nu} & X.}
\end{equation} 

We now have
\[
\begin{array}{lll}
\nu_{k'  *}(\divf(f')) & = & \nu_{k' *} \circ 
q_{k'  *}([p^*_{k'}(0)] - [p^{*}_{k'}(\infty)]) \\
& = &   \nu_{k' *} \circ q_{k' *}([p^{*}_{k'} \circ 
{\rm pr}^*_{{k'}/{k}}(0) -
p^*_{k'} \circ {\rm pr}^*_{{k'}/{k}}(\infty)]) \\
& = &   \nu_{k' \ *} \circ q_{k' *}([{\rm pr}^*_{{k'}/{k}} \circ p^*(0) -
{\rm pr}^*_{{k'}/{k}} \circ p^*(\infty)]) \\
& {=}^{\dagger} & 
\nu_{k' *} \circ {\rm pr}^*_{{k'}/{k}} \circ q_*([p^*(0) - p^*(\infty)]) \\
& = & \nu_{k' \ *} \circ {\rm pr}^*_{{k'}/{k}}(\divf(f)) \\
& {=}^{\dagger \dagger} & {\rm pr}^*_{{k'}/{k}} \circ \nu_*((\divf(f)),
\end{array}
\]
where ${=}^{\dagger}$ and ${=}^{\dagger \dagger}$ follow from
\cite[Proposition~1.7]{Fulton} because all squares in ~\eqref{eqn:PF-fields-1}
are Cartesian, vertical maps are all flat and $q$ as well as $\nu$ is finite.
We conclude that ${\rm pr}^*_{{k'}/{k}} \circ \nu_*((\divf(f)) \in 
\sR_0(X', X'_{\rm sing})$ and this proves (1).

We now consider (2).
It is clear that the map ${\underset{i}\varinjlim} \ \CH_0(X_i) \to
\CH_0(X')$ is surjective. To show injectivity, suppose there is some $i \ge 0$
and $\alpha \in \sZ_0(X_i, (X_i)_{\rm sing})$ such that
${\rm pr}^*_{{k'}/{k_i}}(\alpha) \in \sR_0(X', X'_{\rm sing})$.
We can replace $k$ by $k_i$ and assume $i = 0$.

Let $\nu^j: (C^j, Z^j) \to (X', X'_{\rm sing})$ be good curves
and let $f^j: C^j \dashrightarrow \P^1_{k'}$ be a rational function
which is regular and invertible in a neighborhood $U^j$ of $Z^j$ for
$j =1, \cdots , r$ such that ${\rm pr}^*_{{k'}/{k}}(\alpha) = 
\stackrel{r}{\underset{j =1}\sum} \nu^j_*(\divf(f^j))$.

Since each $C^j$ has a factorization $C^j \inj \P^{N_j}_{X'} \to X'$,
we can find some $i \gg 0$ and curves $W^j$ over $k_i$, 
a closed subscheme $T^j \subsetneq W^j$, an open neighborhood
$V^j \subseteq W^j$ of $T^j$ and invertible regular function 
$g^j:V^j \to \P^1_{k_i}$ such that $(C^j, Z^j) \simeq (W^j, T^j)_{k'},
\ U^j = V^j_{k'}$ and $f^j = g^j_{k'}$. Furthermore, we can find a 
finite map $\delta^j:(W^j, T^j) \to (X_{k_i}, (X_{k_i})_{\rm sing})$
such that $(C^j, Z^j) \simeq (W^j, T^j)_{k'}$ and $\nu^j = \delta^j_{k'}$
for each $j = 1, \cdots , r$. Since $k_i \inj k'$ is separable,
it also follows that $W^j \setminus T^j$ is regular. 

Since the map $X' \to X_i$ is faithfully flat,
it follows from \lemref{lem:lci-5} and our hypothesis on $(X_i)_{\rm sing}$
that each $W^j \to
X_i$ is l.c.i. along $(X_i)_{\rm sing}$ and $(\delta^j)^{-1}((X_i)_{\rm sing})
\subseteq T^j$.
It follows that each $(W^j, T^j)$ is a good curve relative to 
$(X_{k_i}, (X_{k_i})_{\rm sing})$.
Moreover, we have shown in the proof of (1) (with $k$ replaced by $k_i$)
in this situation that 
$\nu^j_*(\divf(f^j)) = {\rm pr}^*_{{k'}/{k_i}}(\delta^j_*(\divf(g^j)))$. 

We now set $\alpha_i = {\rm pr}^*_{{k_i}/{k}}(\alpha)$
and let $\beta = \alpha_i - \stackrel{r}{\underset{j =1}\sum}
\delta^j_*(\divf(g^j)) \in \sZ_0(X_i, (X_i)_{\rm sing})$.
It then follows that 
${\rm pr}^*_{{k'}/{k_i}}(\beta) = {\rm pr}^*_{{k'}/{k_i}}(\alpha_i)
- \stackrel{r}{\underset{j =1}\sum} 
\nu^j_*(\divf(f^j)) 
= {\rm pr}^*_{{k'}/{k}}(\alpha) - \stackrel{r}{\underset{j =1}\sum} 
\nu^j_*(\divf(f^j)) = 0$ in
$\sZ_0(X', X'_{\rm sing})$.
Since the map ${\rm pr}^*_{{k'}/{k_i}}: \sZ_0(X_i, (X_i)_{\rm sing}) \to
\sZ_0(X', X'_{\rm sing})$ of free abelian groups is clearly injective,
we get $\beta = 0$, which means that $\alpha_i \in
\sR_0(X_i, (X_i)_{\rm sing})$. This proves (2).

If $k \inj k'$ is a finite extension, then $\Spec(k') \to \Spec(k)$ is
an l.c.i. morphism. 
%This follows because of the factorization
%$\Spec(k') \inj \P^n_k \to \Spec(k)$ so that $\Spec(k')$ is a smooth closed
%point of $\P^n_k$. 
In particular, it follows from \lemref{lem:lci-5} that
${\rm pr}_{{k'}/{k}}: X' \to X$ is finite, flat l.c.i. morphism.
The push-forward map $({\rm pr}_{{k'}/{k}})_*: \CH_0(X') \to \CH_0(X)$ now
follows from our hypothesis on the singular locus of $X'$ and
\propref{prop:0-cycle-PF}. The property
$({\rm pr}_{{k'}/{k}})_* \circ {\rm pr}^*_{{k'}/{k}} = [k': k]$ is obvious.
This finishes the proof.
\end{proof}

\begin{prop}\label{prop:PF-fields-mod} 
Let $k \inj k'$ be a separable algebraic (possibly infinite)
extension of fields. Let
$X$ be a smooth quasi-projective scheme 
over $k$ with an effective Cartier divisor $D$.
Let $X' = X_{k'}$ and $D' = D_{k'}$ denote the base change of $X$ and
$D$, respectively. Let ${\rm pr}_{{k'}/{k}}: X' \to X$ be the
projection map. Then the following hold.
\begin{enumerate}
\item
There exists a pull-back 
${\rm pr}^*_{{k'}/{k}}: \CH_0(X|D) \to \CH_0(X'|D')$.
\item
If there exists a sequence of separable 
field extensions $k = k_0 \subset k_1 \subset
\cdots \subset k'$ with $k' = \cup_i k_i$, then
we have ${\underset{i}\varinjlim} \ \CH_0(X_{k_i}|D_{k_i}) \xrightarrow{\simeq} 
\CH_0(X'|D')$. 
\item
If $k \inj k'$ is finite, then there exists a push-forward
${\rm pr}_{{k'}/{k} \ *}: \CH_0(X'|D') \to \CH_0(X|D)$ such that
$({\rm pr}_{{k'}/{k}})_* \circ {\rm pr}^*_{{k'}/{k}}$ is multiplication by
$[k': k]$.
\end{enumerate}
\end{prop}
\begin{proof}
Let $x \in X \setminus D$ be a closed point. Since
${\rm pr}_{{k'}/{k}}$ is smooth, it follows from our hypothesis that 
${\rm pr}^*_{{k'}/{k}}([x])$ is a well defined 0-cycle in
$\sZ_0(X'|D')$. We thus have a pull-back map
${\rm pr}^*_{{k'}/{k}}: \sZ_0(X|D) \to \sZ_0(X'|D')$.
To show that this map preserves rational equivalence, we shall use a 
presentation of $\sR_0(X|D)$ which is different from the one given
in \S~\ref{def:DefChowMod1} (see \cite{BS} or \cite{KPv}).  

Let $C \inj X \times_k \P^1_k$ be an integral curve
satisfying the following properties.
\begin{enumerate}
\item
$C \cap (D \times_k \P^1_k)$ is finite.
\item
$C \cap (D\times_k \{0, \infty\}) = \emptyset$.
\item
The Weil divisor $\nu^*(X \times \{1\}) - \nu^*(D \times \P^1_k)$ is effective,
where $\nu: C^N \to C \inj X \times_k \P^1_k$ is the composite finite
map.
\end{enumerate}

The group of rational equivalences $\sR_0(X|D)$ coincides with
the subgroup of $\sZ_0(X|D)$ generated by $[C_{0}] - [C_\infty]$,
where $C$ runs over all curves as above.

Let $C \in X \times_k \P^1_k$ be any such curve.
It follows again from the smoothness of ${\rm pr}_{{k'}/{k}}$
that $C' = {\rm pr}^*_{{k'}/{k}}(C) = C_{k'} \inj 
(X \times_k \P^1_k)_{k'} = X' \times_{k'} \P^1_{k'}$
satisfies conditions (1)-(3) above. In particular, 
$[C'_0] - [C'_\infty]$ dies in $\CH_0(X'|D')$.
However, the flatness of ${\rm pr}_{{k'}/{k}}$ again shows that
$[C'_\infty] = {\rm pr}^*_{{k'}/{k}}([C_\infty])$ and
$[C'_0] = {\rm pr}^*_{{k'}/{k}}([C_0])$ so that
${\rm pr}^*_{{k'}/{k}}([C_{0}] - [C_\infty])$ dies in
$\CH_0(X'|D')$. This proves (1).

It is clear that the map ${\underset{i}\varinjlim} \ \CH_0(X_i|D_i) \to
\CH_0(X'|D')$ is surjective. 
To show injectivity, suppose there is some $i \ge 0$
and $\alpha \in \sZ_0(X_i|D_i)$ such that
${\rm pr}^*_{{k'}/{k_i}}(\alpha) \in \sR_0(X'|D')$.
We can replace $k$ by $k_i$ and assume $i = 0$.

Let $C^j \inj X' \times_{k'} \P^1_{k'} = (X \times_k \P^1_k)_{k'}$ for
$j = 1, \cdots , r$ be a collection of curves as in the proof of (1)
so that ${\rm pr}^*_{{k'}/{k}}(\alpha) =
(\stackrel{r}{\underset{j =1}\sum} [C^j_ {0}]) -
(\stackrel{r}{\underset{j =1}\sum} [C^j_ {\infty}])$.
Let $\nu^j: C^{j, N} \to X' \times_{k'} \P^1_{k'}$ denote the maps from
the normalizations of the above curves.

We can then find some $i \gg 0$ and integral curves
$W^j \inj X_i \times_{k_i} \P^1_{k_i}$
such that $C^j = W^j \times_{k_i} k'$ for each $j =1, \cdots , r$.
In particular, we have $C^j_0 = 
{\rm pr}^*_{{k'}/{k_i}}(W^j_0)$ and
$C^j_{\infty} = {\rm pr}^*_{{k'}/{k_i}}(W^j_\infty)$ for $j = 1, \cdots , r$.
Since ${\rm pr}_{{k'}/{k_i}}$ is smooth, it follows that 
$C^{j,N} = W^{j,N}_{k'}$ for each $j$.
Moreover, it follows from \cite[Lemma 2.2]{KPv} that condition (3) above
holds on each $W^{j,N}$.
It follows that each $W^j$ defines a rational equivalence for 
0-cycles with modulus $D_i$ on $X_i$.

We now set $\alpha_i = {\rm pr}^*_{{k_i}/{k}}(\alpha)$
and let $\beta = \alpha_i - (\stackrel{r}{\underset{j =1}\sum} 
([W^j_0] - [W^j_\infty])) \in \sZ_0(X_i|D_i)$.
It then follows that 
${\rm pr}^*_{{k'}/{k_i}}(\beta) = {\rm pr}^*_{{k'}/{k_i}}(\alpha_i)
- \stackrel{r}{\underset{j =1}\sum} 
([C^j_0] - [C^j_\infty]) 
= {\rm pr}^*_{{k'}/{k}}(\alpha) - \stackrel{r}{\underset{j =1}\sum} 
([C^j_0] - [C^j_\infty]) = 0$ in
$\sZ_0(X'|D')$.
Since the map ${\rm pr}^*_{{k'}/{k_i}}: \sZ_0(X_i|D_i) \to
\sZ_0(X'|D')$ of free abelian groups is clearly injective,
we get $\beta = 0$, which means that $\alpha_i \in
\sR_0(X_i|D_i)$. This proves (2). The existence of push-forward is already 
known as remarked above and the formula $({\rm pr}_{{k'}/{k}})_* \circ 
{\rm pr}^*_{{k'}/{k}} = [k': k]$ is obvious from the definitions.
%The argument of \propref{prop:PF-fields} works through mutatis mutandis and is left to the reader.
\end{proof}

\section{The main results on the Chow groups of 0-cycles}
\label{sec:Main-0-cycles}
In this section, we apply the technical results of the previous section
to prove our main theorem on the Chow groups of 0-cycles with modulus
and the Chow group of 0-cycles on singular varieties.
We shall also derive our first set of applications.
We shall derive 
a new presentation of the Chow group of 0-cycles with modulus and
prove our main comparison theorem for the two Chow groups of 0-cycles
for the double. 

\begin{thm}\label{thm:Main-PB-PF-gen}
Let $k$ be a field and let $X$ be a smooth quasi-projective scheme of 
dimension $d \ge 1$ over $k$ with an effective Cartier divisor
$D \subset X$.
% Assume that one of the following holds.
% \begin{enumerate}
% \item
% $d \le 2$.
% \item
% $k = \ov{k}$ and $X$ is affine.
% \item
% $k = \ov{k}$ with ${\rm char}(k)= 0$ and $X$ is projective.
% \end{enumerate}

Then, there are maps
\begin{equation}\label{eqn:Main-PB-PF-mod-0*}
\Delta^*\colon \CH_0(X) \to  \CH_0(S_X); \ \ \ \
\iota^*_{\pm}\colon \CH_0(S_X) \to \CH_0(X) \ \ \mbox{and}
\end{equation}
\[
p_{\pm,*}\colon \CH_0(X|D) \to \CH_0(S_X)\]
such that $\iota^*_{\pm} \circ \Delta^* = {\rm Id}$ on $\CH_0(X)$.

If $k$ is perfect, then there is also a map
\begin{equation}\label{eqn:Main-PB-PF-mod-1*}
\tau_X^*\colon \CH_0(S_X) \to \CH_0(X|D)
\end{equation}
such that $\tau^*_X \circ p_{\pm,*} = \pm \ {\rm Id}$ on $\CH_0(X|D)$.
Moreover, the sequences
\begin{equation}\label{eq:Main-exact-seuqence}
0 \to \CH_0(X|D) \xrightarrow{p_{+,*}} \CH_0(S_X) \xrightarrow{\iota_-^*} 
\CH_0(X) \to 0
\end{equation}
and
\begin{equation}\label{eq:Main-exact-seuqence-0}
0 \to \CH_0(X) \xrightarrow{\Delta^*} \CH_0(S_X) 
\xrightarrow{\tau^*_X} \CH_0(X|D) \to 0
\end{equation}
are split exact.
\end{thm}
\begin{proof}
% Under our assumption, the canonical map $\CH_0^{LW}(S_X) \to \CH_0(S_X)$
% is an isomorphism by ~\thmref{thm:0-cycle-affine-proj-char0}.
% Note here that the map $\CH_0^{LW}(S_X) \to \CH_0(S_X)$ is induced by the
% identity map at level of $\sZ_0(S_X, D)$.
%{\color{blue
We can clearly assume that $X$ is connected. All the maps (except ~\eqref{eqn:Main-PB-PF-mod-1*}) involved 
in the theorem are well defined thanks to the results of 
the previous sections.
On the level of cycles, we clearly have 
$\iota^*_{\pm} \circ \Delta^* = {\rm Id}$ on $\sZ_0(X,D)$ and
$\tau^*_X \circ p_{\pm,*} = \pm \ {\rm Id}$ on $\sZ_0(X|D)$.
We are only thus left with proving ~\eqref{eqn:Main-PB-PF-mod-1*} when 
$k$ is finite and the exactness of the two sequences in general.
%}

Note that the maps $\tau^*_X: \sZ_0(S_X, D) \to \sZ_0(X|D) \surj
\CH_0(X|D)$ are defined over any field and for any field extension
$k \inj k'$, the diagram
\begin{equation}\label{eqn:Com-sing-mod}
\xymatrix@C1pc{
\sZ_0(S_X, D) \ar[r]^{\iota^*_{\pm}} \ar[d]_{{\rm pr}^*_{{k'}/{k}}} &
\sZ_0(X|D) \ar[d]^{{\rm pr}^*_{{k'}/{k}}} \\
\sZ_0(S_{X'}, D') \ar[r]_{\iota^*_{\pm}} & \sZ_0(X'|D')}
\end{equation}
commutes, where $X' = X_{k'}$. In particular, we have $\tau^*_{X'} \circ {\rm pr}^*_{{k'}/{k}}
=  {\rm pr}^*_{{k'}/{k}} \circ \tau^*_X$.

We have to show that the composite map 
$\tau^*_X: \sZ_0(S_X, D) \to \CH_0(X|D)$ kills the
rational equivalences, assuming $k$ is finite.  
Let us therefore assume that $\nu: (C, Z) \to (S_X, D)$ is a good curve
and that $f: C \dashrightarrow \P^1_k$ is a rational function which is
regular and invertible on a neighborhood of $Z$.
Let $\alpha = \nu_*(\divf(f))$.
We need to show that $\tau^*_X(\alpha) = 0$ in $\CH_0(X|D)$.

We choose two distinct primes $\ell_1$ and $\ell_2$ different from 
${\rm char}(k)$ and let $k_i$ denote the pro-$\ell_i$ extension of $k$ for 
$i = 1,2$.
Since each $k_i$ is a limit of finite separable extensions of the perfect
field $k$, the hypotheses of Propositions~\ref{prop:PF-fields} 
and ~\ref{prop:PF-fields-mod} are satisfied.

It follows from Proposition~\ref{prop:PF-fields} that
${\rm pr}^*_{{k_i}/k}(\alpha) \in \sR_0(S_{X_{k_i}}, D_{k_i})$ for $i = 1,2$.
It follows from the case of infinite perfect fields 
and ~\eqref{eqn:Com-sing-mod}
that $\tau^*_{X_{k_i}}(\alpha_{k_i}) = 0$ for $i =1,2$.
Equivalently, ${\rm pr}^*_{{k_i}/k} \circ \tau^*_X(\alpha) = 0$
for $i =1,2$.
Using \propref{prop:PF-fields-mod},
we can find two finite extensions $k_1'$ and $k_2'$ of $k$ of relatively prime 
degrees such that 
${\rm pr}^*_{{k_i'}/k} \circ \tau^*_X(\alpha) = 0$ for $i =1,2$. 
We conclude by applying \propref{prop:PF-fields-mod} once again
$\tau^*_X(\alpha) = 0$ in $\CH_0(X|D)$. 
This proves ~\eqref{eqn:Main-PB-PF-mod-1*}.

Now we prove the split exactness of the two sequences in the theorem.
Since a cycle $\sZ_0(X|D)$ does not meet $D$, it is clear that
$\iota_-^* \circ p_{+,*} = 0$. Similarly, $\tau^*_X \circ \Delta^* = 0$
by definitions of these maps and \lemref{lem:Sing-ML}.
Using the first part of the theorem, we only have to show that both
sequences are exact at their middle terms.

Let $\gamma\in  \CH_0(S_X)$. We can write $\gamma = \alpha_+ + \beta_-$, 
where $\alpha_+$ is a cycle supported on the component $\iota_+(X)$ and 
$\beta_-$ is a cycle supported on the component $\iota_-(X)$. 
We see then that $\gamma = p_{+,*}(\alpha - \beta) + \Delta^*(\beta)$,
where $\alpha$ and $\beta$ are the cycles $\alpha_+$ and $\beta_-$ seen in 
$X \setminus D$ (identifying the two copies of $X$), so that every element in 
the kernel of $\iota_-^*$ is clearly in the image of $p_{+,*}$.
We have therefore shown that the sequence \eqref{eq:Main-exact-seuqence} is 
split exact.

Next, suppose $\alpha \in \CH_0(S_X)$ is such that $\tau^*_X(\alpha) = 0$.
Since ~\eqref{eq:Main-exact-seuqence} is split exact, as we just showed,
we can write $\alpha =  p_{+,*}(\alpha_1) + \Delta^*(\alpha_2)$.
We then have 
\[
\begin{array}{lll}
\tau^*_X (\alpha) & =  & 0 \\
& \Leftrightarrow & \tau^*_X \circ p_{+,*}(\alpha_1) + \tau^*_X \circ
\Delta^*(\alpha_2) = 0 \\
& \Leftrightarrow & \alpha_1 + 0 = 0 \\
& \Leftrightarrow & \alpha = \Delta^*(\alpha_2).
\end{array}
\]
We have therefore shown that the sequence \eqref{eq:Main-exact-seuqence-0} is 
split exact. 
\end{proof}
%{\color{red}{\bf To be added:} extension of the result to the case of finite fields.}

\subsection{A refinement of the definition of 0-cycles with modulus}
As a consequence of \thmref{thm:Main-PB-PF-gen}, we now give the following
simplified presentation of the Chow group of 0-cycles with modulus
when the ground field is infinite and perfect.

Let $X$ be a smooth quasi-projective scheme of dimension $d \ge 1$ over
an infinite perfect field $k$ and let $D\subset X$ be an effective Cartier 
divisor. Let $\sR^{\rm mod}_0(X|D) \subset \sZ_0(X|D)$ be the subgroup generated
by ${\divf}_C(f)$, where $C \subset X$ is an integral curve not contained in $D$
and is smooth along $D$ and $f \in {\rm Ker}(\sO^{\times}_{C,D} \to 
\sO^{\times}_{C \cap D})$. Here, $\sO_{C,D}$ denotes the semi-local ring
of $C$ at $(C\cap D) \cup \{\eta\}$ with $\eta$ being the generic point of $C$.
Set $\CH^{\rm mod}_0(X|D) = 
\frac{\sZ_0(X|D)}{\sR^{\rm mod}_0(X|D)}$.
There is an evident surjection $\CH^{\rm mod}_0(X|D) \surj \CH_0(X|D)$.
%normalize the curve $C \inj X$ 
\begin{cor}\label{cor:Chow-mod-modified}
Let $X$ be as above. Then the map $\CH^{\rm mod}_0(X|D) \surj \CH_0(X|D)$ is an 
isomorphism.
\end{cor}
\begin{proof}
Under the given assumption, we get maps
\[
\CH^{\rm mod}_0(X|D) \to \CH_0(X|D)
\xrightarrow{p_{+, *}} \CH_0(S_X) \xrightarrow{\tau^*_X}
\CH_0(X|D).
\]
The proofs of Propositions~\ref{prop:map-diff-surfaces} and 
~\ref{prop:map-diff-higher} show that $\tau^*_X$ actually factors
through the map $\CH_0(S_X) \to \CH^{\rm mod}_0(X|D)$.
We thus get maps
\[
\CH^{\rm mod}_0(X|D) \to \CH_0(X|D)
\xrightarrow{p_{+, *}} \CH_0(S_X)
\xrightarrow{\tau^*_X} \CH^{\rm mod}_0(X|D),
\]
whose composite is clearly the identity. 
In particular, the map $\CH^{\rm mod}_0(X|D) \to \CH_0(X|D)$ is injective.
The corollary now follows.
\end{proof}

\subsection{The comparison theorem}
Using the modified presentation of $\CH_0(X|D)$ from 
\corref{cor:Chow-mod-modified}, we prove the following comparison
theorem for the two Chow groups of the double.

\begin{thm}\label{thm:Main-Comparison-Chow}
Let $k$ be an infinite perfect field and let $X$ be a smooth quasi-projective
scheme of dimension $d \ge 1$ over $k$ with an effective Cartier divisor $D$.
Then the canonical map $\CH^{LW}_0(S_X) \to \CH_0(S_X)$ is an isomorphism.
\end{thm}
\begin{proof}
In view of \lemref{lem:0-cycle-com-1}, we can assume $d \ge 2$.
Recall from \S~\ref{sec:PF-mod} that there are two maps
$p_{\pm,*} \colon   \sZ_0(X|D) \rightrightarrows \sZ_0(S_X,D)$.
As the first step in the proof of the theorem, we strengthen
\propref{prop:PF-mod-double} by showing that these maps descend to
group homomorphisms $p_{\pm,*} \colon \CH_0(X|D) \to \CH^{LW}_0(S_X)$.
To show this, we can use \corref{cor:Chow-mod-modified} and replace 
$\CH_0(X|D)$ by $\CH^{\rm mod}_0(X|D)$.

So let $\nu: C \inj X$ be an integral curve not contained in $D$
which is  smooth along $D$ and let $f \in {\rm Ker}(\sO^{\times}_{C,D} \to 
\sO^{\times}_{C \cap D})$. Since $C$ is smooth along $D$, the inclusion
$\nu$ is l.c.i. along $D$. 
Since $C$ is reduced, it follows from Proposition 
\ref{prop:double-prp} that $S_C := S(C,E)$ is reduced, where we let
$E = \nu^*(D)$. It follows from Proposition \ref{prop:double-prp-fine} that 
the double map $\nu': S_C \inj S_X$ is l.c.i. along $D$. In other words,
$S_C \inj S_X$ is a Cartier curve. 

We now consider the rational function $h = (h_+, h_-) := (f,1)$ on $S_C$. 
The modulus condition satisfied by $f$ on $C$ guarantees that $h$ is regular 
and invertible along  $E\subset S_C$. It is clear as in 
\propref{prop:PF-mod-double} that the divisor of $h$ trivializes 
$p_{+,*}(\nu_*{\rm div}(f))$. The argument for $p_{-,*}$ is symmetric.
We denote the maps $\CH_0(X|D) \to \CH^{LW}_0(S_X)$ obtained as above
by $p^{LW}_{\pm,*}$. 
It is clear that the composite
$\CH_0(X|D) \xrightarrow{p^{LW}_{\pm,*}} \CH^{LW}_0(S_X) \surj \CH_0(S_X)
\xrightarrow{\tau^*_X} \CH_0(X|D)$ are the identity maps (up to a sign).

Recall from \lemref{lem:Sing-ML} that $\CH_0(X)$ is the quotient of
free abelian group on the closed points of $X \setminus D$ by the
subgroup generated by $\divf(f)$, where $f$ is a rational function on
an integral curve $C$ not contained in $D$ and $f$ is regular invertible
along $D$.   
Using an easier version of \lemref{lem:reduction-basic}, one can now see that
the rational equivalences for $\CH_0(X)$ can be defined by further restricting
integral curves on $X$ which are smooth along $D$. 
In particular, they are l.c.i. on $X$ along $D$. Using such curves,
one can check from the proof of \thmref{thm:PB-main} that the map
$\Delta^*: \CH_0(X) \to \CH_0(S_X)$ can actually be lifted to the pull-back
map $\Delta^{LW, *}: \CH_0(X) \to \CH_0^{LW}(S_X)$.

We next consider the composite maps 
$\CH^{LW}_0(S_X) \surj \CH_0(S_X) \xrightarrow{\iota^*_\pm} \CH_0(X)$, which
we denote by $\iota^{LW, *}_\pm$.
It is then clear that $\iota^{LW, *}_{\pm} \circ \Delta^{LW, *} = 
{\rm Id}$ on $\CH_0(X)$. Now, the proof of ~\eqref{eq:Main-exact-seuqence}
works in verbatim to give a split exact sequence
\[
0 \to \CH_0(X|D) \xrightarrow{p^{LW}_{+,*}} \CH^{LW}_0(S_X) 
\xrightarrow{\iota_-^{LW, *}} \CH_0(X) \to 0.
\]

We thus have a commutative diagram of split exact sequences

\[
\xymatrix@C.8pc{
0 \ar[r] & \CH_0(X|D) \ar@{=}[d] \ar[r]^{p^{LW}_{+,*}} &
\CH^{LW}_0(S_X) \ar[r]^-{\iota_-^{LW, *}} \ar[d]_{can} &
\CH_0(X) \ar@{=}[d] \ar[r] & 0 \\
0 \ar[r] & \CH_0(X|D) \ar[r]_{p_{+,*}} &
\CH_0(S_X) \ar[r]_{\iota_-^{*}} &
\CH_0(X) \ar[r] & 0.}
\]

We conclude from this that the map $\CH^{LW}_0(S_X) \to \CH_0(S_X)$ is
an isomorphism. 
\end{proof}

\begin{comment}
\begin{remk}\label{remk:new-old}
Using an easier version of \lemref{lem:reduction-basic}, one can see that
the rational equivalences for the 0-cycles on $X$ can be defined using
integral curves on $X$ which are not contained in $D$ and which are smooth
along $D$. In particular, they are l.c.i. on $X$ along $D$. Using such curves,
one can check from the proof of \thmref{thm:PB-main} that the map
$\Delta^*: \CH_0(X) \to \CH_0(S_X)$ can actually be lifted to the pull-back
map $\CH_0(X) \to \CH_0^{LW}(S_X)$. 

However, we do not know any argument which can allow us to lift
the maps $p_{\pm, *}: \CH_0(X|D) \to \CH_0(S_X)$ to $\CH^{LW}_0(S_X)$.
This is where the use of our modified definition of the Chow group of
0-cycles on singular varieties becomes unavoidable.
\end{remk}
\end{comment}

\section{Albanese with modulus over $\C$}\label{sec:Alb-C}
It is classically known that a smooth projective variety $X$ over an 
algebraically closed field has an abelian variety, called the Albanese 
variety ${\rm Alb}(X)$ of $X$, associated to it, which is Cartier dual to 
the Picard variety $\Pic^0(X)$.
The Albanese comes equipped with a (surjective) Abel-Jacobi map from the Chow group $\CH_0(X)_{\deg 0}$ of zero cycles of degree zero on $X$, that is universal among regular maps to abelian varieties. % This is an abelian variety which has the property that it is a
%universal regular quotient of $\CH_0(X)_{\rm deg 0}$.
%For 0-cycles with modulus, a universal regular quotient can not be
%expected to be an abelian variety.
When $X$ is a smooth projective curve with an
effective Cartier divisor $D$ on it, a universal regular quotient for what we now call the Chow group of zero cycles with modulus, $\CH_0(X|D)_{\deg 0}$, was already
constructed by Serre in \cite{Serre} under the name of the generalized
Jacobian variety. Serre showed that this generalized Jacobian
is a commutative algebraic group which is an extension of the
Jacobian variety of the curve by a linear algebraic group.

If $X$ is now a smooth projective variety of arbitrary dimension over the field of complex numbers $\C$
and $D \subset X$ is an effective Cartier divisor such that
$D_{\rm red}$ is a strict normal crossing divisor, a universal regular
quotient of $\CH_0(X|D)$ was constructed in \cite{BS} 
as a relative intermediate Jacobian $J^{\dim(X)}_{X|D}$.
However, not many properties of this universal regular quotient are known
and the techniques used in the construction are not known to generalize
to cover the general case. 

In this paper, we use our doubling trick to 
give a direct and explicit construction of the 
relative Albanese ${\rm Alb}(X|D)$ and show that it is the 
universal regular quotient of the Chow group of 0-cycles with modulus.
As a result of our construction, we are able to prove the Roitman
torsion theorem for the Chow group of 0-cycles with modulus.
In this section, we use the modified Deligne cohomology of Levine
\cite{Levine-4} to
construct the universal regular quotient when the base field is $\C$.
%{\color{blue} 
By Theorem \ref{thm:0-cycle-affine-proj-char0}, we can identify the Levine-Weibel Chow group of zero cycles with our modified definition \ref{ssec:Rat-eq-sing-var}.
%}

\subsection{Relative Deligne cohomology}\label{sec:RDCoh}
Let $X$ be a smooth projective connected scheme over $\C$. Let $D$ be any 
effective Cartier divisor on it. As before, we write $S_X$ for the double 
construction applied to the pair $(X,D)$. We shall frequently refer to the 
following square:
\begin{equation}\label{eq:fundamental-square}
    \xymatrix{ D \ar@{^{(}->}[r]^{j} \ar@{^{(}->}[d]_{j} & 
X \ar@{^{(}->}[d]^{i_-} \\
    X\ar@{^{(}->}[r]_{i_+} & S_X.
    }
\end{equation}

Let $r\geq1$ be an integer. 
Following Levine \cite{Levine-4}, we denote by $\ZDeSXr$ (resp. $\ZDeDr$)  the 
modified Deligne-Beilinson complex on $S_X$ (resp. on $D$). Since $S_X$ is 
projective, both are given by the ``naive'' Deligne complexes
\[\ZDeSXr = \Z(r)_{S_X} \to \cO_{S_X}\xrightarrow{d} 
\Omega^1_{S_X}\to\ldots \to \Omega^{r-1}_{S_X}\]
\[\ZDeDr = \Z(r)_D \to \cO_{D}\xrightarrow{d} \Omega^1_{D}\to\ldots \to 
\Omega^{r-1}_{D}\]
on $(S_X)_{\rm an}$ (resp. on $D_{\rm an}$), the analytic space associated to $S_X$ (resp. to $D$).

We can consider the complex $\ZDeXr$ on $X$ as well. Since $X$ is smooth and 
projective, we have by definition $\ZDe{X}{r} = \ZDeXrun$, where
$\ZDeXrun$ is the classical Deligne complex. Let 
$i_\pm\colon X\hookrightarrow S_X$ be the inclusions of the two components in 
the double $S_X$ and let $\Delta\colon S_X\to X$ be the natural projection. We 
denote by 
$\pi\colon X\amalg X\to S_X$ the  map $(i_+, i_-)$ from the normalization. 
Finally, we let $j\colon D\hookrightarrow X$ denote the inclusion of $D$ in 
$X$. Recall from \propref{prop:double-prp}(8) that $D$ is a conducting 
subscheme for $\pi$.

We define the following objects in the bounded derived category of complexes 
of sheaves of abelian groups on $(S_X)_{\rm an}$ (resp. on $X_{\rm an}$):
\[ \ZDe{X|D}{r} = Cone(\ZDe{S_X}{r}\xrightarrow{i_-^*} \R 
i_{-,*}\ZDeun{X}{r})[-1] \]
\[  \ZDe{(X,D)}{r} = Cone(\ZDeun{X}{r}\xrightarrow{j^*} \R 
j_{*}\ZDe{D}{r})[-1]. 
\]
We define the complexes $\Z(r)_{X|D}$ and $\Z(r)_{(X,D)}$ in a similar
fashion.

Note that the maps are all finite, so that the derived functors $\R i_{-,*}$ 
and $\R j_*$ are superfluous. Since all the maps over which we are taking the 
cones are actually surjective as maps of complexes, the cones can be computed 
directly as kernels in the category of complexes of sheaves on $S_X$. We have 
the following commutative diagram relating them
\begin{equation}\label{eq:comparison-complexes-Deligne}
    \xymatrix{ 0\to \ZDe{X|D}{r} \ar[r] \ar[d]_{\Phi_{X,D}} & \ZDe{S_X}{r} 
\ar[r]^{i_-^*}\ar[d]^{i_+^*}& \ZDeun{X}{r}\ar[d]^{j*} \to0\\
0\to  \ZDe{(X,D)}{r} \ar[r] &\ZDeun{X}{r}\ar[r]^{j^*} &\ZDe{D}{r} \to 0.
}\end{equation}

%We define in a similar fashion the sheaves $\Z(r)_{X|D}$ and $\Z(r)_{X,D}$ as 
%kernels of the maps $i_-^*$ and $j^*$ respectively between the constant 
%sheaves $\Z(r)$ on $S_X$, $X$ and $D$.

\subsection{Construction of the Albanese varieties}
\label{ssec:Construction-Albanese-C} 
Let $r = d = \dim(X)$. In order to simplify the notation, we denote by 
$\ol{S_X}\xrightarrow{\pi} S_X$ the normalization of $S_X$, i.e., 
the disjoint union of two copies of $X$. This gives an induced map 
$\pi_D\colon \ol{D} = D\amalg D \to D$. We shall use the relative complexes 
$\ZDe{X|D}{d}$ and $\ZDe{(X,D)}{d}$ to define in the usual way two natural 
receptors for the group of algebraic cycles with modulus that are 
homologically trivial. Coherently with the excision 
Theorem \ref{thm:Main-PB-PF-gen} in the projective case over the complex 
numbers, we will show that the two, a priori different constructions, give 
rise to the same invariant naturally attached to the pair $(X,D)$, that we 
call {\sl the Albanese variety with modulus}.

\subsubsection{The Albanese with modulus}\label{sec:Alb-mod}
By construction, the map $\Delta\colon S_X\to X$ satisfies 
$\Delta\circ i_\pm = {\rm Id_X}$, and therefore induces a splitting of the 
restriction maps $i_\pm^*$ from the cohomology of $S_X$ to the cohomology of 
$X$. This gives the commutative diagram of the analytic cohomology groups
\[
\xymatrix{ 
0\to \H^{2d}(S_X, \ZDe{X|D}{d})\ar[r] \ar[d]_{\epsilon_{X|D}} & 
\H^{2d}(S_X, \ZDe{S_X}{d} )\ar[d]^{\epsilon_{S_X}} \ar@<.5ex>[r]^{i^*_-}  &  
\H^{2d}(X, \ZDeun{X}{d})\ar[d]^{\epsilon_X} \ar@<.5ex>[l]^{\Delta^*} \to 0 \\
0\to H^{2d}(S_X, \Z({d})_{X|D})\ar[r] & H^{2d}(S_X, \Z(d)_{S_X} ) 
\ar@<.5ex>[r]^{i^*_-}  &  H^{2d}(X, \Z(d)_{X}) \ar@<.5ex>[l]^{\Delta^*} \to 0,
}
\]
where the vertical maps are induced by the natural projections 
$\Z(d)^{\sD*}\to \Z(d)$.

Since $X$ is smooth projective and connected, we have 
$H^{2d}(X, \Z({d})_{X}) = \Z$ and $H^{2d}(S_X, \Z{(d)}_{S_X} )$ \\
$= \Z\oplus \Z$, one copy for each component of 
$S_X$. In particular, the restriction map $i^*_-$ sends $(0,1)$ to $1$ and we 
can identify  $ H^{2d}(S_X, \Z{(d)}_{X|D})$ with the other copy of $ \Z$.
Under this identification, we denote by $A^d({X|D})$ the kernel of the natural 
map
\[ 
0\to A^d({X|D}) \to  \H^{2d}(S_X, \ZDe{X|D}{d}) \xrightarrow{\epsilon_{X|D}}
\Z 
\]
and call it the \textit{Albanese variety of $X$ with modulus $D$}. 

By construction, we have an isomorphism \[A^d({X|D}) = 
\H^{2d-1}(S_X, \Omega^{<d}_{(S_X, X_-)}) / H^{2d-1}(S_X, \Z(d)_{X|D}),\]
where $\Omega^{<d}_{(S_X, X_-)}$ denotes
$Cone(\Omega^{<d}_{S_X}\xrightarrow{i_-^*} \Omega^{<d}_X)[-1]$.
Following 
\cite{ESV}, we define the generalized Albanese variety of $S_X$, denoted 
$A^d(S_X)$, as the kernel of the map \[\H^{2d}(S_X, \ZDe{S_X}{d} ) \surj 
H^{2d}(S_X, \Z(d)_{S_X} ).\] This gives a split short exact sequence
\begin{equation}\label{eq:Alb-modulus-Alb-Sx}
\xymatrix{
0\to A^d(X|D) \xrightarrow{p_{+,*}} 
A^d(S_X)  \ar@<.5ex>[r]^-{i^*_-} &  A^d(X)\ar@<.5ex>[l]^-{\Delta^*} \to 0, } 
\end{equation}
where $A^d(X) = {\rm Alb}(X)$ is  the usual Albanese variety of $X$, namely, 
$\H^{2d-1}(X, \Omega^{<d}_{X}) / H^{2d-1}(X, \Z(d)_{X})$. Write $J^d(S_X)$ for the 
quotient
\[J^d(S_X) = \frac{H^{2d-1}(S_X, \C(d))}{F^0 H^{2d-1}(S_X, \C(d)) + 
{\rm image}H^{2d-1}(S_X, \Z(d))}.\]
It is a semi-abelian variety by a result of Deligne. 
By \cite[Lemma 3.1]{ESV}, there is a natural surjection
\[
\psi\colon A^d(S_X)\to J^d(S_X)
\]
whose kernel is a $\C$-vector space. Moreover, there is a unique structure of 
algebraic group on $A^d(S_X)$ making $\psi$ a morphism of algebraic groups, 
with unipotent kernel.

Notice that there is a short exact sequence
\begin{equation}\label{eq:Alb-Sx-Alb-Sxbar} 
0\to G \to A^d(S_X) \xrightarrow{\pi^*} A^d(X)\times A^d(X) \to 0 
\end{equation}
where $G$ is simply defined as the kernel of the map $\pi^*$, 
that is surjective since $H^{2d}(S_X, \Z(d)) \simeq H^{2d}(\ol{S_X}, \Z(d))$ 
and $\H^{2d-1}(S_X, \Omega^{<d}_{S_X}) \twoheadrightarrow 
\H^{2d-1}(\ol{S_X}, \Omega^{<d}_{\ol{S_X}})$ because of GAGA and the fact that 
$\pi$ is finite and birational. A combination of \eqref{eq:Alb-modulus-Alb-Sx} 
and \eqref{eq:Alb-Sx-Alb-Sxbar} gives then the short exact sequence
\begin{equation}\label{eqn:Alb-Sx-Alb-Sxbar-0}
0\to G\to A^d(X|D)\xrightarrow{\phi} A^d(X)\to 0,
\end{equation}
where the (surjective) forgetful map $\phi$ is same as the map
$i^*_+ \circ p_{_+,*}\colon A^d(X|D)\to A^d(X)$.

\subsubsection{The relative Albanese}\label{sec:Alb-rel}
Taking cohomology of the bottom exact sequence in 
\eqref{eq:comparison-complexes-Deligne}, we get exact sequence
\[ 
\H^{2d-1}(D,\ZDe{D}{d}) \to \H^{2d}(X, \ZDe{(X,D)}{d}) \to 
\H^{2d}(X, \ZDeun{X}{d}) \to 0, 
\]
where the term $\H^{2d}(D, \ZDe{D}{d}) =0$ for dimension reasons.

We denote by $A^d({X,D})$ the kernel of the natural map  
\[
\H^{2d}(X, \ZDe{(X,D)}{d}) \to \H^{2d}(X, \Z(d)_{(X,D)}) \simeq
H^{2d}(X, \Z(d)_X) = \Z,\]
and call it the \textit{relative Albanese variety of the pair $(X,D)$}.
It fits in a short exact sequence
\[ 0\to \frac{\H^{2d-1}( D, \ZDe{D}{d})}{\H^{2d-1}( X, \ZDeun{X}{d})} \to 
A^d({X,D})\to A^d(X)\to 0\]
The vertical maps in \eqref{eq:comparison-complexes-Deligne} 
induce then the following diagram
\begin{equation}\label{eq:diag-Alb-modulus-and-relative-Alb}
\xymatrix{ 
0\ar[r]& A^d(X|D) \ar[r]^{p_{+,*}} \ar[d]_{\phi_{X|D}^d} 
\ar@{->>}[rd]^{\varphi} & A^d(S_X) \ar[d]^{i_+^*}\ar[r]^{i_-^*} & 
A^d(X)\ar[d] \to 0 \\
0\to \frac{\H^{2d-1}( D, \ZDe{D}{d})}{\H^{2d-1}( X, \ZDeun{X}{d})} \ar[r] & 
A^d({X,D}) \ar[r] & A^d (X) \ar[r]& 0.
}
\end{equation}

\begin{prop}\label{prop:surj-alb-any-dimension}
Let $X$ be a smooth projective $\C$-scheme of dimension $d \ge 1$. 
Then the natural map $\phi_{X|D}^d\colon A^d(X|D)\to A^d(X,D)$ of 
\eqref{eq:diag-Alb-modulus-and-relative-Alb} is an isomorphism.
\end{prop}
\begin{proof}
If $d =1$, it follows from \cite[Proposition~1.4]{LW} and
\propref{prop:Main-Sequence-Pic} that $A^1(X|D) \simeq \Pic(X,D)$.
On the other hand, we have
\[A^1(X,D) = \H^2(X, \ZDe{(X,D)}{2}) \simeq H^1(X, (1 + \sI_D)^{\times}),\]
where $\sI_D$ is the ideal sheaf of $D \subset X$.
It follows from \cite[Lemma~2.1]{SV} that $H^1(X, (1 + \sI_D)^{\times}) 
\simeq \Pic(X,D)$
and the proposition follows. We can thus assume that $d \ge 2$.

We can assume that $X$ is connected.
Since the singular cohomology satisfies the Mayer-Vietoris property with 
respect to the square \eqref{eq:fundamental-square}, we are left to show that 
\begin{equation}\label{eq:eq1-proof-Albanese-iso}  
\H^{2d-1}(S_X, \Omega^{<d}_{(S_X, X_-)}) \xrightarrow{i_+^*} 
\H^{2d-1}(X, \Omega^{<d}_{(X, D)})
\end{equation}
is bijective. Here, $\Omega^{<d}_{(X, D)} = 
Cone ( \Omega^{<d}_{X}\to \Omega^{<d}_{D})[-1]$ and the morphism between the 
cohomology groups is induced by the restriction along 
$i_+\colon X\hookrightarrow S_X$. 

Using GAGA, it is equivalent to prove the bijectivity for the
cohomology of the associated Zariski sheaves of differential forms.
It is easy to check from ~\eqref{eqn:double-2} that $i^*_*$
induces an isomorphism $\sI_{X_-} \xrightarrow{\simeq} \sI_D$,
where $\sI_D$ is the ideal sheaf for $D \subset X_{+}$.
Furthermore, $(S_X \setminus D) = (X_{-} \setminus D) \amalg 
(X_{+} \setminus D)$. This information can be used to get a commutative
diagram of exact sequences

\begin{equation}\label{eq:eq1-proof-Albanese-iso-0}
\xymatrix@C.6pc{
0 \ar[r] & \Omega^{[1,d-1]}_{(S_X, X_-)}[-1] \ar[r] \ar[d]_{i^*_+} &
\Omega^{<d}_{(S_X, X_-)} \ar[r] \ar[d]^{i^*_+} & \sI_{X_-} \ar[r] 
\ar[d]^{\simeq} & 0 \\
0 \ar[r] & \Omega^{[1,d-1]}_{(X, D)}[-1] \ar[r]  &
\Omega^{<d}_{(X, D)} \ar[r] & \sI_{D} \ar[r] & 0.}
\end{equation}

We therefore have to prove that the map of Zariski cohomology groups
\begin{equation}\label{eq:eq1-proof-Albanese-iso-1}  
\H^{2d-2}(S_X, \Omega^{[1,d-1]}_{(S_X, X_-)}) \xrightarrow{i_+^*} 
\H^{2d-2}(X, \Omega^{[1,d-1]}_{(X, D)})
\end{equation}
is bijective. 

We start by proving surjectivity.
Letting $\sF_1[-1]$ and $\sF_2[-1]$ denote the kernel and the cokernel of the
left vertical arrow in ~\eqref{eq:eq1-proof-Albanese-iso-0}, respectively,
we have an exact sequence 
\begin{equation}\label{eq:eq1-proof-Albanese-iso-2} 
0 \to \sF_1 \to \Omega^{[1,d-1]}_{(S_X, X_-)} \xrightarrow{i^*_+}
\Omega^{[1,d-1]}_{(X, D)} \to \sF_2 \to 0.
\end{equation}
Setting $\sG = {\Omega^{[1,d-1]}_{(S_X, X_-)}}/{\sF_1}$, we obtain a diagram
of exact sequences
{\small     \[ \xymatrix{ \H^{2d-2}(\cF_1)\ar[r]&  
\H^{2d-2}(S_X, \Omega^{[1, d-1]}_{(S_X, X_-)}) \ar[r] & 
\H^{2d-2}(\cG) \ar[r]\ar@{=}[d] & \H^{2d-1}(\cF_1) &\\
        & \H^{2d-3}(\cF_2) \ar[r] & \H^{2d-2}(\cG) \ar[r]& 
\H^{2d-2}(X, \Omega^{[1, d-1]}_{(X, D)}) \ar[r]& \H^{2d-2}(\cF_2).    
        }
        \]}

Since $d \ge 2$ and $\sF_i$ are complexes of coherent sheaves supported on $D$,
a standard spectral sequence argument shows that
$\H^{2d-2+j}(\cF_i) = 0$ for $i = 1,2$ and $j \ge 0$.
%This uses $2d-2 > d-1$
Analyzing the spectral sequence computing the cohomology group 
$H^{2d-3}(\cF_2)$, we see that thanks to the same dimension argument used 
above, the only surviving term is given by 
$H^{d-1}(\Omega^{d-1}_{(X,D)} / \Omega^{d-1}_{(S_X, X_-)})$. 
We thus get  an exact sequence

\begin{equation}\label{eq:eq2-proof-Albanese-iso} 
0 \to \frac{H^{d-1}(\Omega^{d-1}_{(X,D)} / \Omega^{d-1}_{(S_X, X_-)})}
{\H^{2d-3}(X, \Omega^{[1,d-1]}_{(X, D)})}  
\to \H^{2d-2}(S_X, \Omega^{[1, d-1]}_{(S_X, X_-)}) 
\xrightarrow{i^*_+} \H^{2d-2}(X, \Omega^{[1, d-1]}_{(X, D)}) \to 0
\end{equation}
and this proves the surjectivity of $i^*_+$.

We now prove injectivity. Write $D_{\rm sing}$ for $|D|_{\rm sing}$.
Let $U = X\setminus D_{\rm sing} \hookrightarrow X$ and write $D_U$ for the 
restriction of $D$ to $U$. We note that $(D_U)_{\rm red}$ is a smooth 
(possibly non connected) divisor on $U$. We claim that the restriction to $U$ 
of the sheaf $\cF_3 = \Omega^{d-1}_{(X,D)} / \Omega^{d-1}_{(S_X, X_-)}$ is zero. 
Note that since $\cF_3$ is anyway supported on $D$, our claim actually implies 
that $\cF_3$ is supported on $Y  = D_{\rm sing}$. 
Since $\dim(Y) \leq d-2$,
 this will give the desired vanishing of the cohomology group 
$H^{d-1}(\Omega^{d-1}_{(X,D)} / \Omega^{d-1}_{(S_X, X_-)})$. Since the first term of 
the sequence \eqref{eq:eq2-proof-Albanese-iso} is a quotient of  
$H^{d-1}(\Omega^{d-1}_{(X,D)} / \Omega^{d-1}_{(S_X, X_-)})$, this will complete the 
proof  of the proposition. 

We now prove the claim. 
Since all the components of $(D_U)_{\rm red}$ are regular and disjoint, 
we can assume that at a point 
$y$ of $U$, a local equation for $D_U$ is given by $x^n$, where 
$(x, x_1, \ldots, x_{d-1})$ is a regular system of parameters in 
$A=\cO_{X, y}$. Write $A'= A/(x^n)$, $B= A/(x)$ and let $R$ be the double 
construction applied to the pair $(A, (x^n))$. Write $\Omega^{d-1}_{(R, A_-)}$ 
for the kernel of the second projection 
$i_-^*\colon \Omega^{d-1}_{R/\C} \surj \Omega^{d-1}_{A/\C}$ and write 
$\Omega^{d-1}_{(A, A')}$ for the kernel of the restriction map 
$j^*\colon \Omega^{d-1}_{A/\C}\surj \Omega^{d-1}_{A'/\C}$. 
The claim is equivalent to showing that the induced map
\begin{equation}\label{eq:eq3-proof-Albanese-iso} 
i_+^*\colon \Omega^{d-1}_{(R, A_-)} \to \Omega^{d-1}_{(A, A')}
\end{equation}
is surjective.

Since $A$ is a regular local ring of dimension $d$, $B$ is a regular local 
ring of dimension $d-1$ and the infinitesimal lifting property 
(see \cite[Proposition 4.4]{HartDef}) gives a splitting of the projection 
$A'\to B$. Since the relative embedding dimension of $A'$ in $A$ is 1, this 
gives an isomorphism $A' \simeq B[x]/(x^n)$. We can then compute the following 
modules
\[
\Omega^1_{A'/\C} = (\Omega^1_{B/\C}\tensor_B A')\oplus 
(\bigoplus_{j=0}^{n-2} B x^j \dx) \quad \text{ as $B$-module};
\]
\[
\begin{array}{lll}
\Omega^{d-1}_{A'/\C} & = & (\Omega^{d-2}_{B/\C}\tensor_B A/(x^{n-1})\dx)  \oplus
(\Omega^{d-1}_{B/\C}\tensor_B A')  \\
& = & \stackrel{d-1}{\underset{i = 1}\bigoplus} A/(x^{n-1})\dx \wedge {\rm d}x_1 
\wedge\ldots \wedge \overset{\vee}{{\rm d}x_i}\wedge \ldots \wedge 
{\rm d}x_{d-1} \oplus A'  {\rm d}x_1  \wedge \ldots \wedge  {\rm d}x_{d-1};  \\
\end{array}
\]
\[ 
\Omega^{d-1}_{A/\C} = \bigoplus_{i=1}^{d-1}(A \dx\wedge {\rm d}x_1 
\wedge\ldots \wedge \overset{\vee}{{\rm d}x_i}\wedge \ldots \wedge 
{\rm d}x_{d-1}) \oplus A  {\rm d}x_1  \wedge \ldots \wedge  {\rm d}x_{d-1}.
\]

One can easily check from the above expressions that 
$\Omega^{d-1}_{(A, A')}$ is generated by the forms
$x^n {\rm d}x_1  \wedge \ldots \wedge  {\rm d}x_{d-1}$ and 
$( x^{n-1}\dx\wedge {\rm d}x_1 \wedge\ldots \wedge 
\overset{\vee}{{\rm d}x_i}\wedge \ldots \wedge {\rm d}x_{d-1} )_{i=1}^{d-1}$ 
as $A$-module.

We now consider the diagram 
\[
\xymatrix{ 
0\to \Omega^{d-1}_{(R, A_-)}\ar[d]_{i_+^*}\ar[r] & 
\Omega^{d-1}_{R/\C} \ar[d]^{i_+^*}\ar[r]^{i_-^*} & 
\Omega^{d-1}_{A/\C}\ar[d] \to 0 \\
0 \to \Omega^{d-1}_{(A,A')} \ar[r]   & 
\Omega^{d-1}_{A/\C} \ar[r]& \Omega^{d-1}_{A'/\C}\to 0.}
\]

We can lift (up to multiplication by elements in $\C^{\times}$)  the generators  
$x^n {\rm d}x_1  \wedge \ldots \wedge  {\rm d}x_{d-1}$ and 
$( x^{n-1}\dx\wedge {\rm d}x_1 \wedge\ldots 
\wedge \overset{\vee}{{\rm d}x_i}\wedge \ldots 
\wedge {\rm d}x_{d-1} )_{i=1}^{d-1}$ of $\Omega^{d-1}_{(A,A')}$  via the 
projection $i_+^*$ to elements
\begin{align*} & (x^n,0) {\rm d}(x_1,x_1)  \wedge \ldots \wedge  
{\rm d}(x_{d-1}, x_{d-1}), \\ 
&( {\rm d}(x^n, 0)\wedge {\rm d}(x_1,x_1) 
\wedge\ldots \wedge \overset{\vee}{{\rm d}(x_i,x_i)}\wedge \ldots \wedge 
{\rm d}(x_{d-1}, x_{d-1}) )_{i=1}^{d-1}
\quad\text{ in } \Omega^{d-1}_{R/\C}
\end{align*}

and we immediately see that they go to zero via the second projection 
$i_-^*$, so that they lift to elements in $\Omega^{d-1}_{(R, A_-)}$, proving 
that \eqref{eq:eq3-proof-Albanese-iso} is surjective and therefore completing 
the proof of the proposition.
\end{proof}

As a consequence of \propref{prop:surj-alb-any-dimension},
~\eqref{eq:Alb-Sx-Alb-Sxbar} and ~\eqref{eqn:Alb-Sx-Alb-Sxbar-0},
it follows that $A^d(X|D)$ is an extension of the abelian variety $A^d(X)$
by the linear group 
$G = \frac{\H^{2d-1}( D, \ZDe{D}{d})}{\H^{2d-1}( X, \ZDeun{X}{d})}$.

\begin{remk}\label{remk:Mod-rel-surface}
If $X$ is a surface, then one checks that $\ZDe{(X,D)}{2}$ is
quasi-isomorphic to the complex $j_{!}(\Z) \to \sI_D \to \Omega^1_{(X,D)}$
(see \cite[\S~4]{Levine-4}), where $j\colon  X \setminus D \inj X$ is the
open inclusion.
When $D_{\rm red}$ is a strict normal crossing divisor,
the complex $j_{!}(\Z) \to \sI_D \to \Omega^1_{(X,D)}$ is used
in \cite{BS} to construct a universal regular quotient of 
$\CH_0(X|D)_{\deg  0}$. One consequence of
\propref{prop:surj-alb-any-dimension} is that it provides a cohomological proof
that the Albanese variety
with modulus of \S~\ref{sec:Alb-mod} coincides with the one constructed in
\cite{BS} when $X$ is a surface and $D_{\rm red}$ is strict normal crossing. 
The universality Theorem \ref{thm:univ-Alb-overC} will tell us, more generally, that the two constructions agree in higher dimension as well, whenever 
$D_{\rm red}$ is strict normal crossing divisor.
\end{remk}

%{\color{blue}
\section{An interlude on regular homomorphisms}
Let $k$ be an algebraically closed field and let $Y$ be a projective reduced scheme over $k$. Let $\CH_0^{LW}(Y) = \CH_0^{LW}(Y, Y_{\rm sing})$ and $\CH_0(Y) = \CH_0(Y, Y_{\rm sing})$ be the groups of zero-cycles associated to $Y$. Let $U$ be a dense open subscheme of $Y_{\rm reg}$ and choose a base point $x_i$ in every irreducible component $U_i$ of $U$. The following Definition is taken from \cite{ESV}.
\begin{defn}\label{def:definition-regular-homomorphism} Let $G$ be smooth commutative algebraic group over $k$. We say that a group homomorphism $\rho'\colon \CH_0^{LW}(Y)_{\deg 0}\to G$ of abstract groups is a regular homomorphism if the map $\pi\colon U\to G$ with $\pi_{|U_i}(x)  = \rho'([x] - [x_i])$ is a morphism of schemes (i.e., there exists a morphism of schemes whose restriction to closed points coincides with $\pi$).
\end{defn}
The same definition allow us to talk about regular homomorphisms from the Chow group $\CH_0(Y)_{\deg 0}$ instead.
\begin{remk}\label{remk:Reg-ESV-0}
In \cite[Definition~1.14]{ESV}, there are other equivalent definitions of 
regular map from the Levine-Weibel
Chow group of 0-cycles on a singular projective variety to a smooth 
commutative algebraic group. We will not need this explicitly, but we recall one of them  for the reader who wishes to remove a reference to the base points.

Let $U$ be an open dense in $Y_{\rm reg}$. Let $U_1, \ldots, U_s$ be the irreducible components of $U$. Consider the map
\[ \gamma^{(-)}\colon \Pi_U = \bigcup_{i=1}^s U_i \times U_i \to \CH_0^{LW}(Y)_{\deg 0}\]
defined by $\gamma^{(-)}(u, u') = [u]- [u']$.   We have then:
\end{remk}

\begin{prop}[Corollary 1.13, \cite{ESV}] \label{prop:equiv-conditions-regular-morphisms}Let $G$ be a smooth commutative algebraic group. Let $\rho'\colon \CH_0^{LW}(Y)_{\deg  0}\to G$ be a group homomorphism. Then the following conditions are equivalent.
	\begin{romanlist}
		\item The composition $\rho'\circ \gamma^{(-)}\colon \Pi_{Y_{\rm reg}}\to G$ is a morphism of scheme.
		\item The morphism $\rho'$ is regular in the sense of Definition \ref{def:definition-regular-homomorphism}.
		%Given a base point $u_i$ on each irreducible component $U_i$ of some open dense subscheme $U$ of $Y_{\rm reg}$, the map $\pi\colon U\to G$ with $\pi_{|U_i}(x) = \rho ([x]-[u_i])$ is a morphism of schemes.
		%\item Given any $m>0$ and base points $u_i$  on each irreducible component $U_i$ of some open dense subscheme $U$ of $Y_{\rm reg}$, let $\gamma_{m}=\gamma_{U,m}$ be the map \[\gamma_m\colon S^m(U) = S^m(\bigcup_i U_i)\to \CH_0^{LW}(Y)_{\deg 0} \]
		%defined by $\gamma_m(x_1,\ldots, x_s) = \sum_{j=1}^s([x_i]-[u_{\theta(j)}]$, where $\theta(j)=i$ if $x_j\in Y_i$. 
		%Then the composition $\rho\circ \gamma_{U,m}\colon S^m(U)\to G$ is a morphism of schemes.
		\end{romanlist}
	\end{prop}
As above, the expression ``the map $\phi$ is a morphism of schemes'', stands  for ``there exists a morphisms of schemes whose restriction to closed points coincides with $\phi$''.%® Thanks to the Proposition, the homomorphism $\rho$ is regular if one of the equivalent conditions of Proposition \ref{prop:equiv-conditions-regular-morphisms} is satisfied. 
\subsection{The case of the double}\label{ssec:SettingDoubleRegular}Let $X$ be now a smooth connected projective $k$-variety, equipped with an effective Cartier divisor $D$ and let $X_+$ and $X_-$ denote as above the irreducible components of the double $S_X$ of $X$ along $D$. 
Given any dense open subset $V\subset (S_X)_{\rm reg}$, we denote by $V_+$ and 
$V_-$ the intersection of $V$ with $X_+^o$ and $X_-^o$ respectively. 
We adapt the definition of regular homomorphism recalled above to the 
particular geometry of the double $S_X$.

\begin{defn}\label{defn:Reg-ESV}
Let $G$ be a smooth commutative algebraic group over $k$. A homomorphism 
\[
\rho'\colon \CH_0(S_X)_{\deg 0} \to G
\]
is called a regular homomorphism if given base points $x_{0, \pm}$ on each 
irreducible components $V_\pm$ of some open dense subscheme $V$ of 
$(S_X)_{\rm reg}$, the composition of $\rho'$ with the map
\[ 
\pi^V_{x_{0,\pm}}\colon V\to \CH_0(S_X)_{\deg 0}, \quad x\mapsto [x] - 
[x_{0, \theta(x)}], 
\]
where $\theta(x) =\pm$ according to $x\in V_\pm$, is a morphism of schemes. In a similar fashion, one can define the notion of regular homomorphism using the Levine-Weibel Chow group of $0$-cycles on the double $S_X$.
\end{defn}

%The reader should also keep in mind that `morphism of schemes' above means
%a morphism of schemes which induces $\pi^V_{x_{0,\pm}}$ on closed points. 

It follows from \cite[Lemma~1.4]{ESV} that the image of a regular
homomorphism is a connected algebraic closed subgroup of $G$.
The initial object (whose underlying map is necessarily surjective) in the
category of regular maps $\CH_0(S_X)_{\deg 0} \to G$
is called the universal regular quotient of $\CH_0(S_X)_{\deg 0}$.
It was shown in \cite{ESV} that the universal regular quotient 
of $\CH^{LW}(Y)_{\deg 0}$ always exists
for any projective variety $Y$ over $k$. We state this theorem below
for $S_X$.

\begin{thm}$($\cite[Theorem 1]{ESV}$)$\label{thm:universalityESV} There exists a smooth connected algebraic group ${\rm Alb}(S_X)$, together with a regular surjective homomorphism $\rho_{S_X}\colon \CH_0^{LW}(S_X)_{\deg 0} \to {\rm Alb}(S_X)$ such that $\rho_{S_X}$ is universal among regular homomorphisms from $\CH_0^{LW}(S_X)_{\deg 0}$ to smooth commutative algebraic groups. When $k=\C$, then ${\rm Alb}(S_X)$ coincides with the Albanese variety $A^d(S_X)$ introduced in \eqref{eq:Alb-modulus-Alb-Sx}.
\end{thm}

\subsection{The universal semi-abelian quotient of 
$\CH_0(Y)_{\deg 0} $}\label{sec:SAQ}
Let $Y$ be a reduced projective scheme of dimension
$d \ge 1$ over an algebraically closed field $k$. 
Let us assume the characteristic of $k$ to be positive in this subsection.
In this case, do not know if the canonical map $\CH^{LW}_0(Y) \to \CH_0(Y)$
is an isomorphism. A weaker question is if the Albanese map
$\rho_Y: \CH_0^{LW}(Y)_{\deg 0} \to {\rm Alb}(Y)$ factors through 
$\CH_0(Y)_{\deg 0}$. We expect this to be true, but we do not yet know how
to verify this either. 
The reason for this is that, one does not know any description of
the Albanese variety in positive characteristic except its existence.
We however show in this section that the semi-abelian Albanese variety of
$Y$ indeed has this property. We shall use this to prove a comparison
result for the two Chow groups in positive characteristic.

The following description of the semi-abelian Albanese variety of
$Y$ is recalled from \cite[\S~2]{Mallick}. 
Let $\pi: Y^N \to Y$ denote the normalization map. Let ${\rm Cl}(Y^N)$
and $\Pic_W(Y^N)$ denote the class group and the Weil Picard group of
$Y^N$. Recall (see e.g.~\cite{Weil}) that $\Pic_W(Y^N)$ is the subgroup of
${\rm Cl}(Y^N)$ consisting of Weil divisors which are algebraically
equivalent to zero in the sense of \cite[Chap.~19]{Fulton}.
Let ${\rm Div}(Y)$ denote the free abelian group of Weil divisors on $Y$.
Let $\Lambda_{1}(Y)$ denote the subgroup of ${\rm Div}(Y^N)$ generated
by the Weil divisors which are supported on $\pi^{-1}(Y_{\rm sing})$.
This gives us a map $\iota_Y: \Lambda_{1}(Y) \to 
\frac{{\rm Cl}(Y^N)}{\Pic_W(Y^N)}$.

Let $\Lambda(Y)$ denote the kernel of the canonical map
\begin{equation}\label{eqn:Weil-Pic}
\Lambda_{1}(Y) \xrightarrow{(\iota_Y, \pi_*)} \frac{{\rm Cl}(Y^N)}{\Pic_W(Y^N)}
\oplus {\rm Div}(Y).
\end{equation}

The semi-abelian Albanese variety of $Y$ is the Cartier dual of
the 1-motive \[[\Lambda(Y) \to \Pic_W(Y^N)]\] and is denoted by $J^d(Y)$.
It follows from \cite[\S~4]{Mallick} that $J^d(Y)$ is the universal
semi-abelian quotient of the Esnault-Srinivas-Viehweg Albanese variety
${\rm Alb}(Y)$. Let $\rho^{\rm semi}_Y: \CH_0^{LW}(Y)_{\deg 0} \surj J^d(Y)$ denote
the universal regular homomorphism.

\begin{lem}\label{lem:semi-ab-functor}
Let $f: Z \to Y$ be a smooth projective morphism of relative dimension
$r$. Then there is a push-forward map $f^{\rm semi}_*: J^{d+r}(Z) \to J^d(Y)$
and a commutative diagram
\begin{equation}\label{eqn:semi-ab-functor-0}
\xymatrix@C1pc{
\sZ_0(Z, Z_{\rm sing}) \ar[d]_{\rho^{\rm semi}_Z} \ar[r]^{f_*} &
\sZ_0(Y, Y_{\rm sing}) \ar[d]^{\rho^{\rm semi}_Y} \\
J^{d+r}(Z) \ar[r]_{f_*} & J^d(Y).}
\end{equation}
\end{lem}
\begin{proof}
Since $f$ is smooth and projective, 
it follows that $f^N: Z^N \to Y^N$ is also smooth and
projective. It is clear that the flat pull-back $f^{N,*}$ takes
integral Weil divisors to integral Weil divisors. It is also clear
that this map preserves Weil divisors which are algebraically equivalent to
zero. Since $Z_{\rm sing} = f^{-1}(Y_{\rm sing})$, we see that
$f^{N, *}(\Lambda_{1}(Y)) \subset \Lambda_{1}(Z)$. 
Furthermore, it follows from \cite[Proposition~1.7]{Fulton} that
$f^{N, *}({\rm Ker}(\Lambda_{1}(Y) \to {\rm Div(Y)}) \subset
{\rm Ker}(\Lambda_{1}(Z) \to {\rm Div(Z)})$.
We conclude that $f^*$ induces a morphism of 1-motives
$f^*: [\Lambda(Y) \to \Pic_W(Y^N)] \to [\Lambda(Z) \to \Pic_W(Z^N)]$
and hence a map $f_*:J^{d+r}(Z) \to J^d(Y)$. 

To show the commutative diagram~\eqref{eqn:semi-ab-functor-0}, we need to 
observe that $J^d(Y)$ is a quotient of the Cartier dual $J^d_{\rm Serre}(Y)$
of the 1-motive
$[\Lambda_1(Y) \to \Pic_W(Y^N)]$ and this dual semi-abelian variety
is the universal object in the category of morphisms from $Y_{\rm reg}$ to
semi-abelian varieties (see \cite{Serre-1}).
Since $Z_{\rm reg} = f^{-1}(Y_{\rm reg})$, it follows from this universality
of $J^d_{\rm Serre}(Y)$ that there is a commutative diagram
as in ~\eqref{eqn:semi-ab-functor-0} if we replace $J^d(Y)$ by
$J^d_{\rm Serre}(Y)$. The commutative diagram
\[
\xymatrix@C1pc{
J^{d+r}_{\rm Serre}(Z) \ar[r]^{f_*} \ar@{->>}[d] & J^d_{\rm Serre}(Y) 
\ar@{->>}[d] \\
J^{d+r}(Z) \ar[r]_{f_*} & J^d(Y)}
\]
now finishes the proof. 
\end{proof}

\begin{prop}\label{prop:univ-LW-lci-group} 
The semi-abelian variety $J^d(Y)$ is the universal 
regular semi-abelian variety quotient of $\CH_0(Y)_{\deg 0}$.
\end{prop} 
\begin{proof} 
Let $S$ denote the singular locus of $Y$.
It is enough to show that 
$\rho^{\rm semi}_{Y}: \CH_0^{LW}(Y)_{\deg 0}\surj J^d(Y)$ factors through the 
canonical map $\CH_0^{LW}(Y)_{\deg 0} \surj \CH_0(Y)_{\deg 0}$. 
Let then $\nu\colon C\to Y$ be a finite l.c.i map from a good curve $C$ 
relative to $S\subset Y$ and let 
$f\in \cO^{\times}_{C,E}$ for $E = C_{\rm sing}\cup \nu^{-1}(S)$. 

As in Lemma \ref{lem:factor-PN-projection}, we factor $\nu$ as composition 
$\nu = \pi\circ \mu$, where $\mu\colon C\hookrightarrow \P^N_{Y}$ is a regular 
embedding and $\pi\colon \P^N_{Y}\to Y$ is the projection. 
By \lemref{lem:semi-ab-functor}, we have the commutative 
\nolinebreak
diagram
\[
\xymatrix{ 
\sZ_0(\P^N_{Y}, \P^N_S)_{\deg 0}  \ar[r]^{\pi_*} 
\ar[d]_{\tilde{\rho}^{\rm semi}_{\P^N_Y}} & \
\sZ_0(Y, S)_{\deg 0}  \ar[d]^{\tilde{\rho}^{\rm semi}_{Y}}\\
J^{d+N}(\P^N_{Y}) \ar[r]_{\pi_*} & J^d(Y)}
\]
        where $\tilde{\rho}^{\rm semi}$ denotes on both vertical sides of the 
diagram the composition of the universal map $\rho^{\rm semi}$ with the 
canonical map $\sZ_0(-)_{\deg 0} \to \CH_0^{LW}(-)_{\deg 0}$. Since 
$\tilde{\rho}^{\rm semi}_{\P^N_Y}(\mu_*((\rm div)_C(f))) =0$ in 
$J^{d+N}(\P^N_{Y})$, the result follows.
\end{proof}

%\end{lem}
Using \propref{prop:univ-LW-lci-group} and 
\thmref{thm:Mallick-Rojtman}, we obtain the following comparison 
between the torsion parts of the Levine-Weibel Chow group and our modified 
definition. 

\begin{thm}\label{thm:LW-lci-iso-pprimary-torsion}
Let $Y$ be an equidimensional reduced projective scheme of dimension
$d \ge 1$ over an algebraically closed field $k$ of exponential
characteristic $p$. Then
the canonical 
map $\CH_0^{LW}(Y)\{l\} \xrightarrow{\simeq} \CH_0(Y)\{l\}$ between the 
$l$-primary torsion subgroups, is an isomorphism for every prime $l \neq p$.
\end{thm}
\begin{proof}
Let $L$ denote the kernel of the canonical surjective map 
$\CH_0^{LW}(Y)\surj \CH_0(Y)$.
We first show that $L$ is a $p$-primary group of bounded exponent. 
It follows from \propref{prop:univ-LW-lci-group} that
the semi-Albanese map $\rho^{\rm semi}_{Y}$ from $\CH_0^{LW}(Y)_{\deg 0}$ to 
$J^d(Y)$ factors through $\CH_0(Y)_{\deg 0}$. 
Since $\rho^{\rm semi}_{Y}$ is an isomorphism on $n$-torsion subgroups for $n$ 
prime to $p$ by \thmref{thm:Mallick-Rojtman}, 
we immediately deduce that $L$ (which is same as 
${\rm Ker}(\CH_0^{LW}(Y)_{\deg 0}\to \CH_0(Y)_{\deg 0})$ is a 
$p$-torsion group. 
We now show that $L$ has a bounded exponent.

We first note from \lemref{lem:cycle-class-0-cycles}
that there is a factorization \[\CH^{LW}_0(Y) \xrightarrow{can} \CH_0(Y)
\xrightarrow{cyc_Y} K_0(Y)\] of the cycle class map $cyc^{LW}_Y$ from the
Levine-Weibel Chow group to $K_0(Y)$. On the other hand, it follows from
\cite[Corollary~5.4]{Levine-5} (see also \cite[Corollary~2.7]{Levine-2})
that the kernel of the map $cyc^{LW}_Y$ is a group of exponent 
$N_d:= (d-1)!$. We conclude that $N_d \cdot L = 0$.

We write $N_d = p^aq$, where $a \ge 0$ and $p \nmid q$.
We fix a cycle $\alpha \in L$. Since $L$ is a $p$-primary group,
we can write $p^n\alpha = 0$ for some $n \gg a$.
We then have an identity $xq + yp^{n-a} = 1$ for some $x, y \in \Z$.
This yields 
\[
p^a \alpha = (xq + yp^{n-a})p^a \alpha = (xqp^a + yp^n)\alpha
= xN_d \alpha + yp^n \alpha = 0.
\]
Since $a \in \Z$ depends only on $N_d$ and not on $\alpha$,
we get $p^a \cdot L = 0$. 
  
It is easy to see using 5-lemma that for every prime $l\neq p$, there is an
exact sequence
\[
0 \to L\{l\} \to \CH_0^{LW}(Y)\{l\} \to \CH_0(Y)\{l\} \to
L \otimes_{\Z} {\Q_l}/{\Z_l}.
\]
Since $L$ is a $p$-primary group of bounded exponent and
${\Q_l}/{\Z_l}$ is $p$-divisible, we must have $L\{l\} = 0 =   
L \otimes_{\Z} {\Q_l}/{\Z_l}$. In particular,
$\CH_0^{LW}(Y)\{l\} \xrightarrow{\simeq} \CH_0(Y)\{l\}$
\end{proof}

\begin{cor}\label{cor:Com-proj-p}
Let $k$ be an algebraically closed field of exponential
characteristic $p$. 
Let $Y$ be an equidimensional reduced projective scheme of dimension
$d \le p$. Then the canonical map $\CH^{LW}_0(Y) \to \CH_0(Y)$ is
an isomorphism.
\end{cor}

\subsection{Regular homomorphism from Chow group with modulus}
\label{sec:Reg-mod}
We now turn to the definition of a regular morphism for the Chow group of $0$-cycles with modulus. Let $(X,D)$ be as in  \ref{ssec:SettingDoubleRegular}.
We write $X^o$ for the open complement of $D$ in $X$.
\begin{defn}\label{def:pi-map-albanese}
We fix a closed point $x_0\in X^o$. 
For $U\subset X^o$ open with $x_0\in U$, we define the map of sets 
\[
\pi^U_{x_0}\colon U\to \CH_0(X|D)_{\deg 0}, \quad x\mapsto [x]-[x_0],
\]
where $[x]$ and $[x_0]$ denote the classes of $x$ and $x_0$ respectively in 
$\CH_0(X|D)$. For a commutative algebraic group $G$ over $k$, we say that a 
homomorphism of abelian groups \[\rho\colon \CH_0(X|D)_{\deg 0} \to G\] is 
regular if there exists an open subset $U$ of $X^o$ and a closed 
point $x_0\in U$ such that $\rho\circ \pi^U_{x_0}\colon U\to G$ is a 
morphism of algebraic varieties. 
\end{defn}

\section{The Abel-Jacobi map with modulus and its universality}
\label{sec:AJ-map}
Let $X$ be a smooth projective scheme of dimension $d \ge 1$ over $\C$
and let $D \subset X$ be an effective Cartier divisor.
In this section, we show that $A^d(X|D)$
constructed in \S~\ref{sec:Alb-mod} has a natural structure of a
connected commutative algebraic group and it is the universal regular
quotient of $\CH_0(X|D)_{\rm deg \ 0}$ via an Abel-Jacobi map.
Since we are working over $\C$, 
we keep identifying $\CH^{LW}_0(Z)$ with $\CH_0(Z)$ in this section
(using \thmref{thm:0-cycle-affine-proj-char0}) for any projective 
scheme $Z$.

\subsection{The Abel-Jacobi map}\label{sec:CC-map}
Let $(X,D)$ be as above.
We write again $X^o$ for the open complement of $D$ in $X$ and $S_X$ for the double 
construction applied to the pair $(X,D)$. Recall from 
\cite{EV-DBcohomology} (see also \cite[Lemma 2.1]{ESV}) that given 
$x \in (S_X)_{\rm reg} = X^o\amalg X^o$, there is a unique element 
$[x]\in \H^{2d}_{\{x\}}(S_X, \ZDe{S_X}{d})$ mapping to the topological cycle 
class of $x$ in $H^{2d}_{\{x\}}(S_X, \Z(d))$ as well as to the de Rham cycle 
class of $x$ in $\H^{2d}_{\{x\}}(S_X, \Omega^{\geq d}_{S_X})$. Using the canonical 
forget support map $\H^{2d}_{\{x\}}(S_X, \ZDe{S_X}{d}) \to 
\H^{2d}(S_X, \ZDe{S_X}{d})$ and extending linearly, this gives rise to a well 
defined map 
\[
cyc_{S_X}^{\cD}\colon \sZ_{0}(S_X, D) \to \H^{2d}(S_X, \ZDe{S_X}{d})
\]
that composed with $\H^{2d}(S_X, \ZDe{S_X}{d}) \to H^{2d}(S_X,\Z(d)_{S_X}) = 
\Z\oplus \Z$ coincides with the degree map. The same construction on $X$ 
gives rise to the diagram
\begin{equation}\label{eqn:AJ-sequence-0}
\xymatrix{  \sZ_{0}(S_X, D) \ar[r]^/-.5cm/{i_-^*}\ar[d]_{cyc_{S_X}^{\cD}} & 
\sZ_0(X, D) \ar[d]^{cyc_{X}^{\cD}} \subset \sZ_0(X) \\ 
\H^{2d}(S_X, \ZDe{S_X}{d}) \ar[r]^{i_-^*} & \H^{2d}(X, \ZDeun{X}{d})
}
\end{equation}
whose commutativity is easily checked. We also note that $cyc^{\cD}$  
commutes with the map \[\Delta^*\colon \sZ_0(X, D) \to \sZ_{0}(S_X, D).\] 
By \cite[Lemma 2.6]{ESV}, the cycle class map $\cyc^{\cD}_{S_X}$ factors 
through the Chow group $\CH_0(S_X)$ and therefore determines a 
homomorphism 
\[
\rho_{S_X}\colon \CH_0(S_X)_{\deg 0}\to A^{d}(S_X),
\]
where $ \CH_0(S_X)_{\deg 0}$ denotes the kernel of the degree map 
$\deg\colon \CH_0(S_X) \surj \Z\oplus \Z$, that is the generalized 
Abel-Jacobi map of \cite{ESV}. Since the cycle class map to Deligne 
cohomology anyway factors through the usual Chow group of $0$-cycles for 
smooth projective varieties (see \cite{EV-DBcohomology}), we have then the 
following commutative diagram, with split exact rows
\begin{equation}\label{eqn:AJ-sequence-*}
\xymatrix{ 0\to \CH_0(X|D)_{\deg 0} \ar[r]^-{p_{+,*}} \ar[d]_{\rho_{X|D}} & 
\CH_0(S_X)_{\deg 0}\ar[d]^{\rho_{S_X}} \ar@<.5ex>[r]^{i^*_-}  &  
\CH_0(X)_{\deg 0} \ar[d]^{\rho_X} \ar@<.5ex>[l]^{\Delta^*} \to 0\\
0\to A^d(X|D) \ar[r]^-{p_{+,*}}  & A^d(S_X)  \ar@<.5ex>[r]^-{i^*_-} &  
A^d(X)\ar@<.5ex>[l]^-{\Delta^*} \to 0,
}
\end{equation}
where $\rho_X\colon \CH_0(X)_{\deg 0}\to A^d(X)$ is the usual Abel-Jacobi map 
and \[\rho_{X|D}\colon \CH_0(X|D)_{\deg 0}  \to A^d(X|D)  \] is the induced map on 
the kernels. 
Note that thanks to the existence of the splitting $\Delta^*$ of $i_-^*$ and 
its compatibility with the degree maps, the exactness of the first row follows 
immediately from our main Theorem \ref{thm:Main-PB-PF-gen}.
We shall call $\rho_{X|D}$ the {\sl Abel-Jacobi map with modulus}.

\subsection{Regularity of $\rho_{X|D}$}\label{sec:alg-mod-*}
We now fix base points $x_{0, \pm} = x_0$ on each 
irreducible components of $(S_X)_{\rm reg}$ and consider
the diagram
%The morphism ${i^*_-}: A^d(S_X) \to A^d(X)$ in
%~\eqref{eqn:AJ-sequence-0} and ~\eqref{eqn:AJ-sequence-*}
%is induced by the maps
\begin{equation}\label{eqn:AJ-sequence-*-0}
\xymatrix@C2pc{
X^o_+ \amalg X^o_- \ar[r]^{\pi^{(S_X)_{\rm reg}}_{x_{0, \pm}}} \ar[d]_{\psi} & 
\CH_0(S_X)_{\deg 0} \ar[d]^{i^*_-} \ar@{->>}[r] & A^d(S_X) 
\ar@{-->}[d]^{\wt{i^*_-}}  \\
X^o  \ar[r]^/-.5cm/{\pi^{X^o}_{x_0}} & \CH_0(X)_{\deg 0} \ar@{->>}[r] & A^d(X),}
\end{equation}
where $\psi|_{X_-} = {\rm Id}_X$ and $\psi|_{X_+} = x_0$
so that the left square commutes.

$\psi$ is a morphism of schemes and the composite map on the bottom
is a regular morphism of schemes. It follows that 
the composite map \[X^o_+ \amalg X^o_- \to  \CH_0(S_X)_{\deg 0} \to
\CH_0(X)_{\deg 0} \to A^d(X)\] is a morphism of schemes.
It follows from the universal property of $A^d(S_X)$ that there is
a unique regular homomorphism of algebraic groups
$\wt{i^*_-}\colon  A^d(S_X) \to A^d(X)$ such that the right square  
commutes.

On the other hand, the right square in ~\eqref{eqn:AJ-sequence-*}
also commutes. Since $\CH_0(S_X)_{\deg 0} \to A^d(S_X)$ is surjective,
it follows that $\wt{i^*_-} =   i^*_-$.
We conclude that the map $i^*_-$ on the bottom row of 
~\eqref{eqn:AJ-sequence-*} is a regular homomorphism of connected
commutative algebraic groups.
Since $A^d(X|D)$ is the inverse image of the identity element under
this homomorphism, it follows that $A^d(X|D)$ is a
commutative algebraic group.
We have thus shown that $A^d(S_X)$ is an extension of
the abelian variety $A^d(X)$ by the connected commutative 
algebraic group $A^d(X|D)$.

\begin{lem}\label{lem:reg-Alb-C} 
The group homomorphism $\rho_{X|D}\colon \CH_0(X|D)_{\deg 0}\to A^d(X|D)$ is 
regular and surjective, making $A^d(X|D)$ a regular quotient of 
$\CH_0(X|D)_{\deg 0}$.
\end{lem}
\begin{proof}The surjectivity of the generalized Abel-Jacobi map 
$\rho_{X|D}$ is a consequence of the definition. 
Indeed, $\rho_{S_X}$ is surjective 
by \thmref{thm:universalityESV} while $\rho_X$ is classically known to be 
surjective. The surjectivity of  $\rho_{X|D}$ follows then from the existence 
of the splitting \[\Delta^*\colon \CH_0(X)_{\deg 0}\to \CH_0(S_X)_{\deg 0}\] that 
makes the induced map ${\rm Ker}(\rho_{S_X})\to {\rm Ker}(\rho_X)$ surjective
(see ~\eqref{eqn:AJ-sequence-*}).  

For the regularity, let $V$ be an open dense subset of $(S_X)_{\rm reg}$ such 
that $\rho_{S_X}\circ\pi^V_{x_0}$ is regular. By  \thmref{thm:universalityESV}, 
such $V$ exists. Up to shrinking $V$ further, we can assume that $V$ is of the 
form $U\amalg U$, for $U\subset X^o$ open (dense) subset of $X$ disjoint from 
$D$. Let $i_{U,+}$ denote the inclusion $U\to U\amalg U$ of the first 
component. Then we clearly have \[ \rho_{S_X} \circ p_{+,*} \circ \pi^U_{x_0} = 
\rho_{S_X} \circ \pi^{U\amalg U}_{x_0, \pm} \circ i_{U,+}\] so that the composition 
$U\to \CH_0(X|D)_{\deg 0 } \to A^d(X|D) \inj A^d(S_X)$ 
is a morphism of schemes. Since $A^d(X|D) \inj A^d(S_X)$ is a  
closed immersion, we get the claim.
\end{proof}

\subsection{Universality of $\rho_{X|D}$}\label{sec:Univ}
Our next goal is to prove that the Abel-Jacobi map with modulus
$\rho_{X|D}\colon \CH_0(X|D)_{\deg 0}\to A^d(X|D)$ makes $A^d(X|D)$ the
universal regular quotient of $\CH_0(X|D)_{\deg 0}$.

\begin{lem}\label{lem:Delta*-mor}
The homomorphism $\Delta^*\colon  A^d(X) \to A^d(S_X)$ of ~\eqref{eqn:AJ-sequence-*}
is a morphism of schemes.
\end{lem}
\begin{proof}
We have seen in the proof of \lemref{lem:reg-Alb-C} that there is a
dense open subset $U \subset X^o$ and a closed point $x_0 = x_{0, \pm} \in U$
such that the composition $U \amalg U \xrightarrow{\pi^{U\amalg U}_{x_0, \pm}} 
\CH_0(S_X)_{\deg 0 } \to A^d(S_X)$ is a morphism of schemes.
Let $i_{U,\pm}\colon  U \inj U \amalg U$ denote the inclusions into the
first and the second component, respectively.
Again from the proof of  \lemref{lem:reg-Alb-C}, we have that the maps
$\rho_{S_X} \circ \pi^{U\amalg U}_{x_0, \pm} \circ i_{U,\pm}$ are both
morphisms of schemes.
In particular, the composition
\begin{equation}\label{eqn:addition}
\theta_U\colon  U \xrightarrow{\psi}  A^d(S_X) \times  A^d(S_X) \xrightarrow{+}  
A^d(S_X)
\end{equation}
is also a morphism of schemes, where 
$\psi = (\rho_{S_X} \circ \pi^{U\amalg U}_{x_0, \pm} \circ i_{U,+},
\rho_{S_X} \circ \pi^{U\amalg U}_{x_0, \pm} \circ i_{U, -})$
and the second arrow in ~\eqref{eqn:addition} is the addition.

We now consider a diagram
\begin{equation}\label{eqn:Delta*-mor-0}
\xymatrix@C1pc{
U \ar[r]^/-.4cm/{\pi^U_{x_0}} \ar@{=}[d] & \CH_0(X)_{\deg 0} \ar[r]^/.2cm/{\rho_X} 
\ar[d]^{\Delta^*} & A^d(X) \ar@{-->}[d] \\ 
U \ar[r] & \CH_0(S_X)_{\deg 0} \ar[r]_-{\rho_{S_X}} & A^d(S_X),}
\end{equation}
where the first arrow on the bottom is 
$(\pi^{U\amalg U}_{x_0, \pm} \circ i_{U,+}) + 
(\pi^{U\amalg U}_{x_0, \pm} \circ i_{U,-})$.
It is clear from the definition of $\Delta^*$ in the middle that the
left square of ~\eqref{eqn:Delta*-mor-0} commutes. The composite map on the 
bottom is same as $\theta_U$, which we just showed above to be a
morphism of schemes. We conclude that the map
$\rho_{S_X} \circ \Delta^*$ is a regular homomorphism.
It follows from the universality of $A^d(X)$ that there is a unique
morphism of algebraic groups $\wt{\Delta^*}\colon  A^d(X) \to A^d(S_X)$ such that
the right square of ~\eqref{eqn:Delta*-mor-0} commutes.

On the other hand, we have seen in ~\eqref{eqn:AJ-sequence-*} that
right square also commutes if we replace $\wt{\Delta^*}$ by
$\Delta^*$. Since $\rho_X$ is surjective, we must have 
$\wt{\Delta^*} = \Delta^*$. In particular, $\Delta^*$ is 
morphism of schemes. 
\end{proof}

\subsubsection{}To prove the universality of $A^d(X|D)$, we begin with the 
following construction. Consider the homomorphism
\[ 
\tau^*\colon A^{d}(S_X)\to A^d(S_X), \quad 
\tau^* = {\rm Id}_{S_X} - \Delta^*\circ i_-^*.
\]
Note that $\Delta^*$ is a morphism of schemes by \lemref{lem:Delta*-mor}
and we have shown in \S~\ref{sec:alg-mod-*} that $i^*_-$ is also 
a morphism of schemes. We conclude that $\tau^*$ is morphism of 
algebraic groups.

Note that $\tau^*$ uniquely factors through $A^d(X|D)$, since 
$i_-^* \circ \tau^* =0$ and we have already identified $A^d(X|D)$ with the 
fiber  over the identity element of $A^d(X)$ via $i_-^*$. The map $\tau^*$ 
gives then an explicit isomorphism of algebraic groups 
\begin{equation}\label{eqn:Split-decom}
(\tau^*, i_-^*)\colon A^d(S_X) \xrightarrow{\simeq} A^d(X|D) \times A^d(X).
\end{equation}
Moreover, since $i^*_* \circ \Delta^* = {\rm Id}_{A^d(X)}$, we  get an 
extension of commutative algebraic groups
\begin{equation}\label{eqn:Split-decom-0}
0 \to A^d(X) \xrightarrow{\Delta^*} A^d(S_X) \xrightarrow{\tau^*}
A^d(X|D) \to 0.
\end{equation}

We now claim that the diagram
\begin{equation}\label{eqn:Split-decom-1}
\xymatrix@C1pc{
\CH_0(S_X)_{\deg 0} \ar[r]^{\tau^*_X} \ar[d]_{\rho_{S_X}} & \CH_0(X|D)_{\deg 0}
\ar[d]^{\rho_{X|D}} \\
A^d(S_X) \ar[r]^{\tau^*} & A^d(X|D)}
\end{equation}
commutes.
To see this, we can can write, using \thmref{thm:Main-PB-PF-gen},
any element $\alpha \in \CH_0(S_X)_{\deg 0}$ as $\alpha = \
p_{+,*}(\alpha_1) + \Delta^*(\alpha_2)$.
Since $\tau^*_X \circ \Delta^* = 0$, by definitions, we get
$\rho_{X|D}\circ \tau^*_X(\alpha) = \rho_{X|D}(\alpha_1)$
by ~\eqref{eqn:AJ-sequence-*}.

On the other hand, $\tau^* \circ \Delta^* = 0$ and $\tau^* \circ 
p_{+,*} = {\rm Id}_{A^d(X|D)}$ so that
\[
\begin{array}{lll}
\tau^* \circ \rho_{S_X}(\alpha) & = & 
\tau^* \circ \rho_{S_X} \circ p_{+,*}(\alpha_1)  +
\tau^* \circ \rho_{S_X} \circ \Delta^*(\alpha_2) \\
& {=}^{\dagger} & \tau^* \circ p_{+,*} \circ \rho_{X|D} (\alpha_1) +
\tau^* \circ \Delta^* \circ \rho_X(\alpha_2) \\
& = & \rho_{X|D} (\alpha_1),
\end{array}
\]
where ${=}^{\dagger}$ follows from ~\eqref{eqn:AJ-sequence-*}.
This proves the commutativity of ~\eqref{eqn:Split-decom-1}.

\begin{thm}\label{thm:univ-Alb-overC} 
The Abel-Jacobi map $\rho_{X|D} \colon \CH_0(X|D)_{\deg 0} \to A^d(X|D)$
makes $A^d(X|D)$ the universal regular quotient of $\CH_0(X|D)_{\deg 0}$. 
\end{thm}
\begin{proof}
We only need to prove the universality.
Let $G$ be a commutative 
algebraic group and let $\psi\colon \CH_0(X|D)_{\deg 0} \to G$ be a 
regular homomorphism. Let $U\subset X^o$ be an open dense subset so that 
the 
composite $\psi \circ \pi^U_{x_0}$ is a morphism of schemes, for $x_0\in U$ a 
base point. Let $V = U\amalg U$. We claim that 
$\delta = \psi\circ \tau_X^*\circ \pi^{V}_{x_0, \pm} $ is a morphism of 
schemes, where 
$\tau_X^*\colon \CH_0(S_X)_{\deg 0} \to \CH_0(X|D)_{\deg 0}$ is splitting  the 
map $p_{+,*}$ of Theorem \ref{thm:Main-PB-PF-gen}. 
Indeed, it is actually enough to 
show that the restriction of $\delta$ to the two components $U_\pm$ of $V$ is 
a morphism of schemes. But we have, for $x\in U_+$:
\[
\psi\circ \tau_X^*\circ \pi_{x_0, \pm}^{V} (x)= \psi \circ 
\tau_X^* ( [x]_+ - [x_{0, +}])  = \psi ( [x]-[x_0]) = \psi \circ \pi^U_{x_0} (x),
\]
where $[x]_+$ denotes the class in $\CH_0(S_X)$ of the closed point $x$ in the 
component $X_+$ of $S_X$. Since $\psi \circ \pi^U_{x_0}$ is by assumption a 
morphism of schemes, this proves the claim for the restriction to $U_+$. 
Similarly for $x\in U_-$, we have
\[
\psi\circ \tau_X^*\circ \pi_{x_0, \pm}^{V} (x)= \psi \circ 
\tau_X^* ( [x]_- - [x_{0, -}])  = \psi ( - [x] + [x_0]) =  
- \psi \circ \pi^U_{x_0} (x)
\]
that is also a morphism of schemes, since $G$ is an algebraic group. 

By the claim and \thmref{thm:universalityESV}, there is then a unique morphism 
of algebraic groups $\tilde{\psi} \colon A^d(S_X) \to G$ such that there is
a commutative square
\[\xymatrix{ \CH_0(S_X)_{\deg 0} \ar@{->>}[r]^{\tau_X^*} \ar[d]_{\rho_{S_X}} & 
\CH_0(X|D)_{\deg 0} \ar[d]^{\psi} \\ 
A^d(S_X) \ar[r]_{\tilde{\psi}} & G.}
\]

We now claim that the composition
\[ 
A^d(X) \xrightarrow{\Delta^*} A^d(S_X) \xrightarrow{\tilde{\psi}} G\]
is equal to the constant map $A^d(X)\to 0$.  
Indeed, since the left square of the diagram
\[
\xymatrix{ \CH_0(X)_{\deg 0 } \ar[r]^{\Delta^*} \ar@{->>}[d]_{\rho_X} &
\CH_0(S_X)_{\deg 0} \ar[r]^{\tau_X^*} \ar[d]_{\rho_{S_X}} & 
\CH_0(X|D)_{\deg 0} \ar[d]^{\psi} \\ 
 A^d(X) \ar[r]_{\Delta^*} & A^d(S_X) \ar[r]_{\tilde{\psi}} & G}
\]
also commutes by ~\eqref{eqn:AJ-sequence-*} and since $\rho_X$ is surjective, 
it is enough to show that 
$\tilde{\psi} \circ \Delta^* \circ \rho_X =0$.

But $\tilde{\psi} \circ \Delta^* \circ \rho_X = 
\tilde{\psi} \circ \rho_{S_X} \circ \Delta^* = 
\psi \circ( \tau_X^* \circ \Delta_X^* )= 0$ since 
$\tau^*_X\circ \Delta_X^* = 0$. This proves the claim.
Using this face, the exact sequence ~\eqref{eqn:Split-decom-0}
and the commutative diagram ~\eqref{eqn:Split-decom-1},
it follows immediately that there exists a unique morphism of algebraic groups
$\psi_G\colon  A^d(X|D) \to G$ such that $\psi_G \circ \rho_{X|D} = \psi$.
This finishes the proof.
\end{proof}

\begin{comment}
Since $A^d(X)$ is projective and 
$\Delta^*$ is a section of the morphism $i^*_-$, it follows that 
$\Delta^*$ is a closed immersion of algebraic groups. 
Hence the morphism $\tilde{\psi}$ factors uniquely through the quotient 
$A^d(S_X) / \Delta^*(A^d(X))$.

Finally, we have
\[
\begin{array}{lll}
\theta \circ \rho_{X|D} & = & \theta \circ \rho_{X|D} \circ \tau^*_X \circ
p_{+, *} \\
& = & \phi \circ p_{+, *} \circ \rho_{X|D} \circ \tau^*_X \circ
p_{+, *} \\
& = & \phi \circ \rho_{S_X} \circ p_{+, *} \circ \tau^*_X \circ
p_{+, *} \\
& = & \wt{\phi} \circ \tau^*_X \circ p_{+, *} \circ \tau^*_X \circ
p_{+, *} \\
& = & \wt{\phi} \circ \tau^*_X \circ
p_{+, *} \\
& = & \wt{\phi}.
\end{array}
\]
In particular, $(\ov{\psi} \circ \theta) \circ \rho_{X|D} =
\ov{\psi} \circ \wt{\phi} = \psi$.
\end{comment}

\subsection{Roitman's  theorem for 0-cycles with modulus}\label{sec:RT-0}
The first application of our approach to study  algebraic cycles with 
modulus was already given in \thmref{thm:Intro-3-Pf}.
As second application, we now prove the following Roitman torsion theorem
for 0-cycles with modulus on smooth projective schemes over $\C$.
This will be generalized to positive characteristic in the
next section.

\begin{thm}\label{thm:Roitman-char-0}
Let $X$ be a smooth projective variety of dimension $d \ge 1$ over
$\C$ and let $D\subset X$ be an effective Cartier divisor. 
Then the Abel-Jacobi map $\rho_{X|D}\colon \CH_0(X|D)_{\deg 0} \to A^d(X|D)$ 
induces an isomorphism on the torsion subgroups.
\end{thm}
\begin{proof}
We have the following commutative diagram with split exact rows
\begin{equation}\label{eqn:Roitman-char-0*}
\xymatrix{ 
0\to \CH_0(X|D)_{\deg 0 } \ar[r]^-{p_{+,*}} \ar[d]_{\rho_{X|D}} &
\CH_0(S_X)_{\deg 0} \ar[r]^-{i_-^*} \ar[d]_{\rho_{S_X}} & \CH_0(X)_{\deg 0} 
\ar[d]^{\rho_X}  \to 0\\ 
0\to A^d(X|D) \ar[r]_-{p_{+,*}} & A^d(S_X) \ar[r]_-{i_-^*}  & A^d(X)\to 0.}
\end{equation}

Since the maps $i_-^*$ on the top and the bottom rows
are split by $\Delta^*$, the two 
sequences remain exact even after passing to the torsion subgroups. The 
statement then follows from the theorem of Roitman \cite{Roitman}
for $\rho_X$ and the theorem of Biswas-Srinivas \cite{BS-1} for
$\rho_{S_X}$.
\end{proof}

\subsection{Bloch's conjecture for 0-cycles with modulus}
\label{sec:B-conj}
Let $X$ be a reduced projective surface over $\C$. 
Recall that the Chow group of 0-cycles $\CH_0(X)$ is said to be
{\sl finite dimensional} if the Abel-Jacobi map
$\rho_X\colon  \CH_0(X)_{\deg 0} \to A^2(X)$ is an isomorphism.
Recall that the famous Bloch conjecture about 0-cycles on surfaces
says the following.

\begin{conj}$($Bloch$)$\label{conj:Bloch}
Let $X$ be a smooth projective surface over $\C$ such that $H^2(X, \sO_X) = 0$.
Then $\CH_0(X)$ is finite dimensional. 
\end{conj}

This conjecture is known to be true for surfaces of Kodaira dimension
less than two \cite{BKL} and is open in general. It has been shown by 
Voisin \cite{Voisin} that a generalized version of this conjecture in 
higher dimensions is very closely related to the Hodge conjecture.

Let $X$ be a smooth projective surface over $\C$ and let
$D \subset X$ be an effective Cartier divisor. Let $\sI_D$ denote the
subsheaf of ideals in $\sO_X$ defining $D$.
We shall say that $\CH_0(X|D)$ is finite dimensional if the
map $\rho_{X|D}\colon  \CH_0(X|D)_{\deg 0} \to A^d(X|D)$ is an isomorphism.
We can now state the following analogue of the Bloch conjecture 
for 0-cycles with modulus.

\begin{conj}\label{conj:Bloch-mod}
Let $X$ be a smooth projective surface over $\C$.
Let $D \subset X$ be an effective Cartier divisor such that
$H^2(X, \sI_D) = 0$. Then $\CH_0(X|D)$ is finite dimensional.
\end{conj}

\begin{remk}As explained in \cite[2.1.2]{BS}, the Chow groups with modulus can be used to define a notion of Chow groups with compact support for the complement $X^o = X\setminus |D|$. In this perspective, we can view Conjecture \ref{conj:Bloch-mod} as an analogue of Bloch's conjecture for the open surface $X^o$.
\end{remk}
As an application of \thmref{thm:Main-PB-PF-gen}, we can show the following.

\begin{thm}\label{thm:Bloch-mod-thm}
Let $X$ be a smooth projective surface over $\C$.
Let $D \subset X$ be an effective Cartier divisor such that
$H^2(X, \sI_D) = 0$. Assume that Conjecture~\ref{conj:Bloch} holds for $X$.
Then $\CH_0(X|D)$ is finite dimensional.
\end{thm}
\begin{proof}
We first observe that $\pi\colon  \ov{S_X}:= X \amalg X \to S_X$ is the 
normalization map and hence the Bloch conjecture for $X$ is 
same as that for $\ov{S_X}$. 
Since $H^2(X, \sI_D) \surj H^2(X, \sO_X)$, it follows that $H^2(X, \sO_X) = 0$.
In particular, $\CH_0(X)$ and $\CH_0(\ov{S_X})$ are finite dimensional.

Since the map $\sI_{X_-} \to \sI_D$ is an isomorphism
(see ~\eqref{eq:eq1-proof-Albanese-iso-0}), there is  an
exact sequence
\[
H^2(X, \sI_D) \to H^2(S_X, \sO_{S_X}) \to H^2(X, \sO_X) \to 0.
\]
We conclude that $H^2(S_X, \sO_{S_X}) = 0$.
We now apply \cite[Theorem~1.3]{Krishna-1} to conclude that
$\CH_0(S_X)$ is finite dimensional. We remark here that the statement
of this cited result assumes validity of Conjecture~\ref{conj:Bloch} for all
smooth surfaces, but its proof (see \cite[\S~7]{Krishna-1}) only uses
the fact that the normalization of the surface (which is already smooth
in our case) is finite-dimensional. We now use ~\eqref{eqn:Roitman-char-0*}
to conclude that $\CH_0(X|D)$ is finite dimensional.
\end{proof}

A combination of \thmref{thm:Bloch-mod-thm} and \cite{BKL} yields the
following.

\begin{cor}\label{cor:Bloch-conj-mod-NGT}
Let $X$ be a smooth projective surface over $\C$ of Kodaira dimension
less than two. Let $D \subset X$ be an effective Cartier divisor such that
$H^2(X, \sI_D) = 0$. Then $\CH_0(X|D)$ is finite dimensional.
\end{cor}

{\sl Infinite-dimensionality of $\CH_0(X|D)$:}
The following result provides examples of smooth projective surfaces
$X$ with an effective Cartier divisor $D \subset X$ such that
$\CH_0(X)$ is finite-dimensional but $\CH_0(X|D)$ is not.
In particular, it provides a partial converse to
\thmref{thm:Bloch-mod-thm}.

\begin{thm}\label{thm:Bloc-converse}
Let $X$ be a smooth projective surface over $\C$.
Let $D \subset X$ be an effective Cartier divisor such that
$H^2(X, \sI_D) \neq 0$. Assume that $X$ is regular with $p_g(X) = 0$
and that the Bloch conjecture is true for $X$. Then
$\CH_0(X|D)$ is not finite-dimensional.
\end{thm}
\begin{proof}
The exact sequence
\[
H^1(X,\sO_X) \to H^2(X, \sI_D) \to H^2(S_X, \sO_{S_X}) \to
H^2(X,\sO_X) \to 0
\]
and our assumption together imply that $H^2(S_X, \sO_{S_X}) \neq 0$.
We claim that the map \[\rho_{S_X}\colon \CH_0(S_X)_{\rm deg \ 0}
\to A^2(S_X)\]
is not injective. Suppose on the contrary, that $\rho_{S_X}$ is injective.
But then it must be an isomorphism. It follows then from
\cite[Theorem~7.2]{ESV} (see its proof on pg. 657) that there are
finitely many reduced Cartier curves $\{C_1, \cdots , C_r\}$ on $X$ such that
the map $\stackrel{r}{\underset{i = 1}\oplus} \CH_0(C_i)_{\rm deg \ 0} \to 
\CH_0(S_X)_{\rm deg \ 0}$ is surjective.
However, as $H^2(S_X, \sO_{S_X}) \neq 0$, \cite[Theorem~5.2]{Srinivas-2}
tells us that this is not possible. This proves the claim.
Our assumption says that the map $\rho_X$ in ~\eqref{eqn:Roitman-char-0*}
is an isomorphism. It follows that ${\rm Ker}(\rho_{X|D}) \neq 0$.
\end{proof}

If we let $D \subset \P^2_{\C}$ be a smooth hypersurface of degree 3, we have
\[H^2(\P^2_{\C}, \sI_D) \simeq H^2(\P^2_{\C}, \sO_{\P^2_{\C}}(-D))
\simeq H^2(\P^2_{\C}, \sO_{\P^2_{\C}}(-3)) \simeq \C.\]
All conditions of \thmref{thm:Bloc-converse} are clearly satisfied
and we get an example of a (smooth) pair $(X,D)$ such that 
$\CH_0(X)_{\rm deg \ 0} = 0$ but $\CH_0(X|D)$ is infinite-dimensional.

\section{Albanese with modulus in arbitrary characteristic}
\label{sec:Alb+ve}
In this section, we generalize the results of \S~\ref{sec:AJ-map}
for 0-cycles with modulus on smooth projective varieties
over an arbitrary algebraically closed field.
Most of the arguments are straightforward copies of those
in \S~\ref{sec:AJ-map}. So we keep the discussion brief.
We fix an algebraically closed field $k$ of exponential characteristic
$p \ge 1$.  

%{\color{blue}
Let again $Y$ be a projective reduced $k$-variety of dimension $d$ and write 
${\rm Alb}(Y)$ for the Esnault-Srinivas-Viehweg Albanese variety.
While there is an explicit description of ${\rm Alb}(Y) = A^{d}(Y)$ over 
$\C$ (as recalled in \ref{ssec:Construction-Albanese-C}) using Hodge theory, 
in positive characteristic, \cite{ESV} gives only an existence statement and 
little is known on the properties of ${\rm Alb}(Y)$ 
(some pathological properties of ${\rm Alb}(Y)$ for specific singular 
varieties are studied in \cite[\S~3]{ESV}).

In \cite{Mallick}, V. Mallick proves the following Roitman-style theorem.
        
\begin{thm}$($\cite[Theorem 16]{Mallick}$)$\label{thm:Mallick-Rojtman}  
For any reduced projective variety $Y$ of dimension $d$ over $k$
and for $n$ coprime with the characteristic of $k$, the map $\rho_Y$ 
induces an isomorphism on $n$-torsion subgroups 
\[
\rho_Y\colon  {}_{n} \CH^{LW}_0(Y)_{\deg 0} \xrightarrow{\simeq} 
{}_n{\rm Alb}(Y) = {}_nJ^d(Y).
\]
\end{thm}
For the rest of this section, the Albanese variety ${\rm Alb}(Y)$ will
be denoted by $A^d(Y)$ to keep consistency with the notations
of the previous sections.

\subsection{Albanese with modulus and its universality 
in any characteristic}\label{ssec:Construction-Albanese-anyk} 
Let $X$ be a smooth projective scheme of dimension $d \ge 1$ over $k$ and 
let $D\subset X$ be an effective Cartier divisor. Write as usual $S_X$ for the 
double construction applied to the pair $(X,D)$.  
Using \thmref{thm:Main-Comparison-Chow}, we shall again
identify the two Chow groups $\CH^{LW}_0(S_X)$ and $\CH_0(S_X)$ throughout
this section.

Let $U$ be an open dense subset contained in $X^o = X\setminus D$, 
$x_0\in U$ a closed point and 
$V = U\amalg U \subset (S_X)_{\rm reg}$ such that the compositions 
\[ 
\pi^V_{x_0, \pm} \colon V\to \CH_0(S_X)_{\deg 0}\xrightarrow{\rho_{S_X}} 
A^d(S_X), \quad \  \pi^U_{x_0}\colon U\to \CH_0(X)_{\deg 0}\xrightarrow{\rho_X} 
A^d(X) \]
are morphisms of schemes
(see Definition \ref{def:pi-map-albanese} for the notation 
$\pi^V_{x_0, \pm}$ and $\pi^U_{x_0}$).

Theorem \ref{thm:Main-PB-PF-gen} gives the familiar split 
short exact sequence on the group of zero cycles
\begin{equation}\label{eq:Main-exact-sequence-sec-Alb} 
0\to \CH_0(X|D)_{\deg 0} \xrightarrow{p_{+,*}} \CH_0(S_X)_{\deg 0} 
\xrightarrow{\iota_-^*} \CH_0(X)_{\deg 0} \to 0
\end{equation}
and there is a homomorphism $\Delta^*\colon  \CH_0(X)_{\deg 0}  \to 
\CH_0(S_X)_{\deg 0}$ such that $i^*_- \circ \Delta^* = {\rm Id}$.  

It follows from the discussion in \S~\ref{sec:alg-mod-*} and
the proof of \lemref{lem:Delta*-mor} that the homomorphisms
$\rho_X \circ i^*_-\colon  \CH_0(S_X)_{\deg 0} \to A^d(X)$ and
$\rho_{S_X} \circ \Delta^*\colon  \CH_0(X)_{\deg 0} \to A^d(S_X)$ are both 
regular. It follows from the universality of $A^d(X)$ and $A^d(S_X)$
that there are unique homomorphisms of algebraic groups
$i^*_-\colon  A^d(S_X) \to A^d(X)$ and $ \Delta^*\colon  A^d(X) \to A^d(S_X)$
such that the diagram

\begin{equation}\label{eqn:AJ-sequence-**}
\xymatrix{ 0\to \CH_0(X|D)_{\deg 0} \ar[r]^-{p_{+,*}}  & 
\CH_0(S_X)_{\deg 0}\ar@{->>}[d]^{\rho_{S_X}} \ar@<.5ex>[r]^{i^*_-}  &  
\CH_0(X)_{\deg 0} \ar@{->>}[d]^{\rho_X} \ar@<.5ex>[l]^{\Delta^*} \to 0\\
& A^d(S_X)  \ar@<.5ex>[r]^-{i^*_-} &  
A^d(X)\ar@<.5ex>[l]^-{\Delta^*} \to 0
}
\end{equation}
is commutative.

It also follows from this commutative diagram that
$i^*_- \circ \Delta^*\colon  A^d(X) \to A^d(X)$ is the identity.
Indeed, any $\alpha \in A^d(X)$ can be written as
$\alpha = \rho_X(\beta)$ and then
\begin{equation}\label{eqn:Split-Alb+}
i^*_- \circ \Delta^* \circ \rho_X(\beta) =
i^*_- \circ \rho_{S_X} \circ \Delta^*(\beta)  = 
\rho_X \circ i^*_- \circ \Delta^*(\beta) 
= \rho_X(\beta).
\end{equation}

\begin{defn}\label{defn:Rel-Alb+}
We define $A^d(X|D)$ to be the closed algebraic subgroup of $A^d(S_X)$ given 
by the inverse image of the identity element of $A^d(X)$ via $\iota_-^*$. 
As before, we denote the inclusion $A^d(X|D) \inj A^d(S_X)$ by $p_{+,*}$.
It follows from ~\eqref{eqn:AJ-sequence-**} that there is a unique 
surjective homomorphism $\rho_{X|D}\colon  \CH_0(X|D)_{\deg 0} 
\surj A^d(X|D)$
such that $\rho_{S_X} \circ p_{+,*} = p_{+,*} \circ \rho_{X|D}$.
\end{defn}

The reader can now check easily that using ~\eqref{eqn:AJ-sequence-**}
and ~\eqref{eqn:Split-Alb+}, the proofs of
\lemref{lem:reg-Alb-C} and \thmref{thm:univ-Alb-overC}
go through mutatis mutandis to give the following generalization of
the latter.

\begin{thm}\label{thm:univ-Alb-overk}
The group homomorphism $\rho_{X|D}\colon\CH_0(X|D)_{\deg 0}\to A^d(X|D)$ is 
regular and surjective and makes $A^d(X|D)$ the universal regular quotient of 
$\CH_0(X|D)_{\deg 0}$. 
\end{thm}

%\begin{cor}\label{cor:Surface-+ve}
%Let $X$ be a smooth projective surface over an algebraically closed field $k$ of arbitrary characteristic. Then $A^d(X|D)$ is the universal regular quotient of $\CH_0(X|D)_{\deg 0}$.
%\end{cor}

\begin{remk}\label{remk:Unipotent-part} 
Since $X$ is smooth and projective, the Albanese variety $A^d(X)$ is an 
abelian variety over $k$. By \cite{ESV}, the generalized Albanese 
$A^d(S_X)$ is a smooth commutative algebraic group of general type, i.e., 
an extension of an abelian variety by a linear algebraic group. 
An immediate consequence of the definition and of 
\thmref{thm:univ-Alb-overk} is that the linear part of $A^d(S_X)$ 
coincides with the linear part of the Albanese with modulus $A^d(X|D)$. 
\end{remk}

\subsection{Roitman theorem in arbitrary characteristic}
\label{sec:Roit+ve}
Using \thmref{thm:Mallick-Rojtman}, we can now 
generalize \thmref{thm:Roitman-char-0}
over any algebraically closed field as follows.
We keep the notations and the assumptions of 
\ref{ssec:Construction-Albanese-anyk}.

\begin{thm}\label{thm:Roitman-Main+ve}
Let $X$ be a smooth projective variety of dimension $d \ge 1$ over
an algebraically closed field $k$ and let $D\subset X$ be an effective 
Cartier divisor. 
Let $n$ be an integer prime to the exponential characteristic of $k$. 
Then the Albanese map $\rho_{X|D}\colon \CH_0(X|D)_{\deg 0}\to A^d(X|D)$ 
induces an isomorphism on $n$-torsion subgroups
\[
\rho_{X|D}\colon  
{}_n\CH_0(X|D)_{\deg 0} \xrightarrow{\simeq} {}_n A^d(X|D). 
\]
\end{thm}
\begin{proof}
We consider the commutative diagram
\[
\xymatrix{ 0\to {}_{n}\CH_0(X|D)_{\deg 0 } \ar[r]^-{p_{+,*}} 
\ar[d]_{\rho_{X|D}} &{}_{n}\CH_0(S_X)_{\deg 0} \ar[r]^{i_-^*} 
\ar[d]_{\rho_{S_X}}^{\simeq} & {}_{n}\CH_0(X)_{\deg 0} \ar[d]^{\rho_X}  \to 0\\ 
0\to {}_{n}A^d(X|D) \ar[r]^-{p_{+,*}}  & {}_{n}A^d(S_X) \ar[r]^{i_-^*} 
& {}_{n}A^d(X) \to 0}
\]

It follows from ~\eqref{eqn:AJ-sequence-**} that the top  and the 
bottom rows are compatibly split exact. 
The original Roitman's theorem \cite{Roitman} says  
that the vertical arrow on the left is an isomorphism.
It follows from Theorems~~\ref{thm:Main-Comparison-Chow}
and \ref{thm:Mallick-Rojtman} 
that the middle vertical arrow is an isomorphism. The theorem follows.
\end{proof}

\begin{remk} We note that the linear part of $A^d(X|D)$ 
(that coincides with the linear part of $A^d(S_X)$) depends heavily on the 
geometry of $D$ and will have, in general, both a unipotent and a 
torus part. For example, if $D$ is a smooth divisor inside a smooth surface, 
the presence of a $\G_m$ part in $A^d(X|D)$ depends on the class of $D$ 
in the Neron-Severi 
group of $X$.
\end{remk}

\section{Cycle class map to relative $K$-theory}\label{sec:Rel-K}
As we mentioned in \S~\ref{sec:Intro}, one of the motivations for studying
cycles with modulus is to find a cohomology theory which can describe 
relative $K$-theory of divisors in a scheme in terms of algebraic cycles.
If the higher Chow groups with modulus are indeed the right objects which
serve this purpose, there must be a cycle class map from
the Chow groups with modulus to relative $K$-groups. Moreover, this map must
describe the Chow group as a part of relative $K$-groups in most of the cases.
Our goal in this section is to use our double construction to answer
these questions for the 0-cycles with modulus. 
We first consider the case of line bundles in this setting.

\subsection{Vector  bundles on the double and relative Picard groups}
\label{sec:Pic}
Let $k$ be any field and let $X$ be a 
smooth quasi-projective scheme over $k$ with an effective Cartier 
divisor $D$. Let $X^o$ be the open complement $X\setminus D$. 
We denote as above by $S_X$ the double $S(X,D)$.

\subsubsection{Vector bundles on the double}\label{sec:Vec-D}
Let $\mathcal{P}_S = \mathcal{P}_{S(X,D)}$ denote the category of locally free 
sheaves of finite rank on $S_X$. Since $X$ is quasi-projective, 
$S_X$ is quasi-projective as well and therefore it admits an ample family of 
line bundles. Thus we can replace  Thomason spectrum $K(S_X)$ built out of 
perfect complexes on $S_X$ with $\Omega BQ \mathcal{P}_{S}$, at least for 
computing the groups $K_p(S_X)$ for $p\geq 0$, and similarly for $X$ and $D$. 
By construction, the category $\mathcal{P}_S$ is equivalent to the category of 
triples $(E, E^\prime, \phi)$, where $E$ and $E^\prime$ are locally free sheaves 
on $X$ and $\phi\colon \iota_D^* E\to \iota_D^* E^\prime$ is a fixed 
isomorphism on the restriction to $D$ (see \cite{Levine-HigherChow}, 
but also \cite[Theorem 2.1]{Milnor}).
This description gives us the following.
\begin{lem}\label{lem:composition-KSX-toKD-trivial} 
The composite map of spectra
\begin{equation}\label{eq:rel-K-theory}
K(S_X)\xrightarrow{\iota_1^*, \iota_2^*} K(X)\amalg K(X) \xrightarrow{f} K(D)
\end{equation}
is homotopy trivial, where 
\[
f = \iota_D^* \oplus - \iota_D^* \colon 
K(X)\amalg K(X) 
\xrightarrow{(id, -id)} K(X)\amalg K(X) \xrightarrow{\iota_D^* + \iota_D^*} 
K(D).
\]
\end{lem}
Notice that the maps in the definition of $f$ make sense because they
are defined in the homotopy category of spectra which is an additive category.

\subsubsection{Relative Picard group}\label{sec:Rel-P}
We denote by $\Pic(X,D)$  the group of isomorphism classes of pairs 
$(\sL,\sigma)$ consisting of a line bundle $\sL$ on $X$ together with a fixed 
trivialization $\sigma\colon \sL_{|D} \xrightarrow{\simeq} \cO_D$ along $D$, 
under  tensor product operation. It is called the relative Picard group 
of the pair $(X,D)$. 
%As remarked in  \ref{Sec:RelKDouble}, the category of 
%vector bundles on the double $S_X$ is equivalent to the category of triples 
%$(E,E', \phi)$, where $E$ and $E'$ are vector bundles on $X$ and $\phi$ is an 
%isomorphism on the restrictions to $D$. 
Write $G$ for the group of 
isomorphism classes of triples $\{(\sL_+, \sL_-, \phi)\}$ for $\cL_\pm$ line 
bundles on $X$ with a given isomorphism $\phi\colon  \sL_+|_D \xrightarrow{\simeq}
\sL_-|_D$ along $D$. The above 
description of the category of vector bundles on $S_X$
gives in particular two maps, one inverse to the other
\[
\theta\colon \Pic(S_X)\to  G, \quad \eta\colon G\to \Pic(S_X)
\]
defined by 
\[
\theta(\cL) = (\cL_+ = \iota_+^*(\cL), \cL_- = \iota_-^*(\cL), 
\phi\colon \iota_D^*\iota_+^*\cL \simeq \iota_D^*\iota_-^*\cL ),\quad 
\eta((\cL_+, \cL_-, \phi)) = \cL_+\times_\phi \cL_-
\]
where $\cL_+\times_\phi \cL_-$ is the gluing of $\cL_+$ and $\cL_-$ along 
$\phi$.

We will then identify the group $\Pic(S_X)$ with  $G$. In this way we can easily define maps
\[ 
p_{\pm,*}\colon \Pic(X,D) \rightrightarrows \Pic(S_X), \quad 
\tau_X^*\colon \Pic(S_X)\to \Pic(X,D) \]
using formally the same definitions that we gave for $0$-cycles in 
\S~\ref{section:SSM}. Explicitly, we have
\[
\tau_X^*((\cL_+, \cL_-, \phi)) = 
(\sL_+\tensor \cL_-^{-1}, \phi\tensor \text{id}_{\iota_D^* \cL_-^{-1}})
\]
for $\phi\tensor \text{id}_{\iota_D^* \sL_-^{-1}}\colon   
\iota_D^*(\cL_+\tensor \cL_-^{-1}) =  
\iota_D^* \cL_+ \tensor_{\cO_D} \iota_D^* \cL_-^{-1} 
\xrightarrow{\phi\tensor 1} \iota_D^*(\cL_-)\tensor  
\iota_D^*(\cL_-)^{-1} \xrightarrow{can} \cO_D,$ and 
$p_{+,*}(\cL, \sigma) = \cL \times_\sigma \cO_{X} = (\cL, \cO_{X}, \sigma)$ 
(and similarly for $p_{-,*}$).

It is immediate to check that $p_{\pm, *}$ are injective, splitting 
$\tau_X^*$. Moreover, we have maps
\[
\iota_\pm^*\colon \Pic(S_X) \rightrightarrows \Pic(X),\quad 
\Delta_X^*\colon \Pic(X) \to \Pic(S_X)
\]
given on isomorphism classes by 
$\iota_+^*((\cL_+, \cL_-, \phi)) = 
\cL_+,$ $\iota_-^*((\cL_+, \cL_-, \phi)) = \cL_-$ and 
$\Delta_X^*(\cL) =  (\cL, \cL, \text{id})$. And one clearly has that the 
composition $ \iota_\pm^* \circ \Delta_X^*$ is the identity. 

We summarize the result in the following Proposition, that is the analogue of 
Theorem \ref{thm:Main-PB-PF-gen} for line bundles and is used in 
\S~\ref{sec:Alb-C}.

\begin{prop}\label{prop:Main-Sequence-Pic} 
Let $X$ be a smooth quasi-projective scheme over $k$ and let 
$D\subset X$ be an effective Cartier divisor. Then there are maps
\[ 
\Delta^*\colon \Pic(X)\to \Pic(S_X);  \ \ \ \mbox{and} \ \ \ 
\iota_\pm^*\colon \Pic(S_X) \to \Pic(X);
\]
\[ 
\tau^*_X\colon \Pic(X)\to \Pic(X,D);  \ \ \ \mbox{and} \ \ \ 
p_{\pm,*}\colon \Pic(X,D) \to \Pic(S_X)
\]
such that $\iota_\pm^*\circ \Delta^* ={\rm Id}$ on $\Pic(X)$ and 
$\tau_X^*\circ p_{\pm, *} = \pm {\rm Id}$ on $\Pic(X,D)$. 
Moreover, the sequences
\[
0\to \Pic(X,D) \xrightarrow{p_{+,*}} \Pic(S_X) \xrightarrow{\iota_-^*} \Pic(X) 
\to 0;
\]
\[
0\to \Pic(X) \xrightarrow{\Delta^*} \Pic(S_X) \xrightarrow{\tau_X^*} 
\Pic(X,D) \to 0
\]
are split exact.
\end{prop}

\subsection{Cycle class map for 0-cycles with modulus}\label{sec:CMap-M}
The goal of this subsection is the proof of \thmref{thm:Intro-4}.
In order to do define the cycle class map from the Chow group with modulus to 
relative $K$-group and prove its
injectivity, we need to have an analogue of \thmref{thm:Main-PB-PF-gen}
for the relative $K$-theory. Our strategy for proving this is to first
construct a variant of relative $K$-theory, which we call 
the {\sl $K$-theory with modulus}, for which  \thmref{thm:Main-PB-PF-gen}
is immediate. We then show that this $K$-theory with modulus
coincides with the known relative $K$-theory in as many cases as possible.

Recall that for any map of schemes $f\colon X \to Y$, the relative $K$-theory of the pair $(X,Y)$ 
is the spectrum defined as the homotopy fiber
of the map $f^*\colon K(Y) \to K(X)$.
Let $X$ be a smooth quasi-projective scheme over $k$ and let 
$D\subset X$ be an effective Cartier divisor.
Let $K(X|D)$ denote the homotopy fiber of the restriction map
$i^*_{-}: K(S_X) \to K(X)$. It is clear that $K(X|D)$ is another notation
for the relative $K$-theory $K(S_X, X_-)$. 
We call $K(X|D)$ the $K$-theory with modulus.
We have $i^*_- \circ \Delta^* = {\rm Id}_{K(X)}$ and 
a commutative diagram of homotopy fiber sequences
\begin{equation}\label{eqn:K-map}\
\xymatrix@C1pc{
K(X|D) \ar[r]^{p_{+,*}} \ar[d]_{\phi} & K(S_X) \ar[r]^{i^*_-} \ar[d]^{i^*_+} &
K(X) \ar[d]^{\iota^*_-} \\
K(X,D) \ar[r] & K(X) \ar[r]^{\iota^*_+} & K(D).}
\end{equation} 

We have the following analogue of \propref{prop:surj-alb-any-dimension}
for affine schemes.
\begin{prop}$($\cite[Theorem 6.2, Lemma 4.1]{Milnor}$)$
\label{prop:Milnor-Affine-K0} Let $X=\Spec(A)$ be an affine scheme and let 
$I$ be the ideal defining $D\subset X$. 
Then the map $\phi$ defines isomorphisms
\[
\phi_i\colon K_i(X|D)\xrightarrow{\sim} K_i(X,D) \quad \text{ for } i = 0,1.
\]
\end{prop} 

We do not know if \propref{prop:Milnor-Affine-K0} is true for non-affine
schemes. But we can show that this is indeed the case for curves and surfaces
when $i = 0$.
The case of curves follows directly from \propref{prop:Main-Sequence-Pic}.
We shall prove this for surfaces in \propref{prop:Surface-mod-rel}.

\vskip .3cm
 
We can now prove the main result of this section:

\begin{thm}\label{thm:Intro-4-A}
Let $X$ be a smooth quasi-projective scheme of dimension $d \ge 1$ over 
a perfect field $k$ and let $D \subset X$ be an effective Cartier 
divisor.
Then, there is a cycle class map 
\[
cyc_{X|D}: \CH_0(X|D) \to K_0(X,D).
\]

This map is injective if $k$ is algebraically closed and $X$ is affine.
\end{thm}
\begin{proof}
We have a commutative diagram of short exact sequences
\begin{equation}\label{eqn:Intro-4-A-0}
\xymatrix@C1pc{
0\to \CH_0(X|D) \ar[r]^-{p_{+,*}} 
\ar@{-->}[d] & \CH_0(S_X) \ar[r]^-{i_-^*} 
\ar[d]_{cyc_{S_X}} & \CH_0(X) \ar[d]^{cyc_X}  \to 0 \\ 
0\to K_0(X|D) \ar[r]^-{p_{+,*}}  \ar[d]_{\phi_0} & K_0(S_X) \ar[r]^-{i_-^*} 
\ar[d]^{i^*_+} & K_0(X)  \ar[d]^{\iota^*_-} \to 0 \\
K_0(X,D) \ar[r] & K_0(X) \ar[r] & K_0(D).}
\end{equation}

%{\color{blue} Where the map $cyc_X$ is classically known to be injective and the map $cyc_{S_X}$ is injective by Levine \cite{Levine-5} and \cite{Levine-2}}. 

The maps $cyc_X$  and $cyc_{S_X}$ are the cycle class maps
of \lemref{lem:cycle-class-0-cycles} (where $Y= (S_X)_{\rm sing}$ for $cyc_{S_X}$ and $Y = \emptyset$ for $cyc_{X}$. See also \cite[\S~2]{LW}). 
The above diagram uniquely defines a cycle class map
map $\wt{cyc}_{X|D}: \CH_0(X|D) 
\to K_0(X|D)$ such that $p_{+,*} \circ \wt{cyc}_{X|D} = cyc_{S_X} \circ
p_{+,*}$. We set $cyc_{X|D} = \phi_0 \circ \wt{cyc}_{X|D}$.

If $k$ is algebraically closed and $X$ is affine, the map $cyc_{S_X}$ is injective by 
\cite[Corollary~7.3]{Krishna-2}.
%Levine proves injectivity only in char 0.
It follows that $\wt{cyc}_{X|D}$ is injective. We conclude proof
of the injectivity of $cyc_{X|D}$ using \propref{prop:Milnor-Affine-K0}.
\end{proof}

\begin{remk}There are (at least) two general constructions of a cycle class 
map \[cyc_{X|D}^{d+n,n}\colon \CH^{d+n}(X|D, n) \to K_n(X,D)\] from higher Chow groups with modulus in the sense of \cite{BS} to higher relative $K$-groups (here $d=\dim X$). See \cite[I.4]{BThesis} for one of them. We will study the properties of this map in a different work.
\end{remk}

\section{The case of surfaces}\label{sec:BF-surfaces}
In this section, we shall apply \thmref{thm:Main-PB-PF-gen} to establish
the relation between cycles with modulus and relative $K$-theory for
surfaces. In particular, we prove a modulus version of Bloch's formula.
Before we do this, we need to prove a generalization of 
\propref{prop:Milnor-Affine-K0} for non-affine surfaces.

\begin{lem}\label{lem:Levine-modified}
Let $X$ be a reduced quasi-projective surface over any field $k$ containing
at least three elements. Let $F^2K_0(X)$ be the subgroup of $K_0(X)$ generated
by the cycle classes of regular points of $X$. Then, there is a canonical
isomorphism $H^2(X, \sK_{2,X}) \xrightarrow{\simeq} F^2K_0(X)$.
\end{lem}
\begin{proof}
By the Thomason-Trobaugh spectral sequence for $K$-theory, we get an
exact sequence 
\[
K_1(X) \to H^0(X, \sK_{1,X}) \xrightarrow{\partial} H^2(X, \sK_{2,X}) \to K_0(X).
\]
On the other hand, it is well known that the map $K_1(X) \to H^0(X, \sK_{1,X})$
is surjective (see, for example, \cite[\S~2]{Krishna-3}).
It follows that $H^2(X, \sK_2) \inj K_0(X)$. 

We are only left to show that $H^2(X, \sK_{2,X}) \surj F^2K_0(X)$.
But this follows by the results of Levine \cite{Levine-1}, because 
he shows that there is a surjective
map $\sZ_0(X, X_{\rm sing}) = \amalg_{x \in X_{\rm reg}} \Z \surj H^2(X, \sK_{2,X})$
such that the composite $\sZ_0(X, X_{\rm sing}) \to H^2(X, \sK_{2,X}) \to K_0(X)$
is the cycle class map. We only need to remark here
that at the outset of \cite{Levine-1}, Levine assumes the ground field 
to be infinite. But the the surjectivity of the map
$\sZ_0(X, X_{\rm sing}) \surj H^2(X, \sK_{2,X})$ does not require this assumption.
The assumption on the cardinality of $k$ is required to use Matsumoto's
presentation of Quillen $K_2$ of a field.
\end{proof}

\begin{prop}\label{prop:Surface-mod-rel}
Let $X$ be a smooth quasi-projective surface over a field $k$ containing at 
least three elements. 
Let $D \subset X$ be an effective Cartier divisor. 
Then the canonical map $\phi_0\colon K_0(X|D) \to K_0(X,D)$ is an isomorphism.
\end{prop}
\begin{proof}
We have  convergent spectral sequences
\begin{equation}\label{eq:spseq-K-modulus}  
H^p(S_X,\cK_{q,{X|D}}) \Rightarrow \pi_{-p-q}K(S_X, X_-) = \pi_{-p-q}K(X|D);
\end{equation}
\begin{equation}\label{eq:spseq-rel-K} 
H^p(X,\cK_{q,{(X,D)}}) \Rightarrow \pi_{-p-q}K(X,D)
\end{equation}
induced by the Thomason-Trobaugh spectral sequence for the $K$-theory of 
$S_X$, $X$ and $D$. Here, $\cK_{q,{X|D}}$ is the Zariski sheaf associated 
to $U\mapsto K_q(U|D \cap U)$ and $\cK_{q, {(X,D)}}$ denotes the Zariski sheaf 
associated to $U\mapsto K_q(U, U\cap D)$. We denote by $F^iK_0(X,D)$ the 
filtration on $K_0(X,D)$ induced by \eqref{eq:spseq-K-modulus}. 

Since $\cK_{0, X|D} = {\rm Ker} (\cK_{0, S_X} =\Z \to \cK_{0,X} = \Z) = 0$, 
we have
\[
0\to F^1K_0(X|D) \to K_0(X|D) \to H^0(S_X, \cK_{0, X|D}) = 0
\]
and therefore $F^1K_0(X|D) \xrightarrow{\simeq} K_0(X|D)$. 
Next, we have the exact sequence
\[
0\to F^2 K_0(X|D) \to F^1 K_0(X|D) = K_0(X|D) \to H^1(S_X, \cK_{1, X|D})\to 
0 
\]
and a canonical map $H^2(S_X, \cK_{2, X|D}) \to F^2 K_0(X|D)$. 
Since $X$ is a smooth surface, the corresponding map 
$H^2(X, \cK_{2, X}) \to F^2 K_0(X)$ is an isomorphism and the same holds for the 
map $H^2(S_X, \cK_{2, S_X}) \to F^2 K_0(S_X)$ by \lemref{lem:Levine-modified}. 
By construction, the sequence of sheaves 
\begin{equation}\label{eq:split-seq-K-sheaves}
0\to\cK_{i, X|D} \to \cK_{i, S_X} \to \cK_{i,X}\to 0
\end{equation}
is split exact and the space $K(X|D)$ is a retract of $K(S_X)$, 
so that by taking cohomology and comparing with the $F^2$-piece of the 
filtration on the $\pi_0$ groups, we obtain that 
$H^2(S_X, \cK_{2, X|D}) \to F^2K_0(X|D)$ is indeed an isomorphism. 
In particular, we have a short exact sequence
\begin{equation}\label{eqn:K0-filter}
0\to H^2(S_X, \cK_{2, X|D}) \to K_0(X|D) \to H^1(S_X, \cK_{1, X|D})\to 0.
\end{equation}

Consider now the second spectral sequence \eqref{eq:spseq-rel-K}. 
Since $\cK_{0, (X,D)} = 0$ as well, we have another exact sequence
\[
0\to H^2(X, \cK_{2, (X,D)}) \to K_0(X,D) \to H^1(X, \cK_{1, (X,D)})\to 0\]
where the first map is injective by \cite[Lemma~2.1]{Krishna-3}.
The natural map $\phi\colon K(X|D)\to K(X,D)$ induces then a commutative 
diagram, with exact rows
\[
\xymatrix{ 
0\to H^2(S_X, \cK_{2, X|D}) \ar[r]\ar[d] & K_0(X|D) \ar[r]\ar[d] &
 H^1(S_X, \cK_{1, X|D}) \ar[d] \to 0 \\ 
 0\to H^2(X, \cK_{2, (X,D)}) \ar[r] & K_0(X,D) \ar[r] & H^1(X, \cK_{1, (X,D)})
\to 0.
 }
\]

As $H^1(S_X, \cK_1) \simeq \Pic(S_X)$ and $H^1(X, \cK_1)\simeq \Pic(X)$, 
applying cohomology to \eqref{eq:split-seq-K-sheaves} with 
$i=1$ gives 

\[ 
0\to H^1(S_X, \cK_{1, X|D}) \to \Pic(S_X) \xrightarrow{\iota_-^*} 
\Pic(X)\to 0 
\]
and by \propref{prop:Main-Sequence-Pic}, we have an identification 
$H^1(S_X, \cK_{1, X|D}) = \Pic(X,D)$. Similarly, we have  
$H^1(X, \cK_{1, (X,D)}) = \Pic(X,D)$ by \cite[Lemma~2.1]{SV}
and hence the right vertical map is an isomorphism. 
To finish the proof of the proposition, we are now left with proving
the following Lemma.
\end{proof}

\begin{lem}\label{lem:map-K2-sheaves-onto}
The map $H^2(S_X, \cK_{2, X|D}) \to H^2(X, \cK_{2, (X,D)})$ is an isomorphism.
\end{lem}
\begin{proof}
Given an open subset $W \subset D$, let $U = S_X \setminus (D \setminus W)$
be the open subset of $S_X$. Let $\cK_{i, (S_X, X_-, D)}$ be the sheaf on $D$ 
associated to the presheaf $W \mapsto 
K_i(U, X_+ \cap U, X_- \cap U) =
{\rm hofib}((K(U, X_- \cap U) \xrightarrow{i^*_+} K(X_+ \cap U, D \cap U))$ 
(see \cite[Proposition~A.5]{PW}).
There is an exact sequence of $K$-theory sheaves
\[
\iota_{*}(\cK_{2, (S_X, X_-, D )}) \to \sK_{2, (S_X, X_-)} \to
\sK_{2, (X_+, D)} \to \iota_{*}(\cK_{1, (S_X, X_-, D)}),
\]
where $\iota:D \inj S_X$ is the inclusion.
We have $\cK_{1, (S_X, X_-, D)} = {\sI_D}/{\sI^2_D} \otimes_D \Omega^1_{D/X}$
by \cite[Theorem~1.1]{GW} and the latter term is zero.
Since $\sK_{2, (S_X, X_-)}  = \cK_{2, X|D}$ by definition, 
we  get then an exact sequence
\[
\iota_{*}(\cK_{2, (S_X, X_-, D)} )\to  \cK_{2, X|D} \to
\sK_{2, (X_+, D)} \to 0.
\]
Since $H^2(S_X, \iota_{*}(\cK_{2, (S_X, X_-, D)}) = 
H^2(D, \cK_{2, (S_X, X_-, D)}) = 0$,
the lemma follows.
\end{proof}

We now prove our main result on cycles with modulus on surfaces.

\begin{thm}\label{thm:Intro-4-B}
Let $X$ be a smooth quasi-projective surface over an algebraically closed field 
$k$ and let $D \subset X$ be an effective Cartier 
divisor. 
Then the following hold.
\begin{enumerate}
\item
The cycle class map $cyc_{X|D}: \CH_0(X|D) \to K_0(X,D)$ 
induces a short exact sequence
%\begin{equation}\label{eqn:Intro-4-B-0}
\[
0 \to \CH_0(X|D) \to K_0(X,D) \to \Pic(X,D) \to 0
\]
and the image of $\CH_0(X|D)$ agrees with $F^2 K_0(X|D)$.
%\end{equation}
\item
There are isomorphisms 
\[
\CH_0(X|D) \xrightarrow{\simeq} H^2_{\rm zar}(X, \sK^M_{2,(X,D)}) 
\xrightarrow{\simeq} H^2_{\rm nis}(X, \sK^M_{2,(X,D)}).
\] 
\end{enumerate}
\end{thm}
\begin{proof}
The injectivity of the cycle class map 
follows exactly like
\thmref{thm:Intro-4-A} by using \propref{prop:Surface-mod-rel}
instead of \propref{prop:Milnor-Affine-K0}. 
Next, ~\eqref{eqn:K0-filter} and 
\propref{prop:Surface-mod-rel} show that 
the exactness of (1) is equivalent to
showing that $\CH_0(X|D) \simeq F^2K_0(X|D)$. But this follows again
by observing that $cyc_{S_X}$ and $cyc_X$ become isomorphisms 
(using \cite[Main Theorem]{Levine-1}) if
we replace the middle row of ~\eqref{eqn:Intro-4-A-0} by $F^2K_0(-)$
which keeps the row exact. 

To prove (2), recall that $\sK^M_{2,{(X,D)}}$ is, by definition,
the kernel of the map of Milnor $K$-theory sheaves 
$\sK^M_{2,X} \surj \sK^M_{2,D}$.
If we let $\sK^M_{2, X|D} = {\rm Ker}(\Delta_*(\sK^M_{2, S_X}) \to \sK^M_{2,X})$,
then the top square of ~\eqref{eqn:Intro-4-A-0} has
factorization
\[
\xymatrix@C1pc{
0\to \CH_0(X|D) \ar[r]^-{p_{+,*}} 
\ar@{-->}[d] & \CH_0(S_X) \ar[r]^-{i_-^*} 
\ar[d]_{cyc_{S_X}} & \CH_0(X) \ar[d]^{cyc_X}  \to 0 \\ 
0\to H^2(S_X, \sK^M_{2, X|D}) \ar[r]^-{p_{+,*}} \ar[d] &
H^2(S_X, \sK^M_{2}) \ar[r]^-{i_-^*} \ar[d] &
H^2(X, \sK^M_{2}) \ar[d] \to 0 \\
0\to K_0(X|D) \ar[r]_-{p_{+,*}} & K_0(S_X) \ar[r]_-{i_-^*} 
& K_0(X) \to 0.}
\]

The commutative square
\[
\xymatrix@C1pc{
\CH_0(X|D) \ar[r] & H^2(S_X, \sK^M_{2, X|D}) \ar[r] \ar[d]_{i^*_+} &
K_0(X|D) \ar[d]^{\phi_0} \\
& H^2(X, \sK^M_{2, (X,D)}) \ar[r] & K_0(X,D)}
\]
now shows that there is a factorization
$\CH_0(X|D) \to H^2(X, \sK^M_{2, (X,D)}) \to K_0(X,D)$ of
the map $cyc_{X|D}$. It follows from (1) that the
first map is injective. On the other hand, it follows from
\cite[Theorem~1.2]{Krishna-3} that this map is surjective.
We conclude that the map
$\CH_0(X|D) \to H^2_{\rm zar}(X, \sK^M_{2, (X,D)})$ is an isomorphism.
Furthermore, \cite[Lemma~2.1]{Krishna-3} implies that
\[H^2_{\rm zar}(X, \sK^M_{2, (X,D)}) \xrightarrow{\simeq}
H^2_{\rm nis}(X, \sK^M_{2, (X,D)})\]
as required.
\end{proof}

\section{0-cycles with modulus on affine 
schemes}\label{sec:Affine**}
In \thmref{thm:Intro-4-A}, we gave our first application of 
\thmref{thm:Main-PB-PF-gen} to 0-cycles with
modulus on affine schemes.
In this section, we deduce more applications of \thmref{thm:Main-PB-PF-gen}
for such schemes.

\subsection{Affine Roitman torsion for 0-cycles with modulus}
\label{sec:Aff-torsion}
For affine schemes, our second application is the following
Roitman torsion theorem for 0-cycles with modulus.

\begin{thm}\label{thm:Intro-3-Pf}
Let $X$ be a smooth affine scheme of dimension $d \ge 2$
over an algebraically closed field $k$
and let $D \subset X$ be an effective Cartier divisor.
Then $\CH_0(X|D)$ is torsion-free.
\end{thm}
\begin{proof}
The proof is immediate from \thmref{thm:Main-PB-PF-gen} and
\cite[Theorem~1.1]{Krishna-2}, using the comparison given by 
\thmref{thm:Main-Comparison-Chow} (or 
Theorem \ref{thm:0-cycle-affine-proj-char0}).
\end{proof}
\begin{remk}An independent proof of the vanishing of the prime-to-$p$ torsion part of $\CH_0(X|D)$ for affine varieties (where $p$ denotes the exponential characteristic of $k$) can be found in \cite{Btor}. The argument in \textit{loc.cit.} does not rely on our decomposition \thmref{thm:Main-PB-PF-gen}, but follows instead closely the approach of Levine in \cite{Levine-2}.
	\end{remk}
\subsection{Vanishing theorems}\label{sec:Trivial**}
As another application of \thmref{thm:Main-PB-PF-gen}, we get the
following vanishing theorems.

\begin{cor}\label{cor:finite}
Let $X$ be a smooth affine scheme of dimension $d \ge 2$ over $\ov{\F_p}$
and let $D \subset X$ be an effective Cartier divisor.
Then $\CH_0(X|D) =0$. 
\end{cor}
\begin{proof}
Using Theorems~\ref{thm:0-cycle-affine-proj-char0} and 
~\ref{thm:Main-PB-PF-gen}, it is enough to know that $\CH_0(S_X) = 0$.
But this follows form \cite[Theorem~6.4.1]{KSri1}.
\end{proof}

The same argument, using \cite[Theorem~6.4.2]{KSri1}, shows the following.

\begin{cor}\label{cor:algebraic}
Let $X = \Spec(A)$ be a smooth affine algebra of dimension $d \ge 2$
over $\overline{\Q}$. 
Assume that $A = \bigoplus_{n\geq 0} A_n$ is a graded algebra with $A_0 = 
\ov{\Q}$.
Assume moreover that $D\subset X$ is a divisor on $X$ defined by a 
homogeneous element of $A$. Then $\CH_0(X|D)=0$.
\end{cor}

\subsection{Decomposition of $K_0(X,D)$}\label{sec:Dec-aff*}
Let $X$ be a smooth quasi-projective scheme over a field $k$ and let 
$D \subset X$ be an effective Cartier divisor.
Let $\sZ^1(X|D)$ denote the free abelian group on integral closed
subschemes $Z \subset X$ of codimension one such that $D \cap Z = \emptyset$.
Let $\sR^1(X|D) = {\underset{U}\varinjlim}\ {\rm Ker}(\sO^{\times}(U) \to
\sO^{\times}(D))$, where
$U$ ranges over open subsets of $X$ containing $D$.
Note that if $D$ has an affine open neighbourhood in $X$ then
this limit is same as the limit taken over $X \setminus Z$, where
$Z \subset X$ is a divisor disjoint from $D$ and is principal in an  
affine neighbourhood of $D$.

Recall from \cite[\S~3]{BS} that the Chow group of codimension one cycles 
with modulus $\CH^1(X|D)$ is the cokernel of the map
$\sR^1(X|D) \xrightarrow{\divf} \sZ^1(X|D)$.
We often write $\CH^1(X|D)$ as $\CH_{d-1}(X|D)$ if $\dim(X) = d$. 
If $X$ is a smooth affine surface, we
can refine \thmref{thm:Intro-4-B} to completely describe
$K_0(X,D)$ in terms of the Chow groups with modulus.
In order to do this, we need the following elementary result from
commutative algebra. For any commutative noetherian ring $A$ and $a \in A$,
let $M_a$ denote the localization $M[a^{-1}]$ if $M$ is an $A$-module.

\begin{lem}\label{lem:comm-alg}
Let $A$ be a commutative noetherian ring and let $I \subset A$ be an ideal.
Let $P$ be a projective $A$-module of rank one and let
$\phi: {P}/{IP} \xrightarrow{\simeq} {A}/{I}$ be a given ${A}/{I}$-linear
isomorphism. Then we can find an element $a \in A$ such that $a \equiv 1$
mod $I$ and an isomorphism $\wt{\phi}: P_a \xrightarrow{\simeq} A_a$ and a
commutative diagram
\[
\xymatrix@C1pc{
P_a \ar[r]^{\wt{\phi}} \ar@{->>}[d] & A_a \ar@{->>}[d] \\
{P}/{IP}  \ar[r]^{\phi} & {A}/{I}.}
\]
\end{lem}
\begin{proof}
As $P$ is projective, we do have an $A$-linear map
$\wt{\phi}: P \to A$ and a commutative diagram
\[
\xymatrix@C1pc{
P \ar[r]^{\wt{\phi}} \ar@{->>}[d] & A \ar@{->>}[d] \\
{P}/{IP}  \ar[r]^{\phi}_{\simeq} & {A}/{I}.}
\]

Letting $M = {\rm Coker}(\wt{\phi})$, our assumption says that
$IM = M$. But this implies by the Nakayama lemma that 
there exists $a \in A$ such that $a \equiv 1$ mod $I$ and $M_a = 0$.
It follows that $\wt{\phi}: P_a \to A_a$ is surjective and hence an
isomorphism.
\end{proof}

\begin{thm}\label{thm:Intro-4-B-aff}
Let $X$ be a smooth affine surface over an algebraically closed field $k$ and 
let $D \subset X$ be an effective Cartier divisor.
Then, there is a short exact sequence
\begin{equation}\label{eqn:Intro-4-B-0*}
0 \to \CH_0(X|D) \to K_0(X,D) \to \CH_1(X|D) \to 0.
\end{equation}
\end{thm}
\begin{proof}
In view of \thmref{thm:Intro-4-B}, we only need to show that 
$\CH_1(X|D) \simeq \Pic(X,D)$.
Since $H^1(X, \sK_{1, (X,D)}) \simeq \Pic(X,D)$ as we have seen above,
we need to show that \[\CH_1(X|D) = \CH^1(X|D) \simeq H^1(X, \sK_{1, (X,D)}).\]

For a closed subset $Z \subset X$ of dimension one
with $Z \cap D = \emptyset$, there is an exact sequence
\begin{equation}\label{eqn:Intro-4-B-1} 
H^0(X\setminus Z, \sK_{1,(X\setminus Z,D)})
\to  H^1_Z(X, \sK_{1,(X,D)}) \to H^1(X, \sK_{1,(X,D)}) \to
H^1(X\setminus Z, \sK_{1,(X\setminus Z,D)}).
\end{equation}

The excision theorem says that
$H^1_Z(X, \sK_{1,(X,D)}) = H^1_Z(X\setminus D, \sK_{1,(X,D)})
= H^1_Z(X\setminus D, \sK_{1, X\setminus D})$ and it
follows easily from the Gersten resolution of $\sK_{1, X\setminus D}$
and the Thomason-Trobaugh 
spectral sequence $H^p_Z(X\setminus D, \sK_{q, X\setminus D})
\Rightarrow K_{q-p}(Z)$ that $H^1_Z(X, \sK_{1, X\setminus D})$ is the free abelian
group on the irreducible components of $Z$. 
Using the isomorphism $H^1(X\setminus Z, \sK_{1,(X\setminus Z,D)})
\simeq \Pic(X\setminus Z, D)$, \lemref{lem:comm-alg} tells us
precisely that ${\underset{Z}\varinjlim} \
H^1(X\setminus Z, \sK_{1,(X\setminus Z,D)}) = 0$.

Taking the colimit ~\eqref{eqn:Intro-4-B-1} over all closed subschemes $Z$ as 
above, we thus get an exact sequence
\begin{equation}\label{eqn:Intro-4-B-2} 
\sR^1(X|D) \to \sZ^1(X|D) \to
H^1(X, \sK_{1,(X,D)}) \to 0.
\end{equation}
It follows by a direct comparison of this exact sequence with
the similar sequence for $\sK_{1, X}$ that the arrow on the left
is just the divisor map. We conclude that there is a canonical
isomorphism $\CH^1(X|D) \xrightarrow{\simeq} H^1(X, \sK_{1,(X,D)})$. 
\end{proof}

\subsection{Nil-invariance of 0-cycles with modulus}
As another application of \thmref{thm:Main-PB-PF-gen},
we can prove the following result showing that the Chow group of 0-cycles
with modulus on a smooth affine surface depends only on the support of the
underlying Cartier divisor.

\begin{thm}\label{thm:divisor-support}
Let $X$ be a smooth affine surface over a perfect field $k$
and let $D \subset X$ be an effective Cartier divisor. 
Then the canonical map
\[
\CH_0(X|D) \to \CH_0(X|D_{\rm red})
\]
is an isomorphism.
\end{thm}
\begin{proof}
Using \thmref{thm:Intro-4-B}, it suffices to show that 
$F^2K_0(X|D) \to F^2K_0(X|D_{\rm red})$ is an isomorphism.

Write $X= \Spec(A)$, $D' = D_{\rm red} = \Spec(A/I)$.
Write $S_n = S(X, nD') = \Spec(R_n)$ for the double construction 
applied to the pair $(A,I^n)$. Then we have a chain of inclusions of rings
\[
\ldots \subset R_{n+1} \subset R_n \subset \ldots \subset R_1 \subset R_0 = 
A\times A
\]
where $R_n = \{ (a,b)\in R_0 \,|\, a-b \in I^n\}$. 
This gives a corresponding sequence of maps of schemes
\[ 
S_0 = X\amalg X\to S_1 \to S_2\to\ldots S_n\to \ldots
\xrightarrow{\Delta} X. 
\]

By  \cite[Theorem~3.3]{BPW} and the fact that $SK_1(B) = SK_1(B_{\rm red})$
for any commutative ring $B$ (see \cite[Chap~IX, Propositions~1.3, 1.9]{Bass}), we have for every $n\geq 1$ an exact sequence
\[\xymatrix{ 
0\to \frac{SK_1(D') \oplus SK_1(D')}{SK_1(D')\oplus SK_1(X\amalg X)}  
\ar[r] & F^2K_0(S_{n}) \ar[r]  & F^2K_0(S_0)   \to 0. 
}\] 
(Apply  \cite[Theorem~3.3]{BPW} to $X = S_n$, $\tilde{X} = S_0$ and $Y = D_n = \Spec(A/I^n)$. Note that $F^2K_0(S_n) = SK_0(S_n)$ and $F^2K_0(S_0) = SK_0(S_0)$ in the notations of \cite{BPW}). The natural maps $\rho_n\colon K_0(S_n)\to K_0(S_1)$ give then a commutative diagram with 
exact rows (for $n\geq 1$)

\begin{equation} \label{eq:Sn-affine-surf-indep-n}
\xymatrix{ 
0\to \frac{SK_1(D') \oplus SK_1(D')}{SK_1(D')\oplus SK_1(X\amalg X)} \ar@{=}[d] 
\ar[r] & F^2K_0(S_{n}) \ar[r]\ar[d]^{\rho_n} & F^2K_0(S_0) \ar@{=}[d] \to 0 \\
0\to \frac{SK_1(D') \oplus SK_1(D')}{SK_1(D')\oplus SK_1(X\amalg X)}  \ar[r] & 
F^2K_0(S_1) \ar[r] & F^2K_0(S_0)\to 0.
}
\end{equation}
It follows that $F^2K_0(S_n) \simeq F^2K_0(S_1)$.
Combining this with \propref{prop:Milnor-Affine-K0} and
the commutative diagram of exact sequences
\[ 
\xymatrix{ 0\to F^2K_0(X, nD') \ar[r]\ar[d]& F^2K_0(S_{n}) 
\ar[d]^{\rho_n}\ar[r]^{i_-^*} & F^2K_0(X)\ar@{=}[d]\to 0\\
0\to F^2K_0(X, D') \ar[r]& F^2K_0(S_1) \ar[r]_{i_-^*} & F^2K_0(X)\to 0,   
 } 
\]
we get $F^2K_0(X,  nD') \xrightarrow{\simeq} F^2K_0(X, D')$ for every
$n \ge 1$. Finally, we have inclusions $D' \subset D \subset nD'$ for
$n \gg 1$ and hence a sequence of maps $F^2K_0(X,  nD') \to F^2K_0(X,  D) \to
F^2K_0(X, D')$. It is easy to see using the isomorphism
$\CH_0(X|nD') \simeq F^2K_0(X,nD')$ that two maps in this
sequence are surjective. Since the composite map is injective,
the theorem follows.  
\end{proof}

\bigskip

\noindent\emph{Acknowledgments.} The first-named author wishes to heartily thank Marc Levine for constant support and encouragement. He is also very grateful to Shuji Saito for many friendly and inspiring conversations on these topics. Finally, he would like to thank the Tata Institute of Fundamental Research for its hospitality during three visits in September 2014, September 2015 and November 2016 where part of this paper was written.
The second-named author would like to thank Marc Levine for invitation and 
support to visit the Universit\"at Duisburg-Essen at Essen in 2015 spring, where 
a significant part of this work was carried out.

The authors would like to thank Marc Levine for going through an earlier 
version of this manuscript, providing valuable suggestions and for his help in fixing a gap in a previous version of the paper.
The authors would also like to thank H\'el\`ene Esnault, Moritz Kerz and
Takao Yamazaki for their interest in this work and
for being generous with their advice on improving the results.
Finally, the authors would like to warmly thank Shuji Saito for pointing out an important incongruence in a previous version of this manuscript and for providing much advice  on it.

We thank the referee(s) for helpful suggestions and a careful reading.

%\enlargethispage{25pt}

\bibliography{bibChowDouble}
\bibliographystyle{siam}

\end{document}